\documentclass[sn-mathphys]{sn-jnl}

\usepackage{amsmath}
\usepackage{amssymb}
\usepackage{nicefrac}
\usepackage{comment}
\usepackage{graphicx}
\usepackage{diagbox}

\newcommand{\atantwo}{\mathrm{atan2}}   

\jyear{2023}%

\theoremstyle{thmstyleone}%
\newtheorem{theorem}{Theorem}[section]
\newtheorem{proposition}[theorem]{Proposition}%
\newtheorem{corollary}{Corollary}[theorem]
\newtheorem{lemma}[theorem]{Lemma}

\theoremstyle{thmstyletwo}%
\newtheorem{example}{Example}%
\newtheorem{remark}{Remark}[section]%

\theoremstyle{thmstylethree}%
\newtheorem{definition}{Definition}[section]%

\providecommand{\abs}[1]{\lvert#1\rvert}

\usepackage{xr}
\makeatletter
\newcommand*{\addFileDependency}[1]{
  \typeout{(#1)}
  \@addtofilelist{#1}
  \IfFileExists{#1}{}{\typeout{No file #1.}}
}
\makeatother

\newcommand*{\myexternaldocument}[1]{%
    \externaldocument{#1}%
    \addFileDependency{#1.tex}%
    \addFileDependency{#1.aux}%
}

\myexternaldocument{supplementary}

\begin{document}

\title[Polyarc bounded complex interval arithmetic]
      {Polyarc bounded complex interval arithmetic}


\author*[1]{\fnm{Gábor} \sur{Geréb} }\email{gaborge@uio.no}
\author[2,3]{\fnm{András} \sur{Sándor}}\email{sandora@renyi.hu}

\affil*[1]{\orgdiv{Department of Informatics}, \orgname{University of Oslo}}
\affil[2]{\orgname{Central European University}}
\affil[3]{\orgname{HUN-REN Alfréd Rényi Institute of Mathematics}}


\abstract{
Complex interval arithmetic is a powerful tool for the analysis of computational errors. 
The naturally arising rectangular, polar, and circular (together called primitive) interval types are not closed under simple arithmetic operations, and their use yields overly relaxed bounds. 
The later introduced polygonal type, on the other hand, allows for arbitrarily precise representation of the above operations for a higher computational cost. 
We propose the polyarcular interval type as an effective extension of the previous types. 
The polyarcular interval can represent all primitive intervals and most of their arithmetic combinations precisely and has an approximation capability competing with
that of the polygonal interval.
In particular, in antenna tolerance analysis it can achieve perfect accuracy for lower computational cost then the polygonal type, which we show in a relevant case study. 
In this paper, we present a rigorous analysis of the arithmetic properties of all five interval types, involving a new algebro-geometric method of boundary analysis.
}




\pacs[Author contributions]{
Conceptualization and methodology: Gábor Geréb; 
Formal analysis and investigation: Gábor Geréb, András Sándor; 
Visualization: Gábor Geréb;
Writing -- original draft preparation: Gábor Geréb, András Sándor; 
Writing -- review and editing: Gábor Geréb, András Sándor; 
}

\maketitle
    

\section{Introduction}\label{sec:Introduction}

Interval analysis is an effective mathematical technique that allows for numerical analysis of problems involving sets. \cite{moore_introduction_2009} It has long been used to put bounds on computational errors and has recently been applied in robotics and robust control. It also proved to be useful for finding all the solutions of nonlinear equations and inequalities. \cite{jaulin_applied_2001} Interval arithmetic studies the properties of the numerical representation and arithmetic operations of intervals, and therefore it is an essential part of interval analysis.

Pioneered by Moore in the 1960s \cite{moore_interval_1963}, the interval arithmetic for analyzing real-valued computational errors was soon extended to complex numbers in the form of Cartesian and polar products by Boche \cite{boche_complex_1965}, and as circular regions by Gargantini and Henrici \cite{gargantini_circular_1971}. Hansen studied the linear algebra of complex intervals and introduced a generalized interval arithmetic \cite{hansen_generalized_1975}. The field received increased attention in the 1990s: Ohta et al. introduced the polygon interval arithmetic \cite{ohta_polygon_1990}, a journal titled Interval Computations (later renamed to Reliable Computing) was established \cite{kreinovich_interval_2023,kearfott_reliable_1995}, the Matlab/Octave software package called INTLAB was published by Rump \cite{rump_intlab_1999}, and the BLAS FORTRAN package received an interval extension \cite{dongarra_blas_1995}. Books on interval analysis and arithmetic were published by Petkovic \cite{petkovic_complex_1998}, Jaulin \cite{jaulin_applied_2001}, Moore \cite{moore_introduction_2009} and Dawood \cite{dawood_theories_2011}.

Complex interval arithmetic benefited greatly from its interconnections with geometric algebra, which is widely used in the fields of computer aided design, image processing, mathematical morphology, geometrical optics, and dynamical stability analysis.  The application of the Minkowski algebra to complex intervals opened up the possibility of the representation and arithmetic combination of intervals bounded by arbitrary explicit and implicit curves in the complex plane. \cite{farouki_minkowski_2001,farouki_boundary_2005} Efficient algorithms for calculating the Minkowski sum of polygons, borrowed from computational geometry, have been successfully applied in the tolerance analysis of antenna arrays. \cite{ohta_polygon_1990,ohta_nonconvex_2000,de_berg_computational_2008,anselmi_tolerance_2015,tenuti_minkowski_2017} The polygonal representation produced much more accurate results than the original representations, given that the vertex count was high enough (Fig. \ref{fig:primitive_arithmetic}). However, a high vertex count came with a higher computational cost.

In the tolerance analysis of sensor arrays, polar intervals defined by independent amplitude and phase intervals are typically the primary operands of the evaluation. Motivated by the success of the polygonal representation, and its shortcomings in representing polar and circular intervals and their arithmetic combinations, we set out to find a more suitable interval type. Replacing vertices by circular arcs came as a natural extension, and lead to a new interval type, the polyarc bounded (polyarcular) interval. By providing perfect representation for a much wider set of intervals in return for a moderate increase in complexity, the new interval type suited our application well. 

Polyarc is an explicit curve type, defined by an ordered set of circular arcs and consists of the defining arcs and implicit edges between them (Fig. \ref{fig:polyarc_curve}). It allows the exact representation of the boundary of the rectangular, polar, circular, and polygonal intervals. It is closed under the addition, negative, reciprocal, union, and intersection operations, and in some cases under multiplication, too. 
As a by-product of the development process, we reviewed the literature on the existing complex interval types, and collected the relevant theorems from geometric algebra to produce a summary of the arithmetic and computational properties of complex interval types (Fig. \ref{fig:complex_subspaces}).

In this paper, we provide a review of the rectangular, circular, polar, and polygonal interval types and present the previously unpublished polyarcular type. We present a rigorous mathematical analysis of their arithmetic properties based on Minkowski algebraic theorems, and compare their computational properties. To demonstrate the advantages of the new type, we show an interval analytical case study where it outperforms the existing types. 

We expect this article to be of interest to both theoretical and applied researchers; therefore, we collected results with more theoretical relevance in Sections \ref{sec:complex_interval_subspaces} and \ref{sec:arithmetic_properties}, and those with more application relevance in Sections \ref{sec:computational_properties} and \ref{sec:case_study}. 

In Section \ref{sec:complex_interval_subspaces} we define the metric space of complex intervals and the subspaces of its finite data representations. We show that primitive intervals can be approximated by polygonal intervals and perfectly described by polyarcular intervals.

In Section \ref{sec:arithmetic_properties} we consider basic arithmetic and set operations applied to complex intervals. We consider their boundaries relying on Minkowski algebra and Matheron's work\cite{matheron_random_1975}, as well as on geometric algebra studies of Farouki et al. \cite{farouki_algorithms_2000,farouki_minkowski_2001}. We show that by Gauss map matching we can identify the operand boundary segments that are relevant for the evaluation of the result boundary. We also present three methods for the analysis of the result boundaries of unary and binary operations on complex intervals, one of which has never been used in this context so far to the best of our knowledge. Then we apply these methods to straight edges and circular arcs, which constitute the boundary segments of all the complex interval types considered in this paper, and show that the arithmetic properties of the newly proposed polyarcular interval is better than those of the primitive and polygonal intervals.

In Section \ref{sec:computational_properties} we consider the data representation and computation of complex intervals. We introduce the tightness measure and define the type casting operation. We show how type casting and arithmetic operations can cause a loss of tightness. Finally, we show two utility processes:  the trimming for extracting the simple boundary when an operation results a self-intersecting boundary, and the backtracking for identifying the subsets of the operands of an operation that maps to a point in the result interval. 

In Section \ref{sec:case_study}, we present a case study of antenna tolerance analysis, where the polyarcular interval type outperform the other four types. 

At the beginning of each section we provide a summary of its findings. Detailed derivations and illustrative examples of the proofs are provided in the Supplementary Material. For the sake of brevity, the typesetting of symbols in equations carry specific meanings, which is summarized in Table \ref{tab:Typesetting}.

\begin{table}[t]
    \centering
    \begin{tabular}{llc}
        Typesetting & Meaning & Example \\
        \hline
        Calligraphic ($\mathcal{I,R}$) & 
        Sets of curves or intervals & 
        $\mathcal{I}(\mathbb{R}) = \left\{\boldsymbol{t}= [\underline{t},\overline{t}] \big\vert \underline{t} \le \overline{t} \right\}$ 
        \\
        Bold-italic ($\boldsymbol{a,A}$) & 
        Sets & 
        $\boldsymbol{A}=\{A\in\mathbb{C} \big\vert \Re(A)=1\}$ 
        \\
        Italic ($a,A$) & 
        Numbers & 
        $A \in \boldsymbol{A} = \{t+it \big\vert t\in\boldsymbol{t} \}$ \\
        Lowercase ($a,\boldsymbol{a}$) & 
        Real numbers and sets & $a\in\boldsymbol{a}\subset \mathbb{R}$\\
        Uppercase ($A,\boldsymbol{A}$) & 
        Complex numbers and sets & 
        $A \in \boldsymbol{A} \subset \mathbb{C}$ \\
        Sans-sherif ($\mathsf{n,N}$) & 
        Natural numbers and sets &
        $\mathsf{n}\in \{1..\mathsf{N}\}\subset \mathbb{N}, 
        P_{2\mathsf{n}-1} = O_\mathsf{n} \!+\! r_\mathsf{n} 
        e^{i\overline{\varphi}_\mathsf{n}}$
        \\
        Under- and overlined ($\underline{t},\overline{t}$) & Infimum and supremum & 
        $[\underline{t},\overline{t}]\ni t, \,
        \underline{t}\leq t \leq \overline{t}$ \\
        Fraktur ($\Re,\Im$) & 
        Real and imaginary part & 
        $A=\Re(A)+i\Im(A)=A^\Re+iA^\Im$
    \end{tabular}
    \caption{Typesetting in mathematical expressions}
    \label{tab:Typesetting}
\end{table}

\section{Complex interval subspaces} \label{sec:complex_interval_subspaces}

In this section, we define the metric space of complex intervals. We also define some of its notable subspaces with finite data representations: the primitive intervals (rectangular, polar, and circular) that are characterized by their various real-valued projections (Fig. \ref{fig:primitive_arithmetic}), and the geometrical intervals (polygonal and polyarcular) that are characterized by their boundaries (Fig. \ref{fig:polyarc_curve}). We show that primitive intervals can be approximated by polygonal intervals and perfectly described by polyarcular intervals (Fig. \ref{fig:complex_subspaces}).

\subsection{Complex intervals} \label{sub:complex_intervals}

Complex intervals are all subsets of the complex plane $\mathbb{C}$, which we will consider our universe. Since there is no common definition for complex intervals, for this discussion we assume that all of them are connected and bounded sets.  As a useful simplification, we also assume that all connected and bounded intervals can be described by a simple, piece-wise smooth, closed boundary curve. 

\begin{definition} \label{def:jordan_curve_subspace}
    A simple curve is the continuous image of a closed real interval in the complex plane with no self-intersections. The set of piecewise smooth simple curves on the complex plane is 
    \[\mathcal{O}(\mathbb{C}) = 
    \{\boldsymbol{\Gamma} \!=\! \Gamma(\boldsymbol{t}) \big\vert \Gamma:\boldsymbol{t} \to \mathbb{C}, \text{piecewise smooth}\},
    \]
    where $\boldsymbol{t}$ is a real interval.
\end{definition}

\begin{definition}
    The set $\mathring{\mathcal{O}}(\mathbb{C})\subset\mathcal{O}(\mathbb{C})$ of piecewise smooth Jordan curves consists of piecewise smooth simple curves
    with matching endpoints:
    \(\Gamma(\underline{t}) = \Gamma(\overline{t}).\)
\end{definition}

\begin{definition} \label{def:complex_interval_space}
A complex set belongs to the set $\mathcal{I}(\mathbb{C})$ of complex intervals if it is bounded by a piecewise smooth Jordan curve:
\begin{equation*}
    \boldsymbol{A}\in\mathcal{I}(\mathbb{C}) \iff 
    \partial \boldsymbol{A} \in \mathring{\mathcal{O}}(\mathbb{C}).
\end{equation*}
\end{definition}

\begin{remark}
    In practice, intervals with holes -- such as the annulus $A=[1,2]e^{i[-\pi,\pi]}$ -- can be handled by narrow cuts leading to the holes. Such curves are called weakly simple.\cite{villafuerte_polygonal_2022}     
\end{remark}

The complex plane can be considered a 2-dimensional Euclidean space $\mathbb{R}^2$ endowed with a multiplication. As the complex space is naturally a metric space with the Euclidean distance, the set $2^\mathbb{C}$ of all its subsets can be endowed with a Hausdorff distance. This distance makes the various subsets of complex intervals metric subspaces with the inherited Hausdorff metric.

\begin{definition} \label{def:hausdorff}
    The Hausdorff distance of two complex sets $\boldsymbol{A},\boldsymbol{B} \subset \mathbb{C}$ is
    \[
    d(\boldsymbol{A},\boldsymbol{B})= \inf \left( \varepsilon \ge 0\big\vert \boldsymbol{A} \oplus \boldsymbol{D}(0,\varepsilon) \supset \boldsymbol{B} \text{ and } \boldsymbol{B} \oplus \boldsymbol{D}(0,\varepsilon) \supset \boldsymbol{A} \right),
    \]
    where $\boldsymbol{D}(0,\varepsilon)$ denotes the closed $\varepsilon$-disk around the origin.
\end{definition}

This is an equivalent reformulation of to the standard definition of the Hausdorff distance, in the language of Minkowski operations.

\begin{figure}[hbtp]
\centering
\includegraphics[height=165mm]{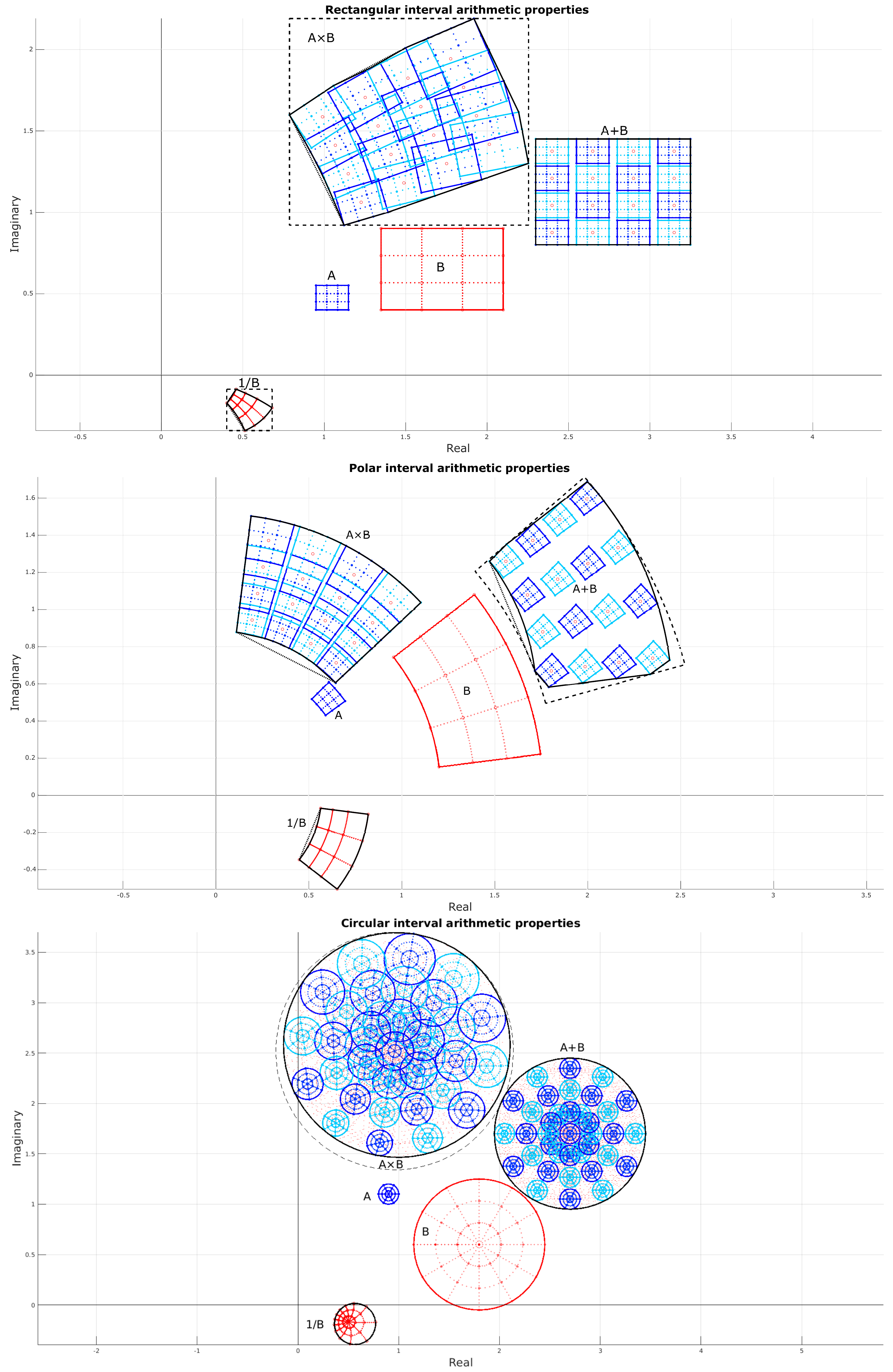}
\caption{Sum, product, and reciprocal of primitive intervals. Dark and light blue lines indicate the boundary of operand $\boldsymbol{A}$ and its translated and rotate-and-scaled copies ($\boldsymbol{A} \oplus\boldsymbol{B}$,$\boldsymbol{A}\otimes \boldsymbol{B}$), while red lines indicate the boundary of operand $\boldsymbol{B}$ and its reciprocal ($\boldsymbol{B}^{-1}$). Red dots identify internal points of the interval $\boldsymbol{B}$ used to translate and rotate-and-scale $\boldsymbol{A}$. Solid black line identifies the polyarcular operation results, while dotted line idenfies the convex polygonal operation results. The dashed black line identifies the boundary of the primitive wrapper of the results.}
\label{fig:primitive_arithmetic}
\end{figure}


\subsection{Primitive intervals}

The foundation stone of primitive complex intervals is the real interval.

\begin{definition} \label{def:real_interval_subspace}
    The set of closed real intervals is
    \[\mathcal{I}(\mathbb{R}) = \left\{\boldsymbol{t}= [\underline{t},\overline{t}] \big\vert \underline{t} \le \overline{t} \right\}
    , \quad [\underline{t},\overline{t}]=\{ x \in \mathbb{R} \big\vert \underline{t} \le x \le \overline{t}\} .\]   
\end{definition}

The three most common primitive complex interval types are the rectangular, the polar \cite{boche_complex_1965} and the circular \cite{gargantini_circular_1971} intervals. Illustrative examples can be found in Figure \ref{fig:primitive_arithmetic}.

\begin{definition} \label{def:rectangular_interval_subspace}
    A rectangular interval 
    \[
    \boldsymbol{a}+\boldsymbol{b}i=\{Z \in \mathbb{C} \; \big\vert \; \Re(Z) \in \boldsymbol{a}, \Im(Z) \in \boldsymbol{b}\}
    \]
    (with $\boldsymbol{a},\boldsymbol{b} \in \mathcal{I}(\mathbb{R})$) is the Cartesian product of two real intervals. The subspace of rectangular intervals is denoted by $\mathcal{R}(\mathbb{C})$.
\end{definition}

\begin{definition} \label{def:polar_interval_subspace}
A polar interval is defined by the polar product of two real intervals, that is
\[
        \boldsymbol{r} e^{i\boldsymbol{\varphi}} = 
        \{Z \in \mathbb{C} \; \big\vert \; \abs{Z} \in \boldsymbol{r}, \angle Z \in \boldsymbol{\varphi} \} 
\]
with $\boldsymbol{r},\boldsymbol{\varphi} \in \mathcal{I}(\mathbb{R})$.
The subspace of polar intervals is denoted by $\mathcal{P}(\mathbb{C})$.
\end{definition}

\begin{definition} \label{def:circular_interval_subspace}
    A circular interval is a closed disk in the complex plane, that is
\[
    O + [0,r]\cdot e^{i[-\pi,\pi]} = 
        \{ Z \in \mathbb{C} \;\big\vert\; \vert Z - O  \vert \leq r  \}
\]
with center $O \in \mathbb{C}$ and radius $r \in \mathbb{R}$. The set of these intervals is denoted by $\mathcal{C}(\mathbb{C})$.
\end{definition}

It can be shown that the boundaries of rectangular, polar, and circular intervals are rectangles, annular sectors, and circles, respectively, which are all piecewise smooth closed curves. Therefore, all primitive intervals are elements of the complex interval subspace.

\begin{lemma}\label{thm:primitive_interval_boundary}
    The boundaries of primitive intervals consist of edges and arcs.
\end{lemma}

\begin{corollary} \label{thm:primitive_intervals_are_complex_intervals}
All primitive intervals are complex intervals.
\begin{equation*}
    \mathcal{R}(\mathbb{C}) \cup 
    \mathcal{P}(\mathbb{C}) \cup
    \mathcal{C}(\mathbb{C}) \subset 
    \mathcal{I}(\mathbb{C})
\end{equation*}
\end{corollary}

An illustration of the subspace hierarchy can be found in Figure \ref{fig:complex_subspaces}.


\subsection{Geometric intervals}

Polygon interval arithmetic (PIA) \cite{ohta_polygon_1990} was introduced to represent uncertainty in robust control systems. Since a polygon can approximate any simple closed curve with a finite dataset, it can also represent any complex interval with arbitrary precision limited only by computational constraints.

\begin{definition} \label{def:linear_curve_subspace}
    An edge between $P_1,P_2 \in \mathbb{C}$ is
    \[
    \Gamma = \Bar{\Gamma}(P_1,P_2): [0,1] \to \mathbb{C}
    , \quad
    t \mapsto (1-t)P_1+tP_2.
    \]
    The set of edges is 
    $\Bar{\mathcal{O}}(\mathbb{C}) \subset \mathcal{O}(\mathbb{C}).$ 
\end{definition}

\begin{definition}\label{def:polygon}
    A closed polygonal curve can be given as
    \[    
    \Gamma: [0,\mathsf{n}] \to \mathbb{C}, \quad 
    \Gamma(t)=\Gamma_\mathsf{n}(t) \text{ for } t \in [\mathsf{n}-1,\mathsf{n}]
    \]
    with
    \[
    \Gamma_\mathsf{n}(t) \!=\! \Bar{\Gamma}(P_\mathsf{n},P_\mathsf{n+1})(t-(\mathsf{n}-1))
    , \quad
    P_\mathsf{N+1} = P_\mathsf{1},  \]
    for $\mathsf{n} \in \{1..\mathsf{N}\}$
    .
    The set of closed polygonal curves is 
    $\mathring{\Bar{\mathcal{O}}}(\mathbb{C}) \subset \mathring{\mathcal{O}}(\mathbb{C})$. 
\end{definition}

\begin{definition} \label{def:polygonal_interval_subspace}
The polygonal interval subspace $\mathcal{G}(\mathbb{C})$ is a collection of all complex intervals bounded by polygonal curves.
\begin{equation*}
    \boldsymbol{A} \in \mathcal{G}(\mathbb{C}) \subset \mathcal{I}(\mathbb{C})  \iff 
    \partial \boldsymbol{A} \in \mathring{\Bar{\mathcal{O}}} 
\end{equation*}
\end{definition}

The boundaries of primitive intervals can be represented as ordered sets of straight edges ($\Bar{\Gamma}$) and circular arcs ($\breve{\Gamma}$). Polygons can, however, represent only the edges precisely, while they need to approximate arcs with a number of edges. 

\begin{remark}
Polygonal intervals are complex intervals. All rectangular intervals are polygonal intervals, but polar and circular intervals are not.
\begin{equation*}
    \mathcal{R}(\mathbb{C}) \subset \mathcal{G}(\mathbb{C}), 
    \mathcal{P}(\mathbb{C}) \not\subset \mathcal{G}(\mathbb{C}),
    \mathcal{C}(\mathbb{C}) \not\subset \mathcal{G}(\mathbb{C}),
\end{equation*}
\end{remark}

\begin{figure*}[t]
\centering
\includegraphics[width=5.5in]{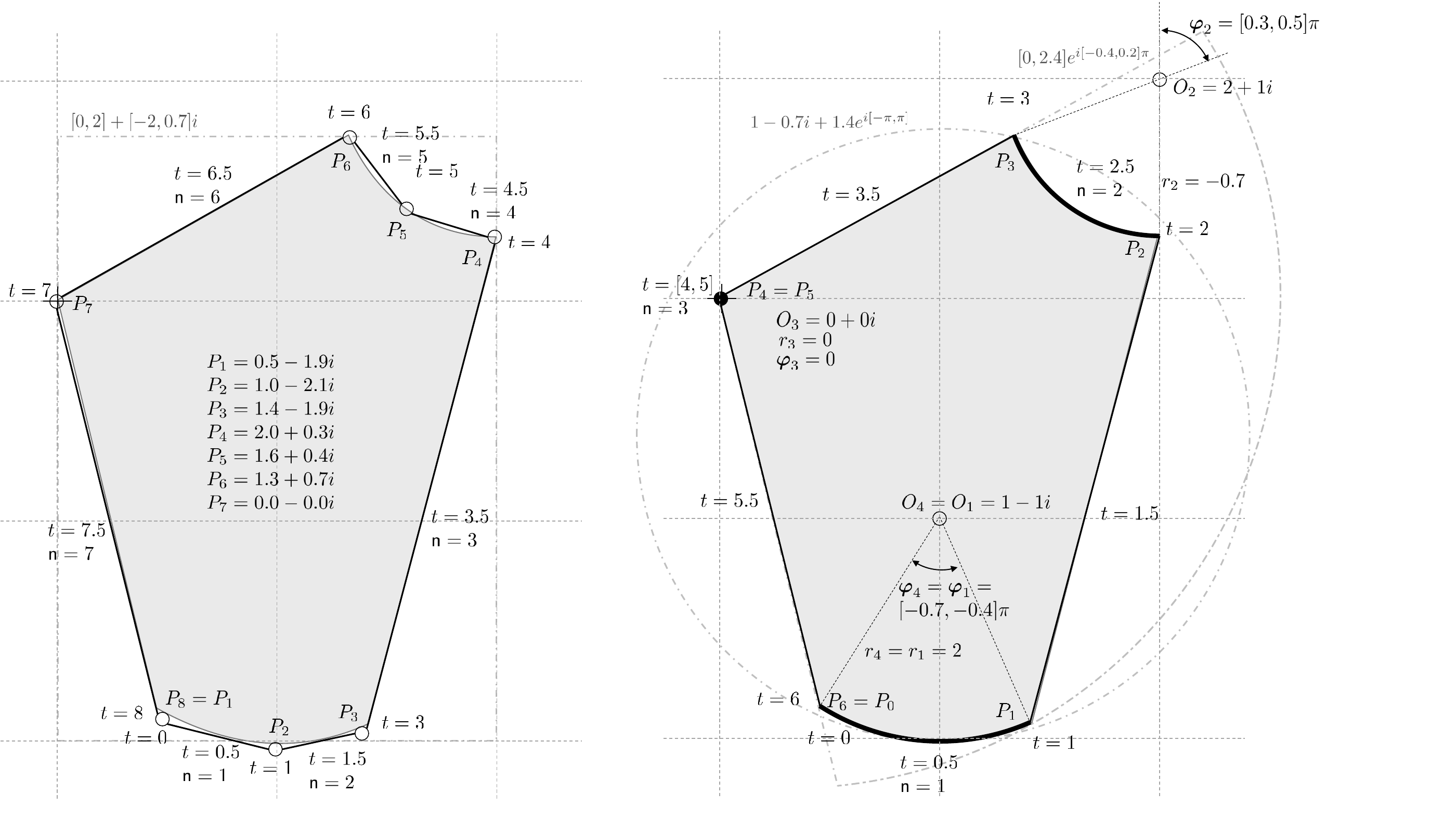}
\caption{Example of a complex interval represented by various complex interval types. The gray area identifies the complex interval; solid black lines identify the polygonal and polyarcular boundary curves. The grey dash-dotted lines indicate the boundaries of inclusive rectangular, polar and circular interval type objects.}
\label{fig:polyarc_curve}
\end{figure*}

It comes as a natural conclusion that a curve consisting of edges and arcs could precisely represent all primitive intervals. The polyarcular curve is a direct extension of the polygonal curve, where instead of complex-valued vertices, we use arcs with a complex-valued center point, a radius and an argument interval each to define a curve consisting of the defining arcs and implicit edges (Fig. \ref{fig:polyarc_curve}.) Based on this curve, we propose a new interval subspace that contains all the primitive complex interval subspaces and has the same or better approximation capability for arbitrary complex intervals as the polygonal intervals. Polyarcular intervals are elements of the complex interval subspace by definition \ref{def:complex_interval_space}.

\begin{figure*}[t]
\centering
\includegraphics[width=4.5in]{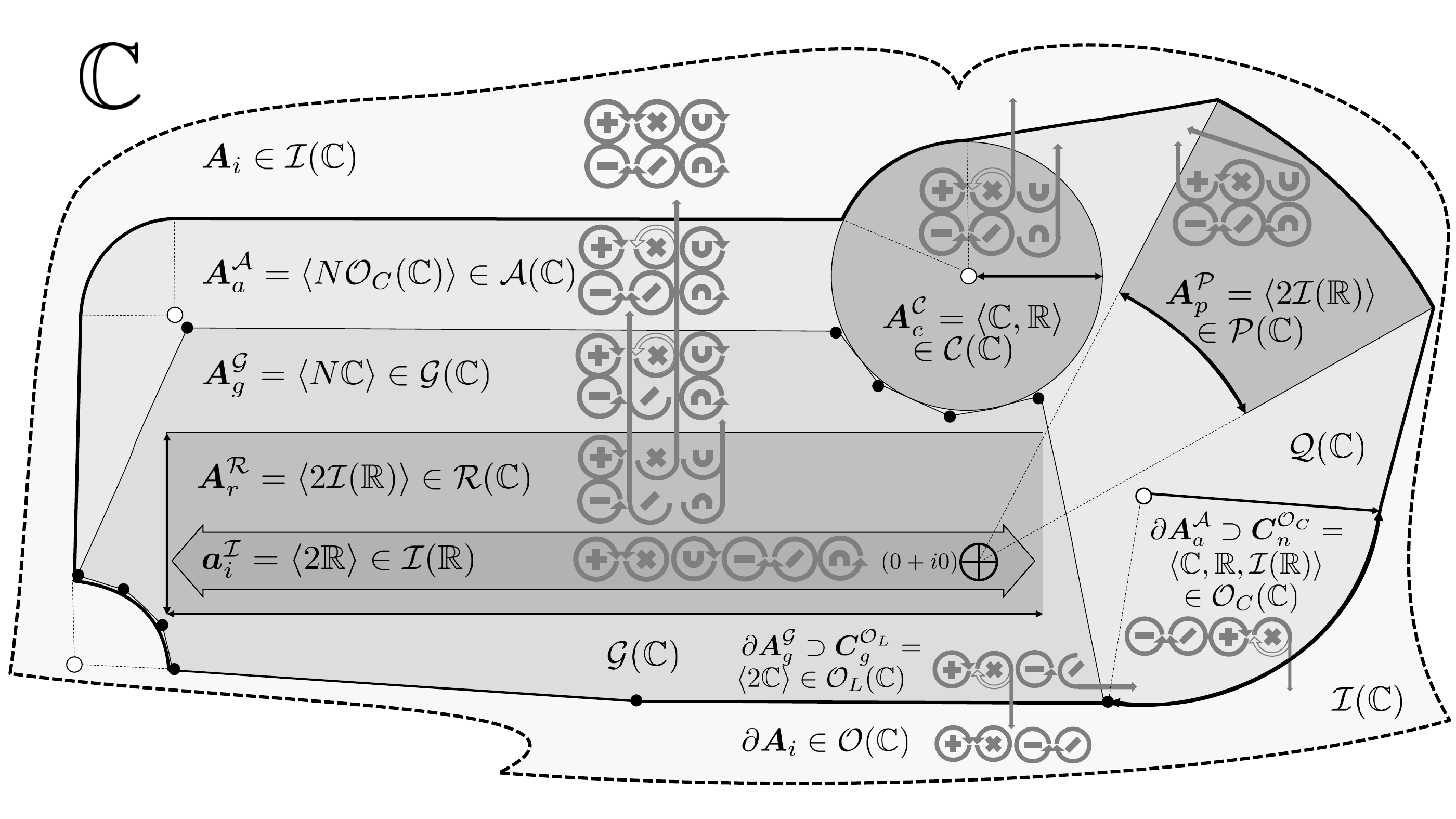}
\caption{Graphical summary of complex interval subspaces, types and operations. 
We figure illustrates a number of complex intervals from the following subspaces: rectangular ($\mathcal{R}$), polar ($\mathcal{P}$), circular ($\mathcal{C}$), polygonal ($\mathcal{G}$), polyarcular ($\mathcal{A}$). In the equation the interval symbol with the type designation is followed by the amount and type of stored parameters, and finally the notation of the narrowest subspace to which it belongs. The $\oplus$ and $\otimes$ indicate the binary addition and multiplication operations; $\ominus$ and $\oslash$ indicate the unary negative and reciprocal operations; $\cup$ and $\cap$ indicate the binary union and intersection operations. Arrows around operators indicate whether the result of an operation is in the same subspace or is outside of it. Hollow arrows indicate significant special cases (trivial exceptions, such as division by zero and union of non-intersecting intervals are not indicated).}
\label{fig:complex_subspaces}
\end{figure*}

\begin{definition} \label{def:circular_curve_subspace}
    An arc is
    \[
    \Gamma = \breve{\Gamma}(O,r,\boldsymbol{\varphi}):  [0,1] \to \mathbb{C}
    , \quad
    t \mapsto O+re^{(1-t)i\underline{\varphi}+ti\overline{\varphi}}
    ,
    \]
    given
    $ O \in\mathbb{C},
    r\in \mathbb{R},
    \boldsymbol{\varphi} \in \mathcal{I}(\mathbb{R}).$
    The set of arcs is 
    $\breve{\mathcal{O}}(\mathbb{C}) \subset \mathcal{O}(\mathbb{C})$.
\end{definition}

\begin{remark}
    When arcs are used to form polyarcular curves, the radius value determines whether the arc will be convex ($r>0$), concave ($r<0$) or just a point ($r=0$). To maintain the required counterclockwise point order in the concave curve segments, the following modified arc function must be used:
    \[
    t \mapsto P + re^{
    (\frac{1}{2}-t') i \underline{\varphi} + 
    (\frac{1}{2}+t') i\overline{\varphi}  }, \quad
    t' = \mathrm{sgn}(r)\left(\frac{1}{2}-t\right)
    .\]
\end{remark}

\begin{definition}\label{def:polyarc}
    A polyarcular curve can be given as
    \[
    \Gamma: [0,2\mathsf{N}] \to \mathbb{C}, \quad \Gamma(t)=
    \begin{cases}
        \breve{\Gamma}_\mathsf{n}
        (t-2\mathsf{n}+2))
        & \text{for } t \in [2\mathsf{n}-2,2\mathsf{n}-1] ,
        \\
        \Bar{\Gamma}_\mathsf{n}
        (t-2\mathsf{n}+1)
        & \text{for } t \in [2\mathsf{n}-1,2\mathsf{n}],
    \end{cases}
    \]
    for $\mathsf{n} \in \{1.. \mathsf{N}\}$, with alternating arc and edge pieces (either can be a single point  if necessary). Precisely
    \[
    \begin{array}{rl}
         \breve{\Gamma}_\mathsf{n} (t) = &
         \breve{\Gamma}(O_\mathsf{n},r_\mathsf{n},\boldsymbol{\varphi}_\mathsf{n})(t),\\
         \Bar{\Gamma}_\mathsf{n} (t) = &
         \Bar{\Gamma}(P_{2\mathsf{n}-1},P_{2\mathsf{n}})(t),
    \end{array}
    \]
    satisfying
    \[
    P_{2\mathsf{n}-1} = O_\mathsf{n} \!+\! 
                     r_\mathsf{n}
                     e^{i\overline{\varphi}_\mathsf{n}},
    \quad
    P_{2\mathsf{n}} = O_\mathsf{n+1} \!+\! 
                     r_\mathsf{n+1} 
                     e^{i\underline{\varphi}_\mathsf{n+1}},
    \]
    with 
    $O_\mathsf{N+1}=O_1,\,
    r_\mathsf{N+1}=r_1, \,
    \boldsymbol{\varphi}_\mathsf{N+1} = \boldsymbol{\varphi}_1, P_{2\mathsf{N}}=P_0$.
    
    
    The set of polyarcular curves is 
    $\mathring{\breve{\mathcal{O}}}(\mathbb{C}) \subset \mathring{\mathcal{O}}(\mathbb{C})$.
\end{definition}

\begin{definition} \label{def:polyarcular_interval_subspace}
The polyarcular interval subspace $\mathcal{A}(\mathbb{C})$ is a collection of all complex intervals bounded by polyarcular curves.
\begin{equation*}
\begin{array}{l}
    \boldsymbol{A} \in \mathcal{A}(\mathbb{C}) \subset \mathcal{I}(\mathbb{C}) \iff
    \partial \boldsymbol{A} \in \mathring{\breve{\mathcal{O}}}
\end{array}
\end{equation*}
\end{definition}

\begin{corollary} \label{thm:primitive_intervals_are_polyarcular_intervals}
All primitive and polygonal intervals are polyarcular intervals.
\begin{equation*}
    (\mathcal{R}(\mathbb{C}) \cup 
    \mathcal{P}(\mathbb{C}) \cup
    \mathcal{C}(\mathbb{C})) \cup
    \mathcal{G}(\mathbb{C}) \subset 
    \mathcal{A}(\mathbb{C})
\end{equation*}
\end{corollary}

An illustration of the subspace hierarchy can be found in Figure \ref{fig:complex_subspaces}. 


\section{Arithmetic properties}
\label{sec:arithmetic_properties}

In this section, we consider basic arithmetic and set operations applied to complex intervals. The aim is to find out whether the results of these operations belong to the same subspace as the operands, which is important for the computational implementation. 
Although there are existing arithmetic algorithms for primitive complex intervals based on their unique definitions \cite{boche_complex_1965,gargantini_circular_1971,candau_complex_2006}, we focus on their boundaries in order to make this arithmetic framework general across the different subspaces. We are highly relying on the Minkowski algebra structure and Matheron's work \cite{matheron_random_1975}, as well as on geometric algebra studies of Farouki et al. \cite{farouki_algorithms_2000,farouki_minkowski_2001}.  
We show that by Gauss map matching we can identify the subset of either operand that contributes to a particular element in the result interval (Fig. \ref{fig:Gauss_map}). 
We present three methods for the analysis of the result boundary of unary and binary operations on complex intervals (Figures \ref{fig:mixed_methodc}, \ref{fig:parametric_method} and \ref{fig:implicit_method}). 
Then we apply these methods to straight edges and circular arcs, which constitute the boundary segments of all the complex interval representations considered in this paper. 
We show that the arithmetic properties of the newly proposed polyarcular interval is better than those of the primitive intervals and the polygonal interval \cite{boche_complex_1965,gargantini_circular_1971,petkovic_complex_1998,candau_complex_2006,ohta_polygon_1990,ohta_nonconvex_2000} (Table \ref{tab:Arithmetic_properties}, Figures \ref{fig:primitive_arithmetic} and \ref{fig:complex_subspaces}).


\begin{table}[t]
    \centering
    \begin{tabular}{c|ccccc}
        $\boldsymbol{A,B}\in$
         & $\mathcal{R}(\mathbb{C})$
         & $\mathcal{P}(\mathbb{C})$
         & $\mathcal{C}(\mathbb{C})$
         & $\mathcal{G}(\mathbb{C})$
         & $\mathcal{A}(\mathbb{C})$
        \\
        \hline
        $\boldsymbol{A}+\boldsymbol{B}$
        & $\mathcal{R}(\mathbb{C})$
        & $\mathcal{A}(\mathbb{C})$
        & $\mathcal{C}(\mathbb{C})$
        & $\mathcal{G}(\mathbb{C})$
        & $\mathcal{A}(\mathbb{C})$
        \\
        $\boldsymbol{A} \times \boldsymbol{B}$
        & $\color{red}\mathcal{I}(\mathbb{C})$
        & $\mathcal{P}(\mathbb{C})$
        & $\color{red}\mathcal{I}(\mathbb{C})$
        & $\color{red}\mathcal{I}(\mathbb{C})$
        & $\color{red}\mathcal{I}(\mathbb{C})^*$
        \\
        $-\boldsymbol{A}$
        & $\mathcal{R}(\mathbb{C})$
        & $\mathcal{P}(\mathbb{C})$
        & $\mathcal{C}(\mathbb{C})$
        & $\mathcal{G}(\mathbb{C})$
        & $\mathcal{A}(\mathbb{C})$
        \\
        $\boldsymbol{A}^{-1}$
        & $\color{blue}\mathcal{A}(\mathbb{C})$
        & $\mathcal{P}(\mathbb{C})$
        & $\mathcal{C}(\mathbb{C})$
        & $\color{blue}\mathcal{A}(\mathbb{C})$
        & $\mathcal{A}(\mathbb{C})$
        \\
        $\boldsymbol{A}\cap\boldsymbol{B}$
        & $\mathcal{R}(\mathbb{C})$
        & $\mathcal{P}(\mathbb{C})$
        & $\mathcal{A}(\mathbb{C})$
        & $\mathcal{G}(\mathbb{C})$
        & $\mathcal{A}(\mathbb{C})$
        \\
        $\boldsymbol{A}\cup\boldsymbol{B}$
        & $\mathcal{G}(\mathbb{C})$
        & $\mathcal{A}(\mathbb{C})$
        & $\mathcal{A}(\mathbb{C})$
        & $\mathcal{G}(\mathbb{C})$
        & $\mathcal{A}(\mathbb{C})$
\end{tabular}
    \caption{Arithmetic properties of complex interval subspaces.
    The red text indicates when the operation points outside the subspace, except when it points to the subspace of the proposed new type, which is indicated by the blue text.
    (* In the special case where one operand is a member of polar intervals, the result is polyarcular, which has a relevance in our case study in Section \ref{sec:case_study}.)}
    \label{tab:Arithmetic_properties}
\end{table}

\subsection{Element-wise operations}

Arithmetic operations on intervals are typically defined by the Minkowski algebra, which is the collection of pointwise operations on sets.  Minkowski addition and multiplication are both closed binary operations ($\mathcal{I}(\mathbb{C}) \times_c \mathcal{I}(\mathbb{C}) \to \mathcal{I}(\mathbb{C})$, where $\times_c$ indicates the Cartesian product), and the element-wise negation is a closed unary operation ($\mathcal{I}(\mathbb{C}) \to \mathcal{I}(\mathbb{C})$). This means that the results of the sum, product and negative of complex intervals are also complex intervals. However, the reciprocal is only a partial unary operation because it maps the intervals that include the complex zero outside the subspace ($\mathcal{I}(\mathbb{C}) \to \mathcal{I}(\mathbb{C} \cup \frac{1}{0})$). The element-wise subtraction and division can be defined as the combination of the operations mentioned above and consequently they are closed and partial binary operations respectively. For a graphical summary, see Figure \ref{fig:complex_subspaces}.

\begin{definition}
For $\boldsymbol{A},\boldsymbol{B} \in \mathcal{I}(\mathbb{C})$
\[
\begin{array}{cl}
     \boldsymbol{A} \oplus \boldsymbol{B} &= 
     \{A + B \big\vert A \in \boldsymbol{A}, B \in \boldsymbol{B}\},
     \\
     \boldsymbol{A} \otimes \boldsymbol{B} &= 
     \{A \times B  \big\vert A \in \boldsymbol{A}, B \in \boldsymbol{B}\},
     \\
     -\boldsymbol{A} &= \{-A \big\vert A \in \boldsymbol{A} \},
     \\
     \boldsymbol{A}^{-1} &= \{A^{-1} \big\vert A \in \boldsymbol{A} \},
     \\
     \boldsymbol{A} \ominus \boldsymbol{B} &= 
     \{A + (-B) \big\vert A \in \boldsymbol{A}, B \in \boldsymbol{B}\}=
     \boldsymbol{A} \oplus (-\boldsymbol{B}),
     \quad 
     \\
     \boldsymbol{A} \oslash \boldsymbol{B} &= 
     \{A \times B^{-1} \big\vert A \in \boldsymbol{A}, B \in \boldsymbol{B}\}=
     \boldsymbol{A} \otimes (\boldsymbol{B}^{-1}) 
     .
\end{array}
\]
\end{definition}

Similarly to the real intervals, the Minkowski addition and multiplication on complex intervals are commutative and associative but not distributive. \cite{giardina_morphological_1988,petkovic_complex_1998} The non-distributivity is also part of a larger question called the dependency problem. \cite{dawood_theories_2011,dawood_logical_2019}

\begin{theorem}
For $\boldsymbol{A},\boldsymbol{B},\boldsymbol{C} \in \mathcal{I}(\mathbb{C})$
\[
\label{eq:complex_interval_associativity_distributivity}
\begin{array}{ll}
    \boldsymbol{A} \oplus \boldsymbol{B}  = \boldsymbol{B} \oplus \boldsymbol{A},
    &
    \boldsymbol{A} \otimes \boldsymbol{B}  = \boldsymbol{B} \otimes \boldsymbol{A},
    \\
    (\boldsymbol{A} \oplus \boldsymbol{B}) \oplus \boldsymbol{C}  = \boldsymbol{A} \oplus (\boldsymbol{B} \oplus \boldsymbol{C}),
    &
    (\boldsymbol{A} \otimes \boldsymbol{B}) \otimes \boldsymbol{C}  = \boldsymbol{A} \otimes (\boldsymbol{B} \otimes \boldsymbol{C})
    ,
    \\
    \boldsymbol{A} \otimes (\boldsymbol{B} \oplus \boldsymbol{C}) \subseteq (\boldsymbol{A} \otimes \boldsymbol{B}) \oplus  (\boldsymbol{A} \otimes \boldsymbol{C})
    .
\end{array}
\]
\end{theorem}

Also, there is no additive and multiplicative inverse element for non-degenerate complex intervals, and therefore no inverse operations either. (Intervals consisting of a single non-zero complex number have inverses.) We only have the following trivial inclusions.

\begin{equation}
\label{eq:complex_interval_additive_multiplicative_inverse}
\begin{array}{ll}
    \boldsymbol{A} \oplus (-\boldsymbol{A}) \supseteq 0,
    &
    (\boldsymbol{A} \oplus \boldsymbol{B}) \ominus \boldsymbol{B} \supseteq \boldsymbol{A},
    \\
    \boldsymbol{A} \otimes \boldsymbol{A}^{-1} \supseteq 1,
    &
    (\boldsymbol{A} \otimes \boldsymbol{B}) \oslash \boldsymbol{B} \supseteq \boldsymbol{A}
    .
\end{array}
\end{equation}

Without a loss of generality we can freely translate (add a complex number to) the operands of the addition or negative operations, and rotate-and-scale (multiply by a complex number) the operands of the multiplication or reciprocal operation \cite{farouki_minkowski_2001}. This allows the operation to be performed on the normalized sets and then get the unnormalized result as
\begin{equation}\label{eq:Curve_segment_normalization}
\begin{array}{l}
    \boldsymbol{A} \oplus \boldsymbol{B} = 
    (\boldsymbol{A} + U) \oplus (\boldsymbol{B} + V) - (U + V), 
    \\
    -\boldsymbol{A} = -(\boldsymbol{A} + U) + U,
    \\
    \boldsymbol{A} \otimes \boldsymbol{B} = 
    (\boldsymbol{A} \times W) \otimes (\boldsymbol{B} \times Z) \times (W^{-1}Z^{-1}),    
     \\
    \boldsymbol{A}^{-1} = (\boldsymbol{A} \times W)^{-1} \times W
    .
\end{array}
\end{equation}

Set operations can be easily defined by element-wise conditions. The union and intersection are both partial binary operations because the union and intersection of two disconnected complex intervals are not bounded by a Jordan curve.

\subsection{Backtracking}
 
The set operations allow the reformulation of the arithmetic operations as

\begin{equation} \label{eq:modified_Minkowski_sum_prod}
\begin{array}{l}
    \boldsymbol{A} \oplus \boldsymbol{B} = 
    \bigcup\limits_{B \in \boldsymbol{B}} \boldsymbol{A} + B = 
    \{ Z \big\vert \boldsymbol{A} \cap (-\boldsymbol{B} + Z) \neq \emptyset\}
    \\
    \boldsymbol{A} \otimes \boldsymbol{B} = 
    \bigcup\limits_{B \in \boldsymbol{B}} \boldsymbol{A} \times B = 
    \{ Z \big\vert \boldsymbol{A} \cap (\boldsymbol{B}^{-1} \times Z) \neq \emptyset\}
    .
\end{array}
\end{equation}

Using these identities, we can identify the subset of either operand that contributes to a particular element in the result interval.

\begin{definition} \label{def:complex_interval_backtracking}
Let $\boldsymbol{A},\boldsymbol{B} \in \mathcal{I}(\mathbb C)$. For an element of the sum $Z \in \boldsymbol{A} \oplus \boldsymbol{B}$, we define
\[ 
    \boldsymbol{A}_Z = \boldsymbol{A} \cap (-\boldsymbol{B} \!+\! Z),\quad
    \boldsymbol{B}_Z = \boldsymbol{B} \cap (-\boldsymbol{A} \!+\! Z),
\]
satisfying
\[
    \forall A \in \boldsymbol{A}_Z, 
    \exists B \in \boldsymbol{B}_Z: A \!+\! B \!=\! Z,
\]
Similarly, for a product element $Z \in \boldsymbol{A} \otimes \boldsymbol{B}$, let
\[
    \boldsymbol{A}_Z = \boldsymbol{A} \cap (\boldsymbol{B}^{-1} \!\times\! Z), \quad
    \boldsymbol{B}_Z = \boldsymbol{B} \cap (\boldsymbol{A}^{-1} \!\times\! Z),
\]
that satisfies
\[
    \forall A \!\in\! \boldsymbol{A}_Z, 
    \exists B \!\in\! \boldsymbol{B}_Z: A \!\times\! B \!=\! Z
    .
\]
\end{definition}

\subsection{Boundary analysis}  \label{sub:boundary_analysis}

For an element $Z$ of the result boundary the subsets $\boldsymbol{A}_Z,\boldsymbol{B}_Z$ of definition \ref{def:complex_interval_backtracking} will be subsets of the operand boundaries, and in many cases they consist of single points. Therefore, the result boundary of the arithmetic operations can be defined using only the boundary points of the operands. \cite{farouki_algorithms_2000}

\begin{proposition}  \label{prop:Minkowski_boundary}
For $\boldsymbol{A},\boldsymbol{B} \in \mathcal{I}(\mathbb C)$.
\[
\begin{array}{ll}
    \partial (\boldsymbol{A} \oplus \boldsymbol{B}) \subseteq 
    \partial \boldsymbol{A} \oplus \partial \boldsymbol{B},
    \quad
     \partial (-\boldsymbol{A}) = - (\partial \boldsymbol{A}),
    \\
    \partial (\boldsymbol{A} \otimes \boldsymbol{B}) \subseteq 
    \partial \boldsymbol{A} \otimes \partial \boldsymbol{B},
    \quad
    \partial (\boldsymbol{A}^{-1}) = (\partial \boldsymbol{A})^{-1},
    \\
    \partial (\boldsymbol{A} \cup \boldsymbol{B}) \subseteq
    \partial \boldsymbol{A}  \cup \partial \boldsymbol{B},
    \\
    \partial (\boldsymbol{A} \cap \boldsymbol{B}) \subseteq 
    \partial \boldsymbol{A}  \cup \partial \boldsymbol{B}
    .
\end{array}
\]
\end{proposition}




\begin{figure}[t!]
    \centering
    \includegraphics[width=4in]{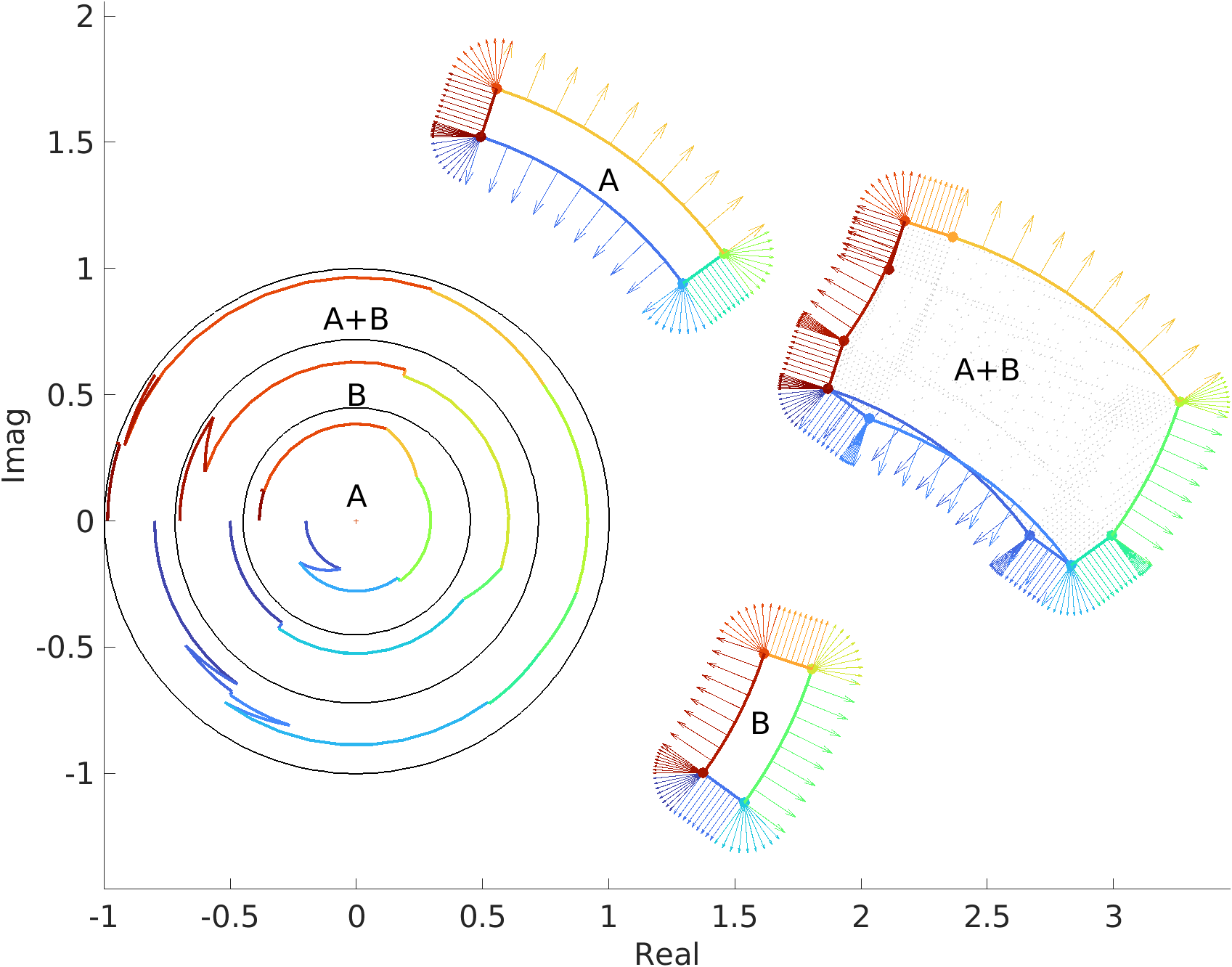}
    \caption{Two polar intervals and their sum combined with their Gauss maps on the same plot. The three concentric curves in the unit circle shows their Gauss maps, where the argument and the color indicates the curve normal angle, and the radius indicates the position on the curve. The arrows show the curve normal vectors.}
    \label{fig:Gauss_map}
\end{figure}

\begin{definition} \label{def:boundary_segment}
    An interval boundary segment is a regular piece of a complex interval boundary (see Definition \ref{def:complex_interval_space}).

    \begin{equation*}
    \boldsymbol{\Gamma}_{A,\mathsf{n}} = 
    \Gamma_{A,\mathsf{n}}(\boldsymbol{s})
    \subset
    \partial\boldsymbol{A} 
      \in \mathring{\mathcal{O}}
\end{equation*}

\end{definition}

\begin{definition} \label{def:boundary_vertex}
    An interval boundary vertex is the intersection point of two interval boundary segments (see definition \ref{def:boundary_segment}).
    \begin{equation}
        P_{A,\mathsf{n}} =  
        \boldsymbol{\Gamma}_{A,\mathsf{n-1}} \cap 
        \boldsymbol{\Gamma}_{A,\mathsf{n}} , \quad
        P_{A,1} = \boldsymbol{\Gamma}_{A,\mathsf{N}} \cap 
        \boldsymbol{\Gamma}_{A,\mathsf{1}},
        \mathsf{n}\in\{1..\mathsf{N}\}
    \end{equation}
\end{definition}

\subsubsection{Gauss map} \label{sub:Gauss_map}

Gauss map matching \cite{farouki_algorithms_2000,farouki_minkowski_2001} or curve convolution \cite{lee_minkowski_1999} is a direct way to identify the operand elements that contribute to the boundary of the sum of two curves. It requires the matching of the normal arguments of the boundary curves. The resulting sum is a boundary point with the same argument as the contributing points (Fig. \ref{fig:Gauss_map}.) This method can be also applied to the boundary of the product using the logarithms of the curves, leading to an equation with the argument of the normalized curve normal.

\begin{proposition}[Gauss map matching] \label{prop:Gauss_map_matching}
For two regular curves $F,G: \mathbb{R} \to \mathbb{C}$
\begin{align*}
    F(s_0) + G(t_0) \in 
    \partial \{ F(\mathbb{R}) + G(\mathbb{R}) \}
    \implies
    \angle F'(s_0) = \angle G'(t_0),
    \\
    F(s_0) \times G(t_0) \in 
    \partial \{ F(\mathbb{R}) \times G(\mathbb{R}) \}
    \implies
    \angle \frac{F'(s_0)}{F(s_0)} = 
    \angle \frac{G'(t_0)}{G(t_0)}.
\end{align*}
\end{proposition}

\begin{remark}
    The normalization of the operands shown in \eqref{eq:Curve_segment_normalization} does not affect their Gauss maps.
\end{remark}


Let us consider the Gauss map
\[
\gamma(F): \boldsymbol{s} \to \mathbb{R}, \quad 
\gamma(F,s) = \angle i F'(s),
\]
that returns the argument of the normal to the curve. Let us denote the argument of $F$ by $\Dot{\gamma}(F,s)=\angle F(s)$. The logarithmic Gauss map function is then defined as 
\[
\mathring{\gamma}(F,s)=\gamma(\log F,s)=\angle \frac{F'(s)}{F(s)} =\gamma_F-\Dot{\gamma}_F
\]
with a consistent choice of complex logarithm.

Let us apply the Gauss map on a curve segment $F(\boldsymbol{s})$. The result is the real set
\[\boldsymbol{\gamma}_F= \gamma(F,\boldsymbol{s}) =\{\gamma(F,s)\big\vert s\in\boldsymbol{s}\}.
\]
 The same applies to the logarithmic Gauss map.

Let us extend the Gauss map (and the logarithmic Gauss map) to a vertex $P$ at the intersection of two neighboring curve segments $F_1(\boldsymbol{s}_1)$ and $F_2(\boldsymbol{s}_2)$. We define $\gamma_P$ to be a monotonic function onto the interval $[\gamma_{F_1}(\overline{s_1}),\gamma_{F_2}(\underline{s_2})]$ with a consistent choice of the argument and logarithm functions.

We can now apply the Gauss map matching (Proposition \ref{prop:Gauss_map_matching}) to this setup. 

\begin{proposition} \label{prop:boundary_Gauss_map}
    For two boundary segments 
    $\boldsymbol{\Gamma}_{A,\mathsf{n}}\subset\partial \boldsymbol{A}$ and 
    $\boldsymbol{\Gamma}_{B,\mathsf{k}}\subset \partial\boldsymbol{B}$
\[
\begin{array}{cc}
\boldsymbol{\Gamma}_{A,\mathsf{n}} \oplus \boldsymbol{\Gamma}_{B,\mathsf{k}} \subset 
\partial(\boldsymbol{A}\oplus\boldsymbol{B}) \implies
\boldsymbol{\gamma}_{\Gamma_{A,\mathsf{n}}} \cap 
\boldsymbol{\gamma}_{\Gamma_{B,\mathsf{k}}} \neq \emptyset,
\\
\boldsymbol{\Gamma}_{A,\mathsf{n}} \otimes \boldsymbol{\Gamma}_{B,\mathsf{k}} \subset 
\partial(\boldsymbol{A}\otimes\boldsymbol{B}) \implies
\mathring{\boldsymbol{\gamma}}_{\Gamma_{A,\mathsf{n}}} \cap 
\mathring{\boldsymbol{\gamma}}_{\Gamma_{B,\mathsf{k}}} \neq \emptyset.
\end{array}
\]
   For a boundary segment
    $\boldsymbol{\Gamma}_{A,\mathsf{n}}\subset\partial \boldsymbol{A}$ 
    and a boundary vertex
    $P_{B,\mathsf{k}}\in \partial\boldsymbol{B}$
\[
\begin{array}{cc}
\boldsymbol{\Gamma}_{A,\mathsf{n}} + P_{B,\mathsf{k}} \subset 
\partial(\boldsymbol{A}\oplus\boldsymbol{B}) \implies
\boldsymbol{\gamma}_{\Gamma_{A,\mathsf{n}}} \cap 
\boldsymbol{\gamma}_{P_{B,\mathsf{k}}} \neq \emptyset,
\\
\boldsymbol{\Gamma}_{A,\mathsf{n}} \times P_{B,\mathsf{k}} \subset 
\partial(\boldsymbol{A}\otimes\boldsymbol{B}) \implies
\mathring{\boldsymbol{\gamma}}_{\Gamma_{A,\mathsf{n}}} \cap 
\mathring{\boldsymbol{\gamma}}_{P_{B,\mathsf{k}}} \neq \emptyset.
\end{array}
\]
    For two boundary vertices
    $P_{A,\mathsf{n}}\in\partial \boldsymbol{A}$ 
    and
    $P_{B,\mathsf{k}}\in \partial\boldsymbol{B}$
\[
\begin{array}{cc}
P_{A,\mathsf{n}} + P_{B,\mathsf{k}} \subset 
\partial(\boldsymbol{A}\oplus\boldsymbol{B}) \implies
\boldsymbol{\gamma}_{P_{A,\mathsf{n}}} \cap 
\boldsymbol{\gamma}_{P_{B,\mathsf{k}}} \neq \emptyset,
\\
P_{A,\mathsf{n}} \times P_{B,\mathsf{k}} \subset
\partial(\boldsymbol{A}\otimes\boldsymbol{B}) \implies
\mathring{\boldsymbol{\gamma}}_{P_{A,\mathsf{n}}} \cap 
\mathring{\boldsymbol{\gamma}}_{P_{B,\mathsf{k}}} \neq \emptyset.
\end{array}
\]
\end{proposition}

\subsubsection{Arithmetic combination of polyarcular boundaries}
\label{sub:arithmetic_properties_of_polyarcular_intervals}

As a consequence of Propositions \ref{prop:Minkowski_boundary} and \ref{prop:boundary_Gauss_map}, the arithmetic combination of two polyarcular intervals will result in a polyarcular interval iff the arithmetic combinations of the operand boundary segments with overlapping Gauss map produce only vertices, edges and arcs. The polyarcular arithmetic properties are summarized by Figure \ref{fig:complex_subspaces} and Tables \ref{tab:Arithmetic_properties} and \ref{tab:envelopes}.

Since the combination of a vertex with any other boundary element type is a simple translation for addition and a scale-rotation for multiplicative combination, the result boundary segment will be an edge or an arc. Therefore, in the following analysis, we will focus on the combinations of edges and arcs. We describe three methods using their implicit equations in the Cartesian and polar coordinate system, and their parametrizations, to determine the envelope and the parametric condition of the envelope being part of the boundary. We will denote the segments at hand as

\begin{equation}\label{eq:Curve_segment_functions}
\begin{array}{ll}
    \boldsymbol{\Gamma}=\boldsymbol{\Gamma}_{A,\mathsf{n}} = \{F(s) \big\vert s \in \boldsymbol{s}\} 
        &=\{x+iy \big\vert f(x,y)\!=\!0 , 
            s(x,y) \!\in\! \boldsymbol{s}\}  
        \\
        &=\{\rho e^{i\theta} \big\vert \mathring{f}(\rho,\theta) \!=\! 0 ,
            \mathring{s}(\rho,\theta) \!\in\! \boldsymbol{s}\},
            
    \\
    \boldsymbol{\Gamma'}=\boldsymbol{\Gamma}_{B,\mathsf{k}} = \{G(t) \big\vert s \in \boldsymbol{t}\} 
    &=\{x+iy \big\vert g(x,y)\!=\!0 , t(x,y) \!\in\! \boldsymbol{t}\} 
    \\
    &= \{\rho e^{i\theta} \big\vert \mathring{g}(\rho,\theta) \!=\!0 , 
    \mathring{t}(\rho,\theta) \in \boldsymbol{t}\} 
    
    .
\end{array}
\end{equation}

Applying the normalization in \eqref{eq:Curve_segment_normalization}, our analysis of additive operations can be restricted to segments of zero-crossing lines ($\Bar{\mathcal{O}}_{0/}$) and zero-centered circles ($\breve{\mathcal{O}}_0$). While our analysis of multiplicative operations can be restricted to segments of zero-crossing horizontal and one-crossing vertical lines 
($\Bar{\mathcal{O}}_{0\_}$ and $\Bar{\mathcal{O}}_{1\vert}$) and zero-centered and one-centered circles ($\breve{\mathcal{O}}_0$ and $\breve{\mathcal{O}}_1$). We collected the functions of these subspaces in Table \ref{tab:edge_arc_functions}.

\begin{table}[t]
    \centering
    \begin{tabular}{l|c|c|c|c|c}
         & Parametri- 
         & Cartesian
         & Cartesian  
         & Polar  
         & Polar
         \\
         & zation
         & equation
         & boundary
         & equation
         & boundary
         \\
         & $F(s)$ 
         & $f(x,y)$ & $s(x,y)$ 
         & $\mathring{f}(\rho,\theta)$  & $\mathring{s}(\rho,\theta)$
         \\
         & $G(s)$ 
         & $g(x,y)$ & $t(x,y)$ 
         & $\mathring{g}(\rho,\theta)$  & $\mathring{t}(\rho,\theta)$
         \\
         \hline
         $\Bar{\mathcal{O}}_{0/}$
         & $s+i a s$ 
         & $ax-y$ & $x$
         & $\vert \tan(\theta)\vert - a$ & $\rho \cos(\theta)$
         \\
         $\Bar{\mathcal{O}}_{0\_}$
         & $s$ 
         &  $y$  & $x$
         & $\tan(\theta)$ & $\rho$
         \\
         $\Bar{\mathcal{O}}_{1\vert}$
         & $1 + i s$ 
         &  $x-1$  & $y$
         & $\frac{1}{\rho}-\cos(\theta)$  & $\rho\sin(\theta)$
         \\
         $\breve{\mathcal{O}}_0$
         & $r e^{i s}$ 
         &  $x^2+y^2-r^2$
         & $\atantwo \genfrac(){0pt}{2}{y,}{x}$
         & $\rho - r$ & $\theta$
         \\
         $\breve{\mathcal{O}}_1$
         & $r e^{i s}+1$ & $(x\!-\!1)^2 \!+\! y^2 \!\!-r^2$
         & $\atantwo \genfrac(){0pt}{2}{y,}{x-1}$
         & $\rho^2\!-\!2\rho \cos(\theta) \!+\! 1\!-\! r^2$ 
         & $\atantwo \genfrac(){0pt}{2}{\rho\sin(\theta),}{\rho\cos(\theta)-1}$
         \\
    \end{tabular}
    \caption{Functions for the implicit and parametric equations used in the analysis of edges and arcs. The gradient of $\Bar{\mathcal{O}}_0=\Bar{\mathcal{O}}(\mathbb{C},P_1,P_2)$ is $a=\Re(P_2-P_1)/\Im(P_2-P_1)$. Conversion between the coordinate systems: 
    $x(\rho,\theta)\!=\!\rho\cos(\theta)$, 
    $y(\rho,\theta)\!=\!\rho\sin(\theta)$, 
    $\rho(x,y)\!=\!\sqrt{x^2\!+\!y^2}$, 
    $ \theta(x,y) \!=\! \atantwo(y,x)$.}
    \label{tab:edge_arc_functions}
\end{table}

It is well published that the results of additive combinations of lines and circles are bounded by linear and circular curves \cite{de_berg_computational_2008}. However, Farouki et al. showed that multiplicative combinations of lines and circles result in envelopes \cite{bruce_what_1981} that are not linear or circular functions \cite{farouki_minkowski_2001}. Therefore it is clear that some, but not all arithmetic combinations of polyarc curve segments yield polyarc bounded sets. 


\subsubsection{Mixed envelope evaluation}

\begin{figure*}[t]
\centering
\includegraphics[width=\textwidth,trim={0 0 0 10mm},clip]{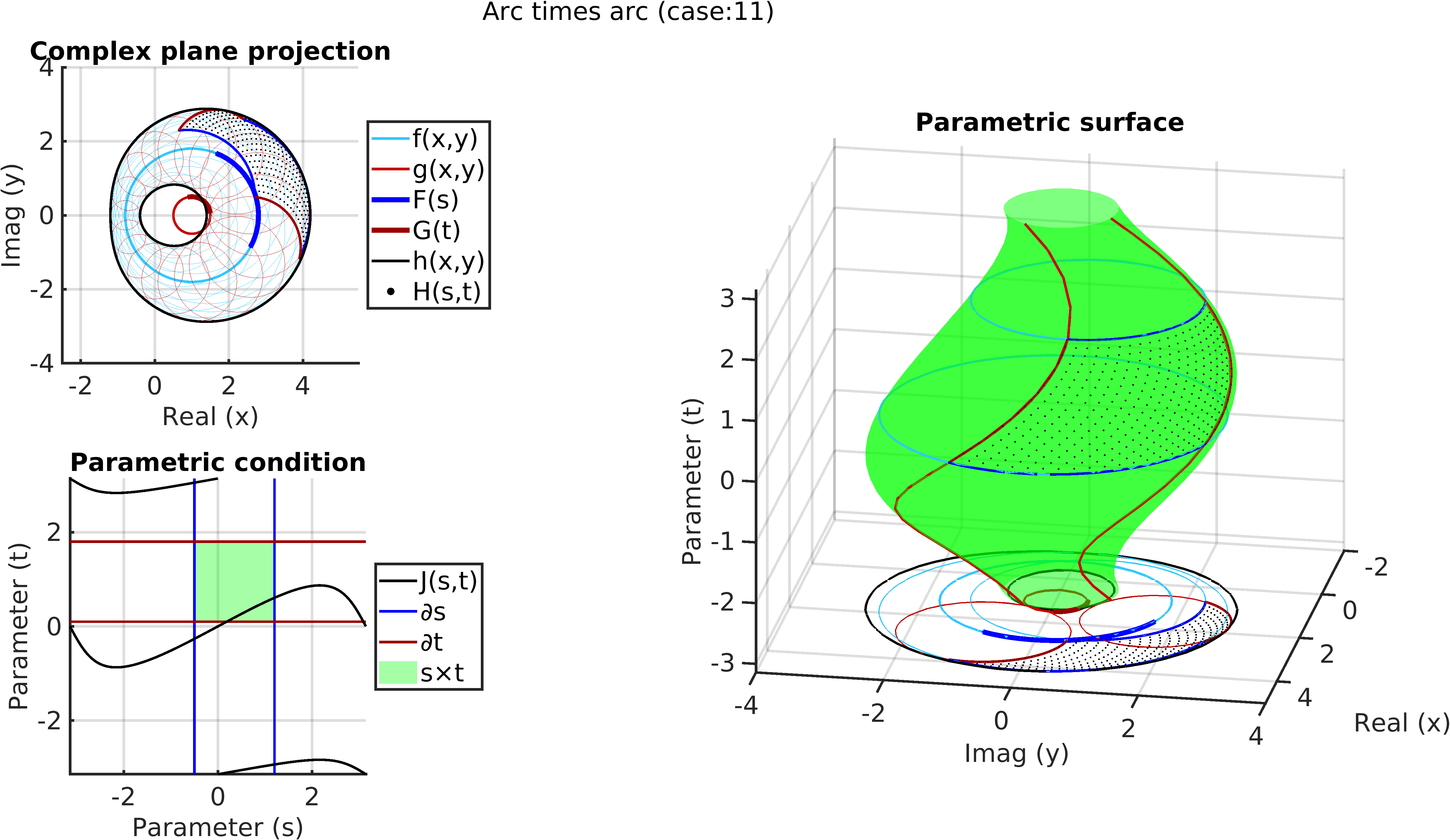}
\caption{Evaluation of the product of two arcs. The plot in the top left  shows the sampled result and the implicit boundary curves besides the operands on the complex plane. The plot on the right shows the sampled result on its implicit surface constrained by planes over the complex plane. The plot in the bottom left shows the parametric condition of the boundary crossing the envelope. Similar figures for all edge and arc arithmetic combinations can be found in the Supplementary Material.}
\label{fig:mixed_methodc}
\end{figure*}

Applying Farouki's method \cite{farouki_minkowski_2001} we can formulate the result of a binary operation ($\bigcirc=\oplus \text{ or }\otimes$) as follows. 
Let us consider the implicit equation of $\boldsymbol{\Gamma}\subset\partial\boldsymbol{A}$ and the parametrization of $\boldsymbol{\Gamma'}\subset\partial\boldsymbol{B}$. That gives the description
\[
\boldsymbol{\Gamma} \bigcirc \boldsymbol{\Gamma'}  
    = \{(x+iy)\bigcirc G(t)\ \big\vert f(x,y)=0, s(x,y)\in \boldsymbol{s}, t\in\boldsymbol{t}\} 
\]
of the result. Similarly to backtracking, if we consider $Z= A+B \in \boldsymbol{\Gamma} \oplus \boldsymbol{\Gamma'}$ (respectively, $Z \in \boldsymbol{\Gamma} \otimes \boldsymbol{\Gamma'}$) then we get $Z-B=Z-G(t)=A \in \boldsymbol{\Gamma}$ for some $t \in \boldsymbol{t}$ (respectively, $Z/G(t) \in \boldsymbol{\Gamma}$). This gives us an equation $\hat{h}=0$ -- in place of the equation $f=0$ -- for the result:

\begin{equation}
    \boldsymbol{\Gamma} \bigcirc \boldsymbol{\Gamma'}  
    = \{x+iy \big\vert \hat{h}(x,y,t)=0, u(x,y,t)\in\boldsymbol{s} \big\vert t\in\boldsymbol{t} \}
    ,
\end{equation}

where 
\begin{equation*}
    \hat{h}(x,y,t) = 
    \begin{cases}
        f\left((x\!+\!iy)-G(t)\right)=f(x\!-\!\Re(G(t)), y\!-\!\Im(G(t))) 
        & \text{if } \bigcirc = \oplus,
        \\
        f\left((x\!+\!iy)/G(t)\right)=\mathring{f}\left(\rho (x,y) / \vert G(t)\vert,\theta(x,y)\!-\!\angle G(t)\right)
        & \text{if } \bigcirc = \otimes,
    \end{cases}
\end{equation*}
where $\rho(x,y)\!=\!\sqrt{x^2+y^2}$, $\theta(x,y)\!=\!\atantwo(y,x)$
and $u(x,y,t)$ can be calculated from $s(x,y)$ and $\mathring{s}(x,y)$ in the same way as $\hat{h}(x,y,t)$ from $f(x,y)$ and $\mathring{f}(x,y)$.

A boundary point in this region is either an envelope point or an extremal point corresponding to one (or both) of the constraints marking the ends of the curve segments. The envelope (or discriminant) of the family defined by $\hat{h}(x,y,t)=0$ is the set of $(x,y)$ points for which there is a $t$ satisfying $h=\partial h / \partial t=0$ (see \cite{bruce_what_1981}, \cite[Chapter 5]{bruce_curves_1984}). Therefore,

\begin{equation} \label{eq: boundary parts}
    \begin{array}{ll}
         \partial(\boldsymbol{\Gamma} \bigcirc \boldsymbol{\Gamma'}) \subset 
        &\{x+iy \big\vert \hat{h}(x,y,t)=0, \partial \hat{h}/\partial t = 0,
        u(x,y,t)\in\boldsymbol{s} \big\vert t\in\boldsymbol{t}\}
        \\
        & \cup \{x+iy \big\vert \hat{h}(x,y,t)=0,u(x,y,t)\in \boldsymbol{s}\vert t\in 
        \{\underline{t},\overline{t}\} 
        \\
        & 
        \cup
        \{x+iy \big\vert \hat{h}(x,y,t)=0,u(x,y,t)\in \{\underline{s},\overline{s}\} \vert t\in\boldsymbol{t}\}.
    \end{array}
\end{equation}

Since the constraint components are edges and arcs translated or multiplied by scalars, the only possibly non-polyarcular segment of the boundary is the envelope component. Therefore the question of this analysis is whether $\{\partial \hat{h}/\partial t=0\}$ is a polyarcular curve, and whether the $\{u(x,y,t)\in\boldsymbol{s},t\in\boldsymbol{t}\}$ condition is fulfilled. A demonstrative example can be found in Figure \ref{fig:mixed_methodc}, and more are availabile in the Supplementary Material.


\subsubsection{Parametric envelope evaluation}

Another method of finding the envelope is by using the parametrizations of both operands.
We are aiming for a parametrization of the result of the operation.
First, let us look at the two parametrizations as map from the two-dimensional parameter space to the Cartesian product $\mathbb{C} \times_c \mathbb{C}=\mathbb{C}^2 \cong \mathbb{R}^4$.

Then apply the map $\varphi_\bigcirc :\mathbb{C}^2 \to \mathbb{C} \cong \mathbb{R}^2$ corresponding to the complex addition or multiplication operation:
\begin{equation} \label{eq:operation map}
\begin{array}{r l}
    \varphi_\oplus (x,y,u,v) & = (x+u,y+v), \\
    \varphi_\otimes (x,y,u,v) & = (xu-yv,xv+yu),
\end{array}
\end{equation}
where $x+iy$ and $u+iv$ are the two complex coordinates of $\mathbb C^2$.

Together, they form a composition
\[
H:\mathbb{R}^2 \overset{F \times_c G}{\longrightarrow} \mathbb{R}^4 \overset{\varphi_\bigcirc}{\longrightarrow} \mathbb{R}^2,
\]
which is the parametrization of the result of the operation.

The points of the envelope are all critical points of the real map $\varphi_\bigcirc$ restricted to the intermediate image $(F \times_c G)(\mathbb{R}^2) \subset \mathbb{R}^4$, or in other words critical points of the parametrization $H(s,t)$. This means that the differential $dH=d (\varphi_\bigcirc \circ (F \times_c G))$ of the composition is not of full rank. This situation is depicted in Fig. \ref{fig:parametric_method}.

\begin{figure*}[t]
\centering
\includegraphics[width=.9\textwidth]{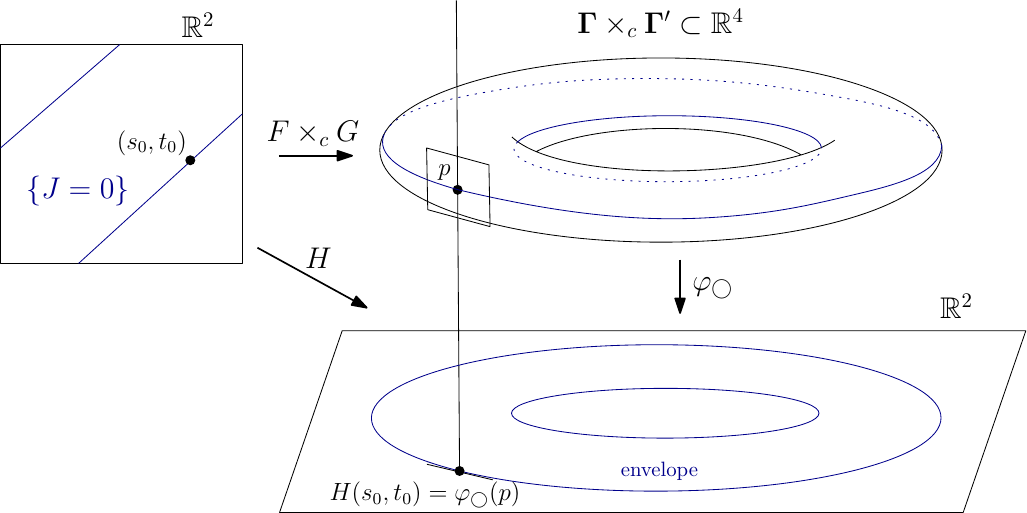}
\caption{Schematic summary of the parametrization of the envelope}
\label{fig:parametric_method}
\end{figure*}

Hence, a point $p=F(s_0) \bigcirc G(t_0)$ can only be an envelope point if the Jacobian determinant $J(s,t)=\det (\mathrm{Jac}_{H}(s_0,t_0))$ of the parametrization $H:\mathbb{R}^2 \to \mathbb{R}^2$ vanish at $(s_0,t_0)$. This gives another way to describe the superset of the boundary with three parts -- as in \eqref{eq: boundary parts}:
\begin{equation}
\partial (\boldsymbol{\Gamma} \bigcirc \boldsymbol{\Gamma'}) \subset \big\{ F(s) \bigcirc G(t) \big\vert J(s,t)=0 \text{ or } s \in \{\underline{s},\overline{s}\} \text{ or } t \in \{\underline{t},\overline{t}\} \big\}.
\end{equation}

In computations we will use the following description. In the additive case
\[
H_\oplus(s,t)=\varphi_\oplus (F(s),G(t))=\big(\Re(F)(s)+\Re(G)(t), \Im(F)(s)+\Im(G)(t)\big), 
\]
\[
J_\oplus (s,t)
=\left\vert \mathrm{Jac}_{H_\oplus}(s,t) \right\vert
=\left\vert \begin{array}{cc}
    \Re(F)'(s) & \Re(G)'(t) \\
    \Im(F)'(s) & \Im(G)'(t)
\end{array} \right\vert.
\]
In the multiplicative case
\[
H_\otimes (s,t)
=\varphi_\otimes (F,G)
=\big(\Re(F)\Re(G)-\Im(F)\Im(G),\Re(F)\Im(G)+\Im(F)\Re(G) \big),
\]
\[
J_\otimes (s,t)
=\left\vert \mathrm{Jac}_{H_\otimes} \right\vert
=\left\vert \begin{array}{cc}
    \Re(F)'\Re(G)-\Im(F)'\Im(G) & \Re(F)\Re(G)'-\Im(F)\Im(G)' \\
    \Re(F)'\Im(G)+\Im(F)'\Re(G) & \Re(F)\Im(G)'+\Im(F)\Re(G)'
\end{array} \right\vert.
\]

In fact, this method of finding the parametrization of the envelope was already mentioned -- in the latter form -- and used in \cite{farouki_boundary_2005}.

\subsubsection{Implicit envelope evaluation} \label{sub:implicit}

A third way of describing the envelope is using the implicit equations of both operands. We aim for an implicit equation of the envelope.

Let the two operands -- without boundary conditions, so far -- be $\boldsymbol{\Gamma}=\{f(x,y)=0\}, \boldsymbol{\Gamma'}=\{g(x,y)=0\} \subset \mathbb{C}$. Their Cartesian product is $\boldsymbol{\Gamma} \times_c \boldsymbol{\Gamma'}=\{f(x,y)=0, g(u,v)=0\} \subset \mathbb{C}^2$. Consider the map $\varphi_\bigcirc : \mathbb{R}^4 \to \mathbb{R}^2$ corresponding to Minkowski addition and multiplication, respectively, as defined in \eqref{eq:operation map}.

Similarly to the previous method, we find the envelope as the set $\widetilde{\boldsymbol{E}}$ of critical points of the map $\varphi_\bigcirc$ restricted to $\boldsymbol{\Gamma}\times_c \boldsymbol{\Gamma'}$. These are the points where the gradients of $f(x,y)$, $g(u,v)$, and that of the two components of $\varphi_\bigcirc$ are not linearly independent. Thus, at these points the Jacobian determinant vanish:
\[
\widetilde{h}(x,y,u,v)=
\det \mathrm{Jac}\big(f,g,\Re(\varphi_\bigcirc),\Im(\varphi_\bigcirc)\big)=0.
\]
The set $\widetilde{\boldsymbol{E}} \subset \mathbb{R}^4$ of critical points is described by the three equations $f=g=\widetilde{h}=0$ and is a real algebraic variety of dimension at most one. \cite{milnor_morse_1963}

The actual envelope is the image $E=\varphi_\bigcirc(\widetilde{E})$ in $\mathbb{R}^2 \cong \mathbb{C}$. In general, the image of a variety on an algebraic map may not be a variety; however, in case of a proper map, it is (see \cite[5.2]{shafarevich_basic_2013}). Our non-trivial cases involve only proper maps. Again, we summarize the spaces involved in Fig. \ref{fig:implicit_method}.

\begin{figure*}[t]
\centering
\includegraphics[width=.65\textwidth]{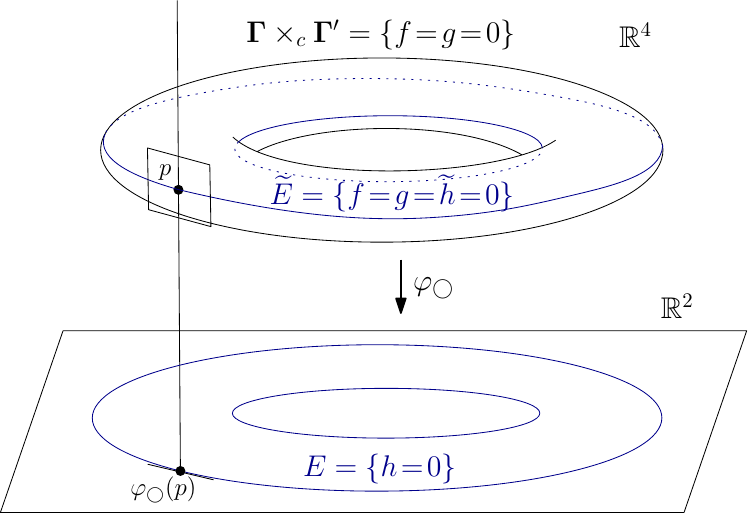}
\caption{Schematic summary of the implicit envelope evaluation}
\label{fig:implicit_method}
\end{figure*}

When we look for the defining equations of the image $E \subset \mathbb{R}^2$, we want to find (the ideal of) those functions $\psi:\mathbb{R}^2 \to \mathbb{R}$ that, when composed with $\varphi_\bigcirc:\mathbb{R}^2 \to \mathbb{R}^2$, the composition $\psi \circ \varphi_\bigcirc$ vanishes on $\widetilde{E}$. (When the map $\varphi_\bigcirc$ is not proper, this ideal defines a variety that contains $E$ as a Zariski open subset.)
To compute the generators of the ideal defining $E$, we need Gröbner bases (\cite[Chapter~1]{adams_introduction_1994}). More precisely, we find the reduced Gröbner basis of the ideal 
\[
I=(f,g,\widetilde{h},\Re(\varphi_\bigcirc)-X,\Im(\varphi_\bigcirc)-Y)
\]
(where $X,Y$ are the coordinate functions of the target $\mathbb{R}^2$) with respect to a suitable elimination order. Lastly, we take those basis elements that only involve $X$ and $Y$. (For the precise details, see \cite[2.4]{adams_introduction_1994}.)
In most of our cases, $\widetilde{E}$ and $E$ are one-dimensional; therefore, the algorithm results in a single polynomial function that will call $h(X,Y)$. We used SageMath for the computations, the code can be found in the Supplementary Section S2.

For each additive and multiplicative combination, the step-by-step computations of the implicit equation $h(x,y)=0$ of the envelope and that of the condition $J(s,t)=0$ describing for what parameter values the envelope is included in the boundary can be found in the Supplementary Materials, together with illustrations.

\subsection{Properties of polyarcular intervals} \label{sub:properties_of_polyarcular_intervals}

\subsubsection{Unary operations} \label{sub:polyarcular_unary_operations}

Let us analyze the result of the two unary operations, negative and reciprocal, using the Cartesian and polar implicit functions. Where necessary, we map between the two coordinate systems. The corresponding computations can be found in the Supplementary Section S1.

\begin{lemma} \label{lem:curve_negative}
    The negative of an edge is also an edge. The negative of an arc is also an arc.
\end{lemma}
\begin{proof}
Taking the negative is simply taking the reflection across the origin which takes lines into lines, and arcs into arcs.
\end{proof}

\begin{lemma} \label{thm:linear_curve_reciprocal} \label{thm:circular_curve_reciprocal}
    The reciprocal of an edge is an edge if the containing line is zero-crossing; otherwise it is an arc. The reciprocal of an arc is also an arc.
\end{lemma}

\begin{proof}
The reciprocal is known to be equivalent to the combination of the geometric inversion with respect to the complex $0$ and the reflection through the real axis. The former takes circles and lines into circles -- except for the lines through the center that are fixed -- and the latter is an isometry.
\end{proof}

\begin{proposition}
    The negative of a polyarcular interval is a polyarcular interval.
    \begin{equation*}
    \boldsymbol{A}\in\mathcal{A}(\mathbb{C})
    \implies
    -\boldsymbol{A} \in \mathcal{A}(\mathbb{C}) 
\end{equation*}
\end{proposition}
\begin{proof}
    Negation is a closed unary operation on complex intervals, and a complex interval is polyarcular if it is bounded by a polyarc curve. According to Proposition \ref{prop:Minkowski_boundary} the negative of a complex interval is bounded by the negative of the interval boundary, and in Lemma \ref{lem:curve_negative} we show that the negative of the polyarc curve segments are also polyarc curve segments.
\end{proof}

\begin{proposition}
    The reciprocal of a polyarcular interval is a polyarcular interval if the operand does not contain the complex zero.
    \begin{equation*}
    \boldsymbol{A}\in\mathcal{A}(\mathbb{C}), 0 \notin \boldsymbol{A}
    \implies
    \boldsymbol{A}^{-1} \in \mathcal{A}(\mathbb{C}) 
\end{equation*}
\end{proposition}
\begin{proof}
    Reciprocal is a partial unary operation on complex intervals, and a complex interval is polyarcular if it is bounded by a polyarc curve. According to Proposition \ref{prop:Minkowski_boundary} the reciprocal of a complex interval is bounded by the reciprocal of the interval boundary, and in Lemma \ref{thm:linear_curve_reciprocal} we showed that the reciprocal of the polyarc curve segments are also polyarc curve segments.
\end{proof}

\begin{table}[t]
    \centering
    \begin{tabular}{c|cc}
        $+$
         & $\Bar{\mathcal{O}}$
         & $\breve{\mathcal{O}}$
        \\
        \hline
        \\
        $\Bar{\mathcal{O}}$
        & $\emptyset$
        & $\Bar{\mathcal{O}}$
        \\
        $\breve{\mathcal{O}}$
        & 
        & $\breve{\mathcal{O}}$
\end{tabular}
\hspace{0.5 cm}
    \begin{tabular}{c|cc}
        $\times$
         & $\Bar{\mathcal{O}}$
         & $\breve{\mathcal{O}}$
        \\
        \hline
        \\
        $\Bar{\mathcal{O}}$
        & $ \mathcal{O}^2
            (\emptyset,
            \Bar{\mathcal{O}})$
        & $ \mathcal{O}^2
            (\Bar{\mathcal{O}},
            \Breve{\mathcal{O}},
            \mathrm{pt})$
        \\
        $\breve{\mathcal{O}}$
        & 
        & $ \mathcal{O}^4
            (\breve{\mathcal{O}})$
\end{tabular}
\caption{Envelope subspace of polyarc segment binary combinations. $\mathrm{pt}$ indicates that the envelope is a point, $\Bar{\mathcal{O}}$ indicates linear curves (edges), $\Breve{\mathcal{O}}$ circular curves (arcs), $\mathcal{O}^2$ quadratic, and $\mathcal{O}^4$ quartic polynomial curves; $\emptyset$ indicates that there is no envelope. (Subspace symbols in brackets indicate special cases, when one or both operands are on a zero crossing line or a zero-centered circle.)}
\label{tab:envelopes}
\end{table}

\subsubsection{Binary operations} \label{sub:polyarcular_binary_operations}

\begin{lemma} \label{thm:edge_plus_edge}
    The sum of two edges is a complex interval bounded by edges. 
\end{lemma}
\begin{proof}
The sum of two lines has no envelope unless the two lines are identical, in which case the envelope is the line itself. Consequently, the sum of two non-parallel edges is bounded only by the operand edges translated by the endpoints of the other operand, while the sum of two parallel edges is a single edge that lies in the line containing the operands. See the derivation in the Supplementary Section S2.1.1.
\end{proof}

\begin{lemma} \label{thm:arc_plus_edge}
    The sum of an arc and an edge is a complex interval bounded by edges and arcs. 
\end{lemma}
\begin{proof}
The sum of a circle and a line has an envelope consisting of two lines parallel to the edge. The boundary of the result contains sections of the envelope when the argument of the arc is normal to the edge. Consequently, the result boundary consists of the operand edge translated by the arc endpoints, the operand arc translated by the edge endpoints, and edges formed by sections of the envelope. See the derivation in the Supplementary Section S2.1.2.
\end{proof}

\begin{lemma} \label{thm:arc_plus_arc}
    The sum of two arcs is a complex interval bounded by arcs. 
\end{lemma}
\begin{proof}
The sum of two circles has an envelope consisting of two circles with radii equal to the sum and difference of the radii of the operand circles. The result boundary contains sections of the envelope if the argument intervals of the two arcs overlap. Consequently, the result boundary consists of the operand arcs translated by the endpoints of the other operand and the arcs formed by sections of the envelope. See the derivation in the Supplementary Section S2.1.3.
\end{proof}

\begin{lemma} \label{thm:edge_times_edge}
    The product of two edges is a complex interval bounded by edges if at least one operand is from a zero-crossing line or if the intersection of their normalized parameter intervals is empty. 
    Otherwise, the boundary of the result has segments of a parabola, a quadratic curve, that are generically neither edges nor arcs.
\end{lemma}
\begin{proof}
The product of two lines has an envelope consisting of a parabolic curve when none of them crosses the origin. The result boundary contains a segment of the envelope when their normalized parameter intervals overlap. If only one line crosses the origin, then there is no envelope, while if both lines cross the origin, then the product is a single line, which is also the envelope. Consequently, the result boundary consist of the operand edges translated by the endpoints of the other operand, and sections of the envelope. See the derivation in the Supplementary Section S2.2.1.
\end{proof}

\begin{lemma} \label{thm:arc_times_edge}
    The product of an edge and an arc is a complex interval bounded by edges and arcs if the edge is from a zero crossing line, or the arc is from a zero-centered circle, or the circle radius is exactly 1, or the following condition is not fulfilled for any of the parameter value pairs of their normalized parameter intervals:
    \[
    r + \cos(s) + \sin(s)/t = 0
    \]
    Otherwise, the boundary of the result has segments of a hyperbola or ellipse, a quadratic curve, that are generically neither edges nor arcs.
    
\end{lemma}
\begin{proof}
The envelope of a non-zero-centered circle and a non-zero-crossing line is a hyperbola if the circle radius is less than 1, and an ellipse if its radius is more than 1. The result boundary contains a segment of the envelope when the above condition holds for the arguments.
If the circle radius equals 1, the envelope is the real line, except the real 0 and 2 points, and the result is the union of two concave, arc bounded regions touching at the real 0 and 2 points, which technically can be considered a complex interval.
If the circle is zero-centered, but the line is non-zero-crossing, the envelope is a circle. If the line is zero-crossing but the circle is non-zero centered, the envelope is two lines. If both operands are zero-centered, the envelope is the point of origin.
Consequently, the result boundary consist of the operand arc and edge scale-rotated by the endpoints of the other operand, and sections of the envelope.
See the derivation in the Supplementary Section S2.2.2.
\end{proof}

\begin{lemma} \label{thm:arc_times_arc}
    The product of two arcs is a complex interval bounded by arcs if at least one operand is from a zero-centered circle or if the following condition is not fulfilled for any of the parameter value pairs from their normalized parameter intervals:
    \[
    \sin(s-t)=r_1 \sin(t)-r_2 \sin(s),
    \]
    where $r_1,r_2$ denote the radii of the normalized circles. \\
    Otherwise, the boundary of the result has segments of a Cartesian oval, a quartic curve, that are generically neither edges nor arcs.
\end{lemma}
\begin{proof}
The product of two non-zero centered circles has a quartic envelope, a Cartesian oval. The result boundary contains sections of it if the above condition holds for the arguments.
If only one arc is zero-centered, then the envelope consists of two rescaled copies of the zero-centered circle, and the result boundary includes envelope segments if the normalized argument of the nonzero centered circle avoids the values $0$ and $\pi$.
If both arcs are zero-centered, their product is a single arc that also belongs to the boundary (similarly to the case of zero-crossing lines). 
Consequently, the result boundary consist of the operand arcs scale-rotated by the endpoints of the other operand, and sections of the envelope.
See the derivation in the Supplementary Section S2.2.3.
\end{proof}

\begin{proposition}
    The sum of two polyarcular intervals is a polyarcular interval.
\begin{equation*}
    \boldsymbol{A},\boldsymbol{B}\in\mathcal{A}(\mathbb{C})
    \implies
    (\boldsymbol{A} \oplus \boldsymbol{B}) \in \mathcal{A}(\mathbb{C}) 
\end{equation*}
\end{proposition}

\begin{proof}
    Addition is a closed binary operation on complex intervals, and a complex interval is polyarcular if it is bounded by a polyarc curve. According to Proposition \ref{prop:Minkowski_boundary}, the sum of complex intervals is bounded by the sum of the interval boundaries. In equation \eqref{prop:Minkowski_boundary} we show that the sum of polyarc curves is bounded by a subset of the sums of the curve segments. In lemmas \ref{thm:edge_plus_edge}, \ref{thm:arc_plus_edge} and \ref{thm:arc_plus_arc} we show that the sum of polyarcular curve segments is always bounded by edges and arcs. 
\end{proof}

\begin{proposition}
    The product of two polyarcular intervals is a polyarcular interval if in each combination of curve segments that contributes to the result boundary at least one of the operands is from a zero-crossing line or zero-centered circle, or none of the results of the combinations of curve segments touch the envelope (see conditions in lemmas \ref{thm:edge_times_edge}, \ref{thm:arc_times_edge} and \ref{thm:arc_times_arc}). Otherwise, it is a complex interval.
\begin{equation*}
    \boldsymbol{A},\boldsymbol{B}\in\mathcal{A}(\mathbb{C})
    \implies
    (\boldsymbol{A} \otimes \boldsymbol{B}) \in
    \begin{cases}
        \mathcal{A}(\mathbb{C}) & \text{conditions above}
        \\
        \mathcal{I}(\mathbb{C}) & \text{otherwise}   
    \end{cases}
\end{equation*}
\end{proposition}

\begin{proof}
    Multiplication is a closed binary operation on complex intervals, and a complex interval is polyarcular if it is bounded by a polyarc curve. According to Proposition \ref{prop:Minkowski_boundary} the product of complex intervals is bounded by the product of the interval boundaries. In equation \eqref{prop:Minkowski_boundary} we show that the product of polyarc curves is bounded by a subset of the products of the curve segments. In lemmas \ref{thm:edge_times_edge}, \ref{thm:arc_times_edge} and \ref{thm:arc_times_arc} we show that in case one of the curve segments in the combination is zero-crossing or zero-centered, then the product is always bounded by edges and arcs. In the same lemma we also show that the product envelope of non-zero-crossing/centered curve segments is neither linear nor circular; therefore, the product is bounded by edges and arcs only if the envelope is not part of the boundary.
\end{proof}

\subsubsection{Set operations} \label{sub:polyarcular_set_operations}

\begin{proposition}
    The union and intersection of polyarcular intervals are polyarcular intervals if the intersection of the operands is not empty.
    \begin{equation*}
    \boldsymbol{A},\boldsymbol{B} \in \mathcal{A}(\mathbb{C}), 
    \boldsymbol{A} \cap \boldsymbol{B} \neq \emptyset
    \implies
    \begin{cases}
        (\boldsymbol{A} \cup \boldsymbol{B}) \in \mathcal{A}(\mathbb{C})
        \\
        (\boldsymbol{A} \cap \boldsymbol{B}) \in \mathcal{A}(\mathbb{C})
    \end{cases}
\end{equation*}
\end{proposition}
\begin{proof}
   Proposition \ref{prop:Minkowski_boundary} shows that the boundary of the union and the intersection of complex intervals consist of the boundary of the operands. The boundaries of polyarcular intervals consist of edges and arcs by definition, and any continous closed curve constructed from intersection polyarc curves will also consist of edges and arcs. 
\end{proof}


\section{Computational properties} \label{sec:computational_properties}

In this section, we consider the data representation and computation of complex intervals. We define a data type for each complex interval representation and determine their data storage requirements. To measure the goodness of their representation capability, we introduce the tightness measure. We then define the type-casting operation between interval types and show how it can cause a loss of tightness through arithmetic and set operations (Table \ref{tab:interval_type_properties}). Finally, we show two utility processes, one for extracting the simple boundary when an operation results a self-intersecting boundary, and another for backtracking the subsets of the operands of an operation that map to a point in the result interval.

\subsection{Interval types}

Interval types are finite data representations of intervals, meaning that they can be identified with particular elements of the corresponding subspace using finite sets of parameter values. A type is a more restrictive class than a subspace as it determines how the complex interval is stored and manipulated in the computational environment, and therefore an interval instance can have only one type assigned to it. It is possible to represent an interval using any of the types, but if the interval is not a member of the corresponding subspace, then we assume the smallest bounding interval of the type, in which case the representation will not be tight (see the definition of tightness in Section \ref{sub:Tightness}). As a shorthand format, we indicate the type of a variable in its upper index. First, we define some utility data types.

\begin{definition}
    The real type represents real numbers with the double precision floating-point data type variable (float).
\end{definition}

\begin{definition}
    The complex type represents complex numbers as two real-type variables (2 floats).
\end{definition}

\begin{definition} \label{def:real_interval_type}
The real interval type represents a bounded real set $\boldsymbol{a} \subset \mathbb{R}$ by storing its bounds as real type variables (2 floats in total).

\begin{equation*}
    \label{eq:real_interval_def}
    \begin{array}{l}
        \boldsymbol{a}^\mathcal{I} =
        \mathcal{I}(\boldsymbol{a}) := 
        \mathcal{I}(\mathbb{R} \big\vert 
        \underline{a},\overline{a}) = 
        [\underline{a},\overline{a}]
        \quad \underline{a}=\inf(\boldsymbol{a}),
        \overline{a}=\sup(\boldsymbol{a})
    \end{array}
\end{equation*}
\end{definition}

\begin{definition} \label{def:linear_curve_segment_type}
The edge type represents edges by storing their endpoints as two complex type variables (4 floats in total).

\begin{equation*}
\label{eq:linear_curve_segment_type}
   \boldsymbol{\Gamma} \in \Bar{\mathcal{O}}(\mathbb{C}) \implies
   \boldsymbol{\Gamma}^{\Bar{\mathcal{O}}} 
   := \{\Bar{\Gamma}(P_1,P_2)(t) \big\vert t\in[0,1]\}
\end{equation*}
\end{definition}

\begin{definition}\label{def:circular_curve_segment_type}
The arc type represents arcs by storing their center point as a complex type, radius as a real type and argument as a real interval type variable (5 floats in total).

\begin{equation*}
\label{eq:circular_curve_segment_type}
   \boldsymbol{\Gamma} \in \breve{\mathcal{O}}(\mathbb{C})
   \implies
   \boldsymbol{\Gamma}^{\breve{\mathcal{O}}} 
   :=   \{\Breve{\Gamma}(O,r,\boldsymbol{\varphi})(t) \big\vert t\in[0,1]\}
\end{equation*}
\end{definition}

A complex interval $\boldsymbol{A} \in \mathcal{I}(\mathbb{C})$ can be represented by various complex interval types in the following way.

\begin{definition} \label{def:rectangular_interval_type}
The rectangular interval type represents complex intervals by storing their bounds along the real and imaginary axes as real intervals (4 floats in total).

\begin{equation*}
    \label{eq:rectangular_interval_type}
        \boldsymbol{A}^\mathcal{R} =  
        \mathcal{R}(\boldsymbol{A}) := 
        \mathcal{R}(\mathbb{C} \big\vert \boldsymbol{a},\boldsymbol{b}) = 
        \boldsymbol{a} + i \boldsymbol{b}
        , \quad
        \boldsymbol{a} = \Re(\boldsymbol{A})
        ,
        \boldsymbol{b} = \Im(\boldsymbol{A})
\end{equation*}
\end{definition}

\begin{definition} \label{def:polar_interval_type}
The polar interval type represents complex intervals by storing their bounds along the radial and angular axes as real intervals (4 floats in total).

\begin{equation*}
    \label{eq:polar_interval_type}
        \boldsymbol{A}^\mathcal{P} =
        \mathcal{P}(\boldsymbol{A}) := 
        \mathcal{P}(\mathbb{C} \vert \boldsymbol{r},\boldsymbol{\varphi}) = 
        \boldsymbol{r} e^{i\boldsymbol{\varphi}}
        ,\quad
        \boldsymbol{r} = \abs{\boldsymbol{A}}
        ,\,
        \boldsymbol{\varphi} = \angle \boldsymbol{A}
\end{equation*}
\end{definition}

\begin{definition} \label{def:circular_interval_type}
The circular interval type represents complex intervals by storing the center point and radius of the smallest bounding circle as a complex type and a real type variable respectively (3 floats in total).

\begin{equation*}
\label{eq:circular_interval_def}
    \boldsymbol{A}^\mathcal{C} =  
    \mathcal{C}(\boldsymbol{A}) := 
    \mathcal{C}(\mathbb{C} \big\vert O, r ) = 
    O + [0,r] e^{[-\pi,\pi]}
\end{equation*}
\end{definition}

\begin{definition} \label{def:polygonal_interval_type}
The polygonal interval type represents complex intervals by storing the ordered set of vertices of the smallest bounding polygon of a given vertex count $\mathsf{N}$ as complex-type variables ($2 \mathsf{N}$ floats in total).

\begin{equation*}
\label{eq:Polygonal_interval_type}
\begin{array}{l}
      \boldsymbol{A}^\mathcal{G} = 
      \mathcal{G}(\boldsymbol{A}) := 
      \mathcal{G}(\mathbb{C} \big\vert 
      \{P_\mathsf{n} \big\vert 
      \mathsf{n} \!\in\! \{1..\mathsf{N}\}\},
      \quad
      \partial\boldsymbol{A}^\mathcal{G} = 
      \{ \Bar{\boldsymbol{\Gamma}}_\mathsf{n}
      \big\vert \mathsf{n}\in\{1..\mathsf{N}\} \}
      ,
\end{array}
\end{equation*}
where $P_\mathsf{n}$ is the $\mathsf{n}^\mathrm{th}$ vertex, and 
\[\Bar{\boldsymbol{\Gamma}}_\mathsf{n} = 
\{\Bar{\Gamma}(P_\mathsf{n},P_\mathsf{n+1})(t)\big\vert t \in [0,1]\}\] 
is the $\mathsf{n}^\mathrm{th}$ implicit edge between the corresponding vertices (see also Definition \ref{def:polygon} and Figure \ref{fig:polyarc_curve}).
\end{definition}

\begin{definition} \label{def:polyarcular_interval_type}
The polyarcular interval type represents complex intervals by storing the ordered set of arcs of the smallest bounding polyarc of a given arc count $\mathsf{N}$ as arc-type variables (total $5 \mathsf{N}$ floats).

\begin{equation*}
\label{eq:Polyarcular_interval_type}
\begin{array}{c}
      \boldsymbol{A}^\mathcal{A} = 
      \mathcal{A}(\boldsymbol{A}) := 
      \mathcal{A}(\mathbb{C} \big\vert 
      \{\Breve{\boldsymbol{\Gamma}}_\mathsf{n} \in \breve{\mathcal{O}}(\mathbb{C}) \big\vert \mathsf{n} \!\in\! \{1..\mathsf{N}\}\}, 
      \quad
      \partial\boldsymbol{A}^\mathcal{A} = 
      \{\Breve{\boldsymbol{\Gamma}}_\mathsf{n}
      \cup
      \Bar{\boldsymbol{\Gamma}}_\mathsf{n}
      \big\vert
      \mathsf{n}\in\{1..\mathsf{N}\}
      \}
      ,
\end{array}
\end{equation*}
where 
\[\Breve{\boldsymbol{\Gamma}}_\mathsf{n} = 
\{\breve{\Gamma}(O_\mathsf{n},r_\mathsf{n},\boldsymbol{\varphi}_\mathsf{n})(t)\big\vert t \in [0,1]\}\] 
is the $\mathsf{n}^\mathrm{th}$ arc, 
\[\Bar{\boldsymbol{\Gamma}}_\mathsf{n} = 
\{\Bar{\Gamma}(P_\mathsf{2n-1},P_\mathsf{2n})(t)\big\vert t \in [0,1]\}\] 
is the $\mathsf{n}^\mathrm{th}$ implicit edge between the corresponding arcs,
$P_\mathsf{2n-1}$ and $P_\mathsf{2n}$ are implicit vertices
(see also Definition \ref{def:polyarc} and Figure \ref{fig:polyarc_curve}).
\end{definition}

\begin{remark}
Polyarcular curves consist of arcs defined by its data set and the implicit edges connecting the endpoints of adjacent arcs. It is possible to suppress circular segments by setting $r_\mathsf{n} =0$, while the linear segment can be suppressed by making sure that $P_\mathsf{2n-1}=P_\mathsf{2n}$. Concave arcs can be created using negative radius values. 
\end{remark}

Figure \ref{fig:polyarc_curve} gives a demonstrative example of a complex interval represented by the mentioned complex interval types.


\subsection{Tightness}\label{sub:Tightness}

The size of a real interval $\boldsymbol{a} \in \mathcal{I}(\mathbb{R})$ is its length:
\begin{equation*}
    \mu(\boldsymbol{a}) = \overline{a} - \underline{a}.
\end{equation*}

Let us measure complex intervals with the standard area (or Lebesgue measure). According to Green's theorem the following holds (\cite[Chapter~16]{stewart_calculus_1999} \cite[X,1]{lang_calculus_2012}). 

\begin{proposition}
The size of a complex interval $\boldsymbol{A} \in \mathcal{I}(\mathbb{C})$ bounded by a simple, closed, piecewise smooth curve equals the following closed line integral
\begin{equation}
    \mu(\boldsymbol{A}) = 
    \frac{1}{2} \oint_{\partial \boldsymbol{A}} x \, d y - y \, d x = 
    \frac{1}{2i} \oint_{\partial \boldsymbol{A}} z^* \, dz 
\end{equation}
where $z = x+iy$ and $z^*=x-iy$.
\end{proposition}

This leads to the equation of the polyarcular interval size, where polygonal and primitive intervals represent special cases.

\begin{equation}\label{eq:Polyarcular_size}
\begin{split}
    \mu(\boldsymbol{A}^\mathcal{A}) =
    & \sum_n \frac{1}{2i} \oint_{\breve{\boldsymbol{\Gamma}}_{n}} z^* \, dz + 
    \sum_n \frac{1}{2i} \oint_{\Bar{\boldsymbol{\Gamma}}_{n}} z^* \, dz 
    \\
    = & \sum_n \frac{(\overline{\varphi}_\mathsf{n}-\underline{\varphi}_\mathsf{n}) r_\mathsf{n}^2}{2} + 
            i\frac{O_\mathsf{n}^* r}{2}(e^{i\underline{\varphi}_\mathsf{n}} - e^{i\overline{\varphi}_\mathsf{n}})  
    \\            
    & + \sum_n  \frac{1}{4i}(\vert P_{\mathsf2n} \vert ^2 - \vert P_\mathsf{2n-1} \vert ^2 + 
            2i(P_\mathsf{2n-1}^\Re P_\mathsf{2n}^\Im - P_\mathsf{2n-1}^\Im P_\mathsf{2n}^\Re)
            ,
\end{split}
\end{equation}
where $\Breve{\boldsymbol{\Gamma}}_\mathsf{n},\Bar{\boldsymbol{\Gamma}}_\mathsf{n}\in\partial\boldsymbol{A}$, $\mathsf{n}\in(1..\mathsf{N})$. (See also Definition \ref{def:polyarcular_interval_type}.)

Then we can define the tightness of a representation in the following way.

\begin{definition}
The tightness of a complex interval representation $\boldsymbol{A} \in \mathcal{I}(\mathbb{C})$ is the ratio of its original size and its represented size.
\begin{equation*}
    \tau (\boldsymbol{A}^\mathcal{X}) = \frac{\mu(\boldsymbol{A})}{\mu(\boldsymbol{A}^\mathcal{X})} \in [0,1], \, 
\end{equation*}
where $\mathcal{X}$ is a placeholder for one of the complex interval types.
\end{definition}

\begin{example}
The tightness of a circular interval represented by the rectangular type is
\begin{equation*}
    \boldsymbol{A} \in \mathcal{C}(\mathbb{C}) \implies
    \tau(\boldsymbol{A}^\mathcal{R}) = 
    \frac{\mu(\boldsymbol{A})}{\mu(\boldsymbol{A}^\mathcal{R})} = 
    \frac{r^2 \pi}{(2 r)^2} = 
    \frac{\pi}{4},
\end{equation*}
where $r$ is the radius of the circular interval.
\end{example}

If an interval belongs to the data type's corresponding subspace, then that representation will be tight ($\tau=1$), otherwise it will be loose ($\tau<1$). Figure \ref{fig:complex_subspaces} offers an intuitive demonstration as a Venn diagram where each set is shaped accordingly. The representation and operation tightness of our interval types are listed in Table. \ref{tab:interval_type_properties}.

\begin{remark}
    In certain applications -- such as algorithms for the approximation of boundary curves with polygonal or polyarcular curves -- the Hausdorff distance (Definition \ref{def:hausdorff}) may be preferred over the tightness metric, because it can measure the representation error of a boundary segment, while tightness can only be applied to an entire interval. 
\end{remark}


\subsection{Type casting}

Changing data type can be a useful and sometimes necessary step when handling complex intervals. Binary operations, for example, are typically only defined between operands of the same type. Although it is possible and can be practical in certain cases to define binary operations between different types (e.g. a fast algorithm for determining the smallest polar interval enclosing the product of a polar and a circular interval), we don't discuss these in this paper. Therefore, type casting is necessary when two complex intervals of different types are to be combined, when using another interval type in an operation is preferred, or when the result is not in the operand subspace.

\begin{definition}
Type casting is a unary operation that transforms a finite representation of an interval into another finite representation.
\begin{equation*}
    \mathcal{Y}(\boldsymbol{A}^\mathcal{X}) = 
    \mathcal{Y}(\mathcal{X}(\boldsymbol{A})),
\end{equation*}
where $\mathcal{X}$ and $\mathcal{Y}$ are placeholders for complex interval types.
\end{definition}

There are three kinds of type casting. If the subspace corresponding to the source data type is a subset of the target data type's subspace, it is a widening casting and no tightness will be lost in the process; if it is the other way we talk about narrowing casting, which can result in a loss of tightness if the interval is not in the narrower subspace. If the data types are on the same level in the hierarchy we can talk about lateral casting, which typically results in a loss of tightness as these subspaces have no or very small overlapping regions.

\begin{equation*}
    \tau(\mathcal{Y}(\boldsymbol{A}^\mathcal{X}))
    \begin{cases}
        = \tau(\boldsymbol{A}^\mathcal{X}) 
        & \text{ if } 
        \mathcal{X}(\mathbb{C}) \subset \mathcal{Y}(\mathbb{C})
        \\
         = \tau(\boldsymbol{A}^\mathcal{X}) 
         & \text{ if } 
         \mathcal{X}(\mathbb{C}) \not\subset \mathcal{Y}(\mathbb{C})
         \text{ and } 
         \boldsymbol{A}^\mathcal{X} \in (\mathcal{X}(\mathbb{C}) \cap \mathcal{Y}(\mathbb{C}))
         \\
         < \tau(\boldsymbol{A}^\mathcal{X}) & \text{ otherwise } 
    \end{cases}
\end{equation*}


\begin{table}[t]
    \centering
    \begin{tabular}{c|ccccc|cc|cccc}
        & \multicolumn{5}{c|}{Representation of $\boldsymbol{A} \in$}
        & \multicolumn{2}{c|}{Unary
        $(\boldsymbol{A}^\mathcal{X})$}
        & \multicolumn{4}{c}{Binary 
        $\boldsymbol{A}^\mathcal{X} \bigcirc \boldsymbol{B}^\mathcal{X}$}
        \\
        Type
         & $\mathcal{R}$
         & $\mathcal{P}$
         & $\mathcal{C}$
         & $\mathcal{G}$
         & $\mathcal{A}$
         & $-(.)$
         & $(.)^{-1}$
         & $+$
         & $\times$
         & $\cap$
         & $\cup$
        \\
        \hline
        $\mathcal{R}(\boldsymbol{A})$
        & $=$ & $<$ & $<$ & $<$ & $<$
        & $=$ & $<$
        & $=$ & $<$ & $=$ & $<$
        \\
        $\mathcal{P}(\boldsymbol{A})$
        & $<$ & $=$ & $<$ & $<$ & $<$ 
        & $=$ & $<$ 
        & $<$ & $=$ & $=$ & $<$
        \\
        $\mathcal{C}(\boldsymbol{A})$
        & $<$ & $<$ & $=$ & $<$ & $<$
        & $=$ & $=$
        & $=$ & $<$ & $<$ & $<$
        \\
        $\mathcal{G}(\boldsymbol{A})$
        & $=$ & $\approx$ & $\approx$ & $=$ & $\approx$
        & $=$ & $\approx$
        & $=$ & $\approx$ & $=$ & $=$
        \\
        $\mathcal{A}(\boldsymbol{A})$
        & $=$ & $=$ & $=$ & $=$ & $=$
        & $=$ & $=$
        & $=$ & $\approx$ & $=$ & $=$
    \end{tabular}
    \caption{Tightness of interval representation, and unary and binary operations of complex interval types ($\mathcal{R}$: rectangular, $\mathcal{P}$: polar, $\mathcal{C}$: circular, $\mathcal{G}$: polygonal, $\mathcal{A}$: polyarcular, $\mathcal{X}$ is a placeholder). Each row represents an interval type, and the properties are ordered in columns. The first vertical block indicates how well the type can represent an interval from a given subspace (with the same symbols as the type). The second vertical block shows how well it can represent the result of a unary operation on an interval from the same subspace as the type. The third vertical block indicates how well it can represent the result of a binary operation with two intervals of the same subspace as the type. The property value $=$ indicates perfect tightness ($\tau\!=\!1$), $<$ indicates imperfect tightness ($\tau\!<\!1$) and $\approx$ indicates arbitrarily high tightness ($\tau\!\approx\! 1$).}
    \label{tab:interval_type_properties}
\end{table}

\subsection{Arithmetic operations}

Some of the most important properties of interval types are the computational complexity and accuracy, which are often connected to each other, forcing a compromise between computational speed and accuracy of arithmetic operations. For example, the polygonal type provides increasing accuracy for an increasing number of vertices when representing an interval that is not a member of the polygonal subspace.

Aligned with computational arithmetic conventions, let us force all complex type operations to result in a variable of the same type, and therefore implicit type-casting is not allowed. (For example, in C++ language, the division between the integer variables 2 and 5 results in $5/2=2$ in which case the result is truncated to become an integer.)

\begin{equation*}
    f(\boldsymbol{A}^\mathcal{X}) = 
    \mathcal{X}(f(\boldsymbol{A}^\mathcal{X}))
    ,\quad
    f(\boldsymbol{A}^\mathcal{X},\boldsymbol{B}^\mathcal{X}) = 
    \mathcal{X}(f(\boldsymbol{A}^\mathcal{X},\boldsymbol{B}^\mathcal{X}))
\end{equation*}

In Section \ref{sec:arithmetic_properties} we show that not all combinations of subspace operations result in an interval belonging to the same subspace as the operand(s). When intervals are represented as types, this means that the result has to be re-represented after the operation, which unavoidably decreases the tightness of the representation.

\begin{example}

The inverse of a rectangular interval is not rectangular; therefore, if we perform the inverse operation on its rectangular-type representation, the result will be relaxed to the smallest enclosing rectangle. An accurate result can be achieved by type-casting the operand to polyarcular type and performing the inverse on that.

\begin{equation*}
    \boldsymbol{A} \in \mathcal{R}(\mathbb{C}),
    \boldsymbol{A}^{-1} \not\in \mathcal{R}(\mathbb{C}) \implies
    \left(\boldsymbol{A}^\mathcal{R}\right)^{-1} = 
    \mathcal{R}\left(\boldsymbol{A}^{-1}\right) \supset \boldsymbol{A}^{-1} 
    = \left(\boldsymbol{A}^\mathcal{A}\right)^{-1}
\end{equation*}
\end{example}

When a function combines several intervals from a subspace that is not closed under the operations, the re-representation error can result a cumulative loss of tightness. In this case the bounds of the result representation are not necessarily touching the bounds of the result.

\begin{example}\label{eg:cascading_type_error}

The sum of polar intervals is not polar. Therefore, if the sum is performed on two polar-type operands, the result will be relaxed to the smallest enclosing polar interval. If we add a third polar interval to this relaxed sum, then the result will contain the representation error of the first operation's result and the relaxation of the second operation. This results a cascading error.

\begin{equation*}
\begin{array}{c}
    \boldsymbol{A},\boldsymbol{B},\boldsymbol{C} \in \mathcal{P}(\mathbb{C}), \quad
    \mathcal{P}(\boldsymbol{A}) + \mathcal{P}(\boldsymbol{B})+ \mathcal{P}(\boldsymbol{C}) \not\in \mathcal{P}(\mathbb{C}) \implies
    \\
    \boldsymbol{A}^\mathcal{P} + \boldsymbol{B}^\mathcal{P} + 
    \boldsymbol{C}^\mathcal{P} = 
    \mathcal{P}(\mathcal{P}(\boldsymbol{A} + \boldsymbol{B}) + \boldsymbol{C}) 
    \subset  \mathcal{P}(\boldsymbol{A} + \boldsymbol{B} + \boldsymbol{C})
    \subset  (\boldsymbol{A} + \boldsymbol{B} + \boldsymbol{C})
\end{array}
\end{equation*}
\end{example}

Unlike Example \ref{eg:cascading_type_error}, if the subspace is closed under the operation, then the combination of all the representation errors equals the representation error of the result.

\begin{example}
The sum of rectangular invervals is rectangular. Therefore, if we represent three nonrectangular intervals with the rectangular type, then their sum will contain only the representation error and will be the same as the rectangular representation of the interval sum. The error does not cascade.
\begin{equation*}
\begin{array}{c}
    \boldsymbol{A},\boldsymbol{B},\boldsymbol{C} \not\in \mathcal{R}(\mathbb{C}),
    \quad
    \mathcal{R}(\boldsymbol{A}) + \mathcal{R}(\boldsymbol{B})+ \mathcal{R}(\boldsymbol{C}) \in  \mathcal{R}(\mathbb{C}) \implies
    \\
    \boldsymbol{A}^\mathcal{R} + 
    \boldsymbol{B}^\mathcal{R} + 
    \boldsymbol{C}^\mathcal{R} = 
    (\mathcal{R}(\boldsymbol{A}) + 
    \mathcal{R}(\boldsymbol{B})) + 
    \mathcal{R}(\boldsymbol{C}) =  
    \mathcal{R}(\boldsymbol{A} + \boldsymbol{B} + \boldsymbol{C})
    \subset  \boldsymbol{A} + \boldsymbol{B} + \boldsymbol{C}
\end{array}
\end{equation*}
\end{example}

Algorithms for all basic arithmetic operations with primitive interval types are available in the literature. \cite{petkovic_complex_1998,moore_introduction_2009,dawood_theories_2011} Typically operations that yield a result in the operand's subspace are simpler and faster (e.g. rectangular addition or polar multiplication), while operations that have to re-represent the result are more complex (such as rectangular multiplication, or polar addition). For the latter group of operations, fast and loose algorithms are typically also available, which do not provide the tightest representation around the result but offer a faster calculation (e.g. circular multiplication or rectangular inverse).

The so-called Minkowski method allows performing operations on polygonal intervals using their defining vertices only. This allows the edges of the polygon to stay implicit throughout the entire process, with the exception to the reciprocal operation, which requires the approximation of convex arcs. \cite{ohta_polygon_1990} If combined with the Gauss map matching method \cite{farouki_algorithms_2000}, which is a simple intersection operation between the curve normal argument intervals of each vertex pair, we can get an algorithm with linear complexity by the number of vertices. For the addition and negative operations, this results tight bounds, while for the multiplication it replaces the concave parabolic curves with straight edges. 

Polyarcular intervals are defined by a set of arcs, with the vertices and edges stored implicitly. Similarly to the reciprocal of polygons, the implicit components have to be calculated for certain operations. The negative operation, for example, requires the negation of the defining arcs only. The sum operation has to consider the Gauss map matched sum of the operand vertices and arcs, where vertices simply translate the other segments, but for two arcs the envelope segment has to be calculated. Reciprocal and multiplication operations require the evaluation of all three types of curve segments. Since the product envelope of arcs and edges are not polyarcular in general, their approximation may be necessary. However, since in many cases vertices cover most of the Gauss map, the approximated curve segments typically constitute only a small part of the total curve. The computational complexity therefore depends significantly on the operands and varies case-by-case.  

\subsection{Set operations} \label{sub:set_operations}

Algorithms for all basic set operations that involve primitive interval types are available in the literature. \cite{boche_complex_1965,gargantini_circular_1971,candau_complex_2006,moore_introduction_2009}

In case of polygonal and polyarcular interval types, set operations require the identification of intersections between curve segments, where they can be split, and then recombined according to the operation's logic. The intersection can be easily found at the points where the parametric equations of the segments are equal, while splitting a segment can be easily done by duplicating it and then splitting its parameter interval at the intersection point. Once ordered sets of split segments are available, inner and outer segments can be identified by finding a single extremal point in the whole set and then counting the number of intersections along the boundaries. The boundary of the union then consists of all the outer segments, and the boundary of the intersection consists of all the inner segments.

\subsection{Trimming} \label{sub:trimming}
A special case of the set operations is the trimming, which is a closed unary set operation from closed curves to simple closed curves. The addition and multiplication of the polygonal and polyarcular curve segments can result in self-intersecting boundaries when the operands are non-convex (in the multiplicative case: non-log-convex). Such a boundary is not acceptable as complex intervals have to be bounded by simple curves; therefore we have to extract the outer curve. This can be done using the trimming method described in \cite{farouki_boundary_2005}, or if holes in the interval is not a concern, a simple always-turn-right rule can be used at each intersection of the counter-clock-wise oriented curve. For the sake of brevity, we will assume this step to be implicit and will not indicate in arithmetic equations.

\begin{equation*}
\boldsymbol{\Gamma} = \{\Gamma(t), t\in\boldsymbol{t}\} \notin \mathcal{O}(\mathbb{C}),
\mathcal{O}(\boldsymbol{\Gamma}) \in \mathcal{O}(\mathbb{C})\implies
\mathcal{I}(\mathbb{C} \big\vert \boldsymbol{\Gamma}) := 
\mathcal{I}(\mathbb{C} \big\vert \mathcal{O}(\boldsymbol{\Gamma})),
\end{equation*}
where $\mathcal{O}(\boldsymbol{\Gamma})$ is the trimming operation.

 An example of a non-simple result boundary curve can be seen on Fig. \ref{fig:Gauss_map}.

\subsection{Backtracking} \label{sub:backtracking}

As shown in Definition \ref{def:complex_interval_backtracking} it is possible to elementwise backtrack the ${\mathbb{C} \times_c \mathbb{C} \rightarrow \mathbb{C}}$ mapping of the Minkowski addition and multiplication operations on complex intervals. This requires the negation, translation and intersection operations for the backtracking of an addition; and it requires the reciprocal, rotate-and-scale and intersection operations for the backtracking of a multiplication. We defined all of these operations for polyarcular intervals above; therefore, the algorithms are ready to be implemented.

Backtracking can be used to verify tightness, because the points in the relaxation region of the representation have no corresponding subsets in the operand intervals. It can also be used to investigate the cause of certain outcomes, for example, to identify worst-case error patterns in tolerance analysis. 
\begin{example}
It has been shown that the smallest rectangle enclosing the sum of two non-rectangular intervals will not be tight. If we backtrack a point in the relaxation region (the difference of the interval and its representation), the corresponding operand subsets will be empty (see Definition \ref{def:complex_interval_backtracking}).
\begin{equation*}
    \boldsymbol{A},\boldsymbol{B} \notin \mathcal{R}(\mathbb{C})
    \implies
    \exists Z \in (\boldsymbol{A}^\mathcal{R} \oplus \boldsymbol{B}^\mathcal{R}):
    \boldsymbol{A}_Z = \boldsymbol{A} \cap (-\boldsymbol{B}+Z) = \emptyset
    .
\end{equation*}
\end{example}

Backtracking can also be of use in the investigation of the interval dependency problem. \cite{dawood_logical_2019} When an operand interval appears more than once in a more complicated expression (such as $\boldsymbol{A} \oslash (\boldsymbol{A} \oplus \boldsymbol{B})$), the naive result evaluated assuming independent operands will be too relaxed. The backtracking of each point in this relaxation region to the instances of the repeated operand reveals that they are not valid.

\begin{example}
For the sake of simplicity let us consider the real valued function ${\boldsymbol{a}/(\boldsymbol{a}+\boldsymbol{b})}$ 
with the values 
${\boldsymbol{a}=[1,3]}$ and ${\boldsymbol{b}=[2,4]}$. 
The correct solution is 
${[\underline{a}/(\underline{a}+\overline{b}),\overline{a}/(\overline{a}+\underline{b})] = [1/5,3/5]}$, however the interval arithmetic result is the relaxed interval
${[\underline{a}/(\overline{a}+\overline{b}),\overline{a}/(\underline{a}+\underline{b})] = [1/7,1]}$. If we now backtrack the subsets of the $\boldsymbol{a}$ instances corresponding to the $z\!=\!4/5$ point in the relaxation region, we get
${\boldsymbol{a}_{1z}=\boldsymbol{a}\cap[(\boldsymbol{a}+\boldsymbol{b}) z]=[2.40,3.00]}$ and
${\boldsymbol{a}_{2z}=\boldsymbol{a}\cap[\boldsymbol{a}/z-\boldsymbol{b}]=[1,1.75]}$, which has an empty intersection, so there is no such solution where 
${a_1,a_2\in\boldsymbol{a},b\in\boldsymbol{b},\;a_1\!=\!a_2,\;a_1/(a_2+b)=z}$.
\end{example}

The location and size of the corresponding operand subsets can also provide information about the probability of an outcome if we assume a probability density function over the operand intervals (e.g. uniform).


\begin{figure*}[t!]
\centering
\includegraphics[width=\textwidth]{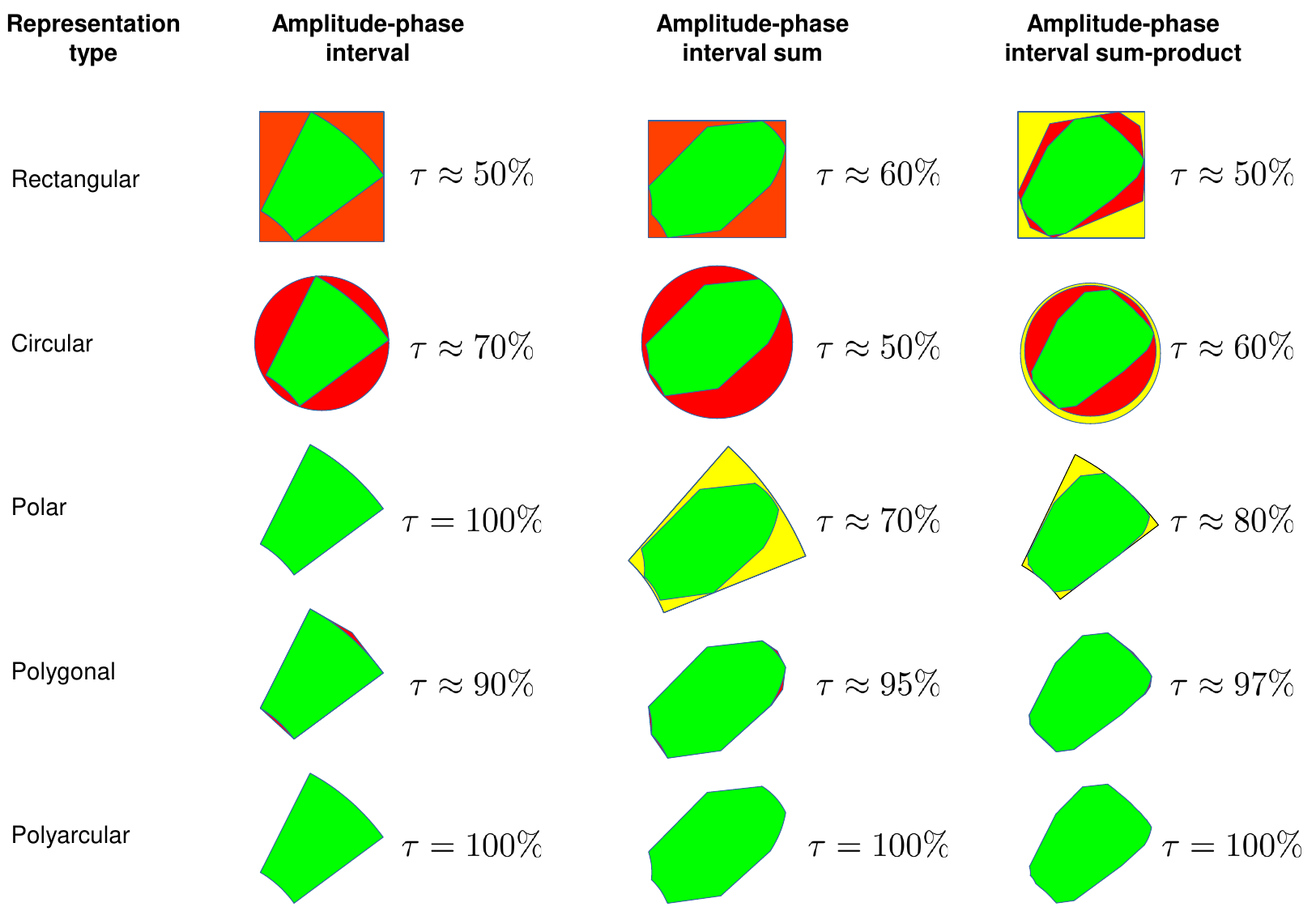}
\caption{Example of a sensor array tolerance analysis using various complex interval types. Green area indicate the complex interval, red area indicates the representation error, yellow area indicates the operation error. (For illustration only.)}
\label{fig:case_study}
\end{figure*}

\section{Case study} \label{sec:case_study}

In this section, we present a case study that motivated the development of the polyarcular interval type. The tolerance analysis of antenna arrays using interval analysis is an active research area within sensor array design and signal processing. \cite{he_comparison_2019,he_impact_2021} We have previously analyzed the worst-case spatial response of acoustic arrays impacted by calibration errors and mutual coupling. \cite{arnestad_worst-case_2023} Later, we found that given the physical model in our study, the polyarcular interval type yields a tight bound around the complex interval of the array response, which we demonstrate in the following.

Let 
${\{\boldsymbol{A}_\mathsf{n} \in \mathcal{P}(\mathbb{C}) \big\vert \mathsf{n} \in \{1..\mathsf{N}\}\}}$ 
represent the combined amplitude and phase sensitivity interval of the individual elements of a sensor array, and let 
${\{\boldsymbol{C}_\mathsf{m,n} \in \mathcal{P}(\mathbb{C}) \big\vert \mathsf{m,n} \in \{1..\mathsf{N}\}\}}$
represent the coupling coefficient interval of each pair of elements, including the self-coupling ${\boldsymbol{C}_\mathsf{n,n}=1}$. Then, assuming a narrow-band, far-field operation, the complex response interval of the array is 
${ \boldsymbol{B} = \sum_\mathsf{n} \boldsymbol{A}_\mathsf{n} \sum_\mathsf{m} \boldsymbol{C}_\mathsf{m,n} }$.

Since the addition of polar intervals is not only not tight, but also computationally heavy, the literature considers using rectangular, circular, and convex polygonal representations to calculate bounds on the array response interval. While computationally very light, the usage of the rectangular and circular types will introduce a signficant loss of tightness at the type-casting and another loss of tightness though the multiplicaton
\[
\tau\left(\boldsymbol{B}^{\mathcal{R} \text{ or } \mathcal{C}}\right) \leq \tau\left(\boldsymbol{C}_\mathsf{m,n}^{\mathcal{R} \text{ or } \mathcal{C}}\right) \leq 1.
\]
The convex polygonal type can represent the outer convex arc of a polar interval with arbitrary precision, but it replaces the concave inner arc with an edge. One could argue that since the addition is a convexifying operation -- in other words, the sums of many concave sets are approximately convex \cite{schneider_convex_1993} -- the representation of the sum will be sufficiently tight in most cases. However, the product of polygonal intervals can include concave parabolic segments, which is relaxed to edges through the mupltiplication causing additional loss of tightness
\[
\tau\left(\boldsymbol{B}^\mathcal{G}\right) \leq \tau\left(\sum_\mathsf{m} \boldsymbol{C}_\mathsf{m,n}^\mathcal{G}\right) \approx 1.
\]

We found that the polyarcular type is an ideal choice for this application, as it provides perfect tightness for a limited increase in computational complexity.
We showed that the sum of the polar intervals is in the polyarcular subspace:
${ \sum_\mathsf{m}\boldsymbol{C}_\mathsf{m,n}^\mathcal{A} \in \mathcal{A}(\mathbb{C}) }$, 
and that the product of the polyarcular intervals is polyarcular if at least one of each pair of operand segments is from a zero-crossing line or a zero-centered circle. 
Since the boundary segments of the polar intervals are all from zero crossing lines or zero-centered circles, the product of a polyarcular and a polar interval is in the polyarcular subspace:
$\boldsymbol{A}_\mathsf{n}^\mathcal{A} \sum_\mathsf{m}\boldsymbol{C}_\mathsf{m,n}^\mathcal{A} \in \mathcal{A}(\mathbb{C})$. 
Finally, the sum of polyarcular intervals is polyarcular; therefore, the complex response interval is in the polyarcular subspace:
${\boldsymbol{B} \in \mathcal{A}(\mathbb{C})}$. 
This means that the polyarcular type provides perfectly tight bounds on the complex response interval:
\[{\tau(\boldsymbol{B}^\mathcal{A})=1.}\]

Figure \ref{fig:case_study} shows an example of such a tolarance analysis.


\section{Conclusion}\label{sec:conclusion}

In this paper, we showed that all commonly used complex interval types can be represented and arithmetically combined using the polyarcular interval type with improved or equivalent tightness in return for a moderate increase in complexity. This makes the polyarcular type a valuable option for performing calculations in various cases of interval analysis. We have shown that, similar to the polygonal type, simple arithmetic operations can be performed with a computational complexity linearly dependent on the number of elements constituting the operand boundaries.  We presented a case study that served as our motivation to develop this method. It shows that the polyarcular interval type is applicable in practical design tasks. 

We found no commonly used formal definition of complex intervals, so we followed suit in finding a practical definition fitting the paper's scope. While it was tempting to further discuss the general properties of complex intervals, we rather relied on the existing literature and focused on the representations in this paper. However, we could not resist trying to fill in some of the gaps in the discussion of the arithmetic properties of primitive complex intervals, which is a significantly less published area than its two neighbours: the real interval arithmetic and the geometry of plane curves.

One could argue that polygonal interval arithmetic is the most general approach to the representation of complex intervals that provides very tight bounds as it can sample the infinite boundary set with an arbitrary resolution and perform point-wise operations on the vertices. (Hence, polygonal interval arithmetic is also called the Minkowski method.) However, as we showed, the inclusiveness of the polygonal bounds is not automatically guaranteed when sampling convex curves or performing the reciprocal operation. \cite{ohta_polygon_1990} We can also argue that within the complex interval subspaces the polyarcular type is more general than the polygonal, and that perfect tightness for a fixed computational complexity can be preferred in some cases over a variable tightness depending on the interval attributes and chosen number of vertices. 

While our derivations provide small theoretical progress beyond the work of Farouki et al., \cite{farouki_minkowski_2001} the closed form parametric conditions allowed us to design efficient algorithms by dividing the segment-wise operations into sub-cases of increasing complexity. We also hope that the used method (see Section \ref{sec:arithmetic_properties}) will enable the analysis of other problems too.

Areas of potential further research are the derivation of additional algebraic functions and conformal mappings to widen the range of applicability in interval analytical problems, while it would also be desirable to attempt to bridge the gap between statistics and interval analysis by investigating how simple unary and binary operations work on complex probabilistic variables (inspired by the grayscale morphology in \cite{giardina_morphological_1988}). We also intend to publish a polyarcular interval type code implementation in the near future in our existing repository of complex interval arithmetic codes: \url{https://github.com/unioslo-mn/ifi-complex-interval-arithmetic/}.

\backmatter

\bmhead{Supplementary information}
Derivation of the arithmetic properties of edges and arcs.

\bmhead{Acknowledgments}

Special thanks to Håvard Arnestad (Department of Informatics, University of Oslo) and Tor Inge Lønmo (Kongsberg Discovery) for their assistance in the case study and the review and editing of the manuscript.

We appreciate the assistance of Andreas Austeng, Jan Egil Kirkebø, Sven Peter Näsholm (Department of Informatics, University of Oslo) and Tom Louis Lindstrøm (Department of Mathematics, University of Oslo) for useful comments on an earlier draft of this manuscript.

Many thanks to Jacob Sznajdman (Neo4j) for his advice during the concept development.

We are grateful to László Surányi for connecting the two authors and we thank Gergő Pintér (Budapest University of Technology) for the enlightening conversations.

G.\ Ger\'eb acknowledge funding from the Research Council of Norway project \emph{Element calibration of sonars and echosounders}, project number 317874. 

A.\ S\'andor acknowledges funding from the Global Teaching Fellowship Program of Central European University and the Élvonal (Frontier) Grant KKP144148 of the NKFIH.

\bibliography{bibliography.bib}


\end{document}


\newcommand{\atan}{\mathrm{atan}}   
\newcommand{\atantwo}{\mathrm{atan2}} 

\renewcommand*{\thesection}{S\arabic{section}}

\title{Derivation of the arithmetic properties of edges and arcs
\\ 
{\normalsize Supplementary material for "Polyarc-bounded complex interval arithmetic" paper}}

\author{Gábor Geréb, András Sándor}

\maketitle

\tableofcontents

\vfill \eject

\section{Unary operations}\label{sup:unitary_operations}

\textbf{Operand equation}

For $\boldsymbol{A} \in \mathcal{I}(\mathbb{C})$,
$\boldsymbol{\Gamma} \subset \partial \boldsymbol{A}$

\begin{equation*}
\begin{array}{lll}
\boldsymbol{\Gamma}&=\left\lbrace F(s)\big\vert s\in \boldsymbol{s}\right\rbrace
\\
&=\left\lbrace x+iy\big\vert f(x,y)=0,s(x,y)\in \boldsymbol{s}\right\rbrace 
\\
&=\left\lbrace x+iy\big\vert\mathring{f}(\rho,\theta)=0,\mathring{s}(\rho,\theta)\in \boldsymbol{s}\right\rbrace 
\end{array}
\end{equation*}

\subsection{Negative}\label{sup:negative}

\textbf{Result equation}

\begin{equation*}
\begin{array}{ll}
    -\boldsymbol{\Gamma} &= \left\lbrace -F(s)\big\vert s\in \boldsymbol{s}\right\rbrace
    \\ &=   \left\lbrack -x-iy \big\vert f(x,y)=0,s(x,y)\in \boldsymbol{s}\right\rbrack 
    \\ &=   \left\lbrack x+iy \big\vert f(-x,-y)=0,s(-x,-y)\in \boldsymbol{s}\right\rbrack 
\end{array}
\end{equation*}

\vfill\eject

\subsubsection{Edge negative}\label{sup:edge_negative}

\begin{figure}[H]
\centering
\includegraphics[width=\textwidth,trim={0 0 0 10mm},clip]{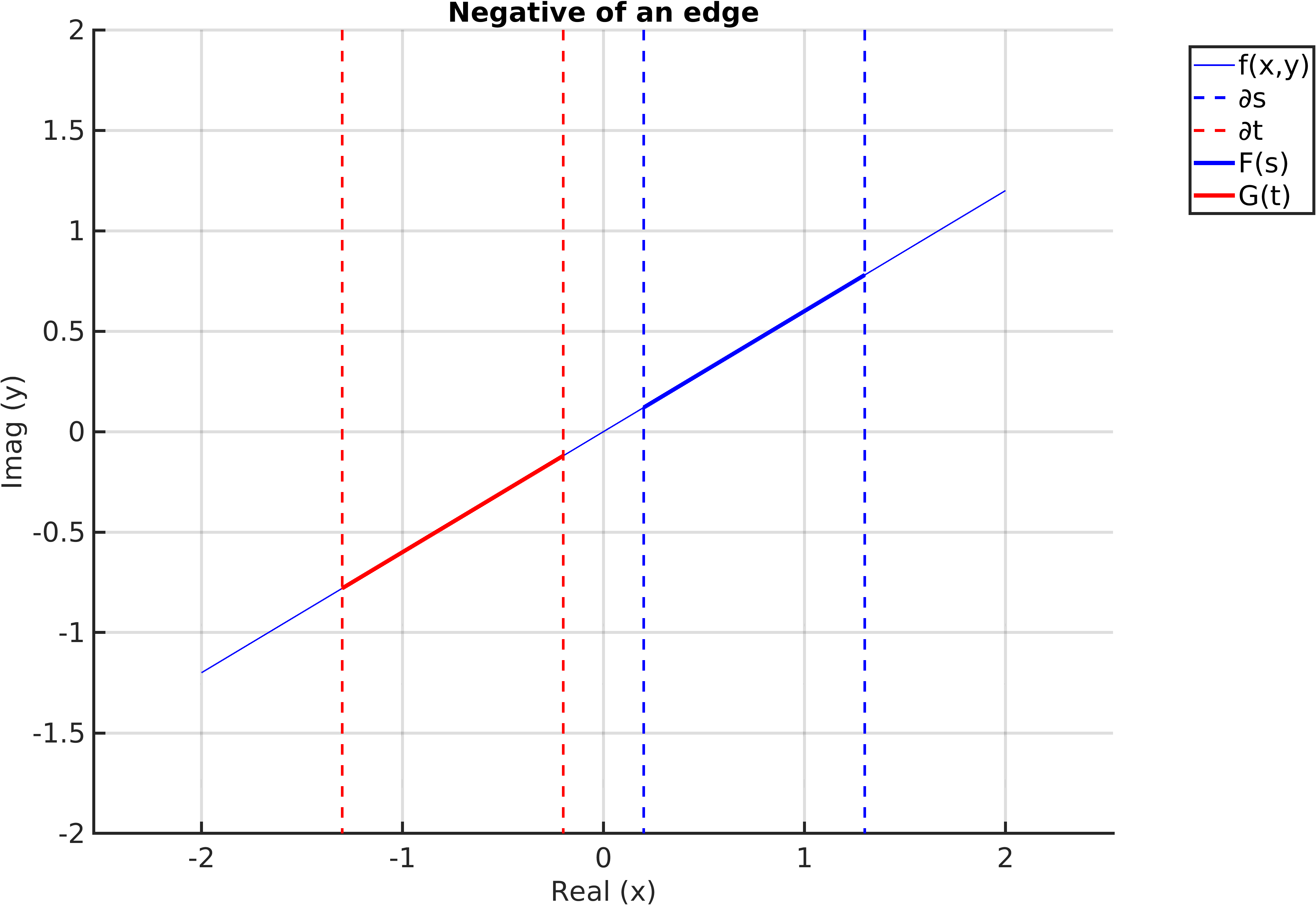}
\caption{Negative of an edge, where $G(t)=-F(s)$. Parameters: $a=0.6$, $\boldsymbol{s}=(0.2,1.3)$.}
\label{fig:EdgeNegative}
\end{figure}

\begin{equation*}
\begin{array}{lll}
f(x,y)&=a x-y
\\
s(x,y)&=x
\end{array}
\end{equation*}

\begin{equation*}
\begin{array}{ll}
    -\boldsymbol{\Gamma} &= \left\lbrack -x-iy\big\vert ax-y=0,x\in \boldsymbol{s}\right\rbrack 
    \\
    &=\left\lbrack x+iy\big\vert ax-y=0,-x\in \boldsymbol{s}\right\rbrack
\end{array}
\end{equation*}

\vfill \eject

\subsubsection{Arc negative}\label{sup:arc_negative}

\begin{figure}[H]
\centering
\includegraphics[width=\textwidth,trim={0 0 0 10mm},clip]{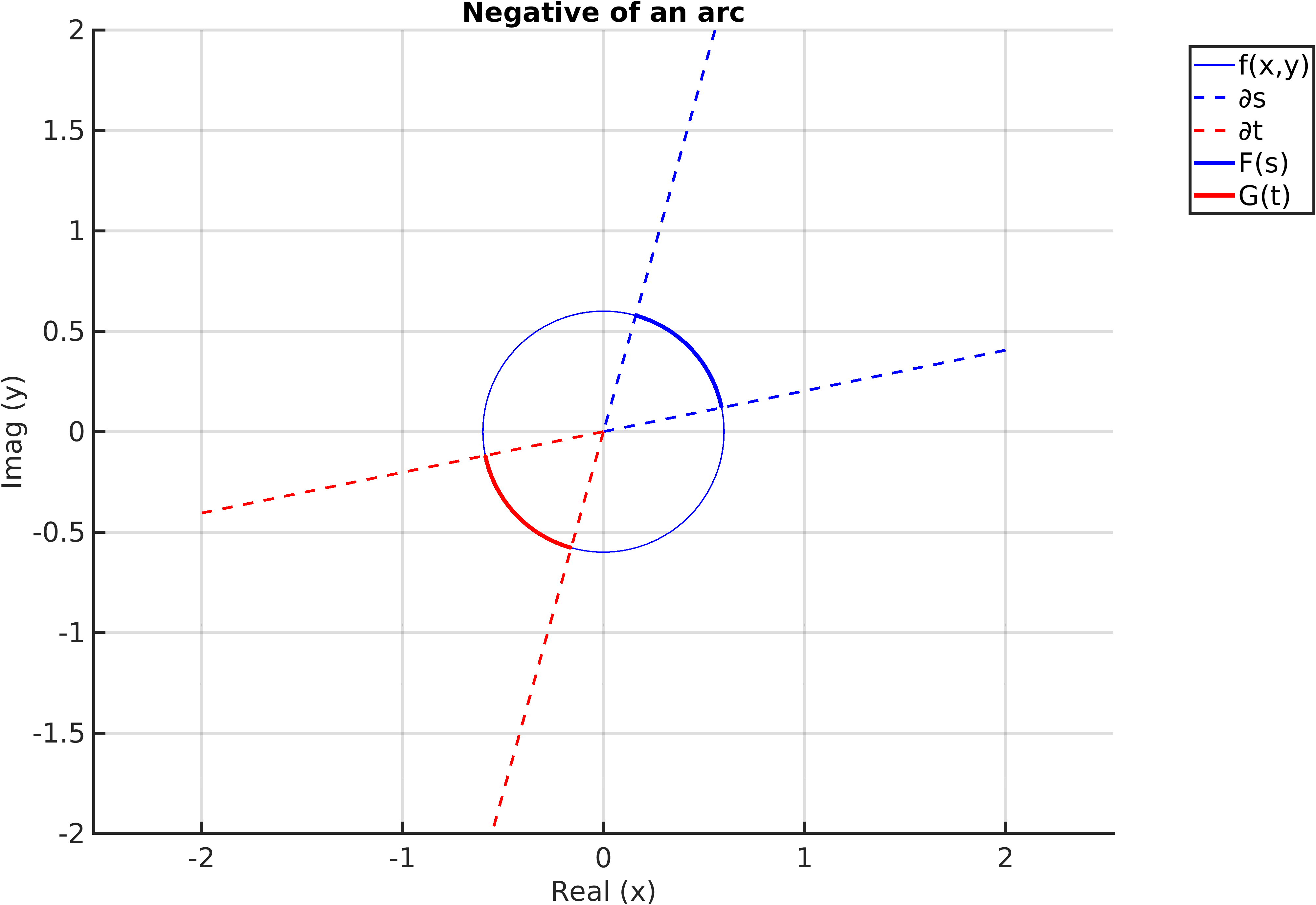}
\caption{Negative of an arc, where $G(t)=-F(s)$. Parameters: r=0.6, s=(0.2,1.3).}
\label{fig:ArcNegative}
\end{figure}

\begin{equation*}
\begin{array}{lll}
f(x,y)&=x^2-y^2-r^2
\\
s(x,y)&=\atantwo(y,x)
\end{array}
\end{equation*}

\begin{equation*}
\begin{array}{ll}
    -\boldsymbol{\Gamma} &= 
    \left\lbrace -x-\mathrm{iy}\big\vert x^2 +y^2 -r^2 =0,\atantwo\left(-y,-x\right)\in \boldsymbol{s}\right\rbrace
    \\
    &=\left\lbrace x+\mathrm{iy}\big\vert x^2 +y^2 -r^2 =0,\atantwo\left(y,x\right)-\pi \in \boldsymbol{s}\right\rbrace
\end{array}
\end{equation*}

\vfill \eject

\subsection{Reciprocal}\label{sup:reciprocal}

\begin{equation*}
\begin{array}{ll}
    \boldsymbol{\Gamma}^{-1} 
    &= \left\lbrace \frac{1}{F(s)}\big\vert s\in \boldsymbol{s}\right\rbrace
    \\ &=   \left\lbrack \frac{1}{\rho}e^{-i\theta} \big\vert \mathring{f}(\rho,\theta)=0,\mathring{s}(\rho,\theta)\in \boldsymbol{s}\right\rbrack 
    \\ &=   \left\lbrack \rho e^{i\theta} \big\vert \mathring{f}(\frac{1}{\rho},-\theta)=0,\mathring{s}(\frac{1}{\rho},-\theta)\in \boldsymbol{s}\right\rbrack
\end{array}
\end{equation*}

\vfill \eject
\subsubsection{Edge reciprocal}\label{sup:edge_reciprocal}

\begin{figure}[H]
\centering
\includegraphics[width=\textwidth,trim={0 0 0 10mm},clip]{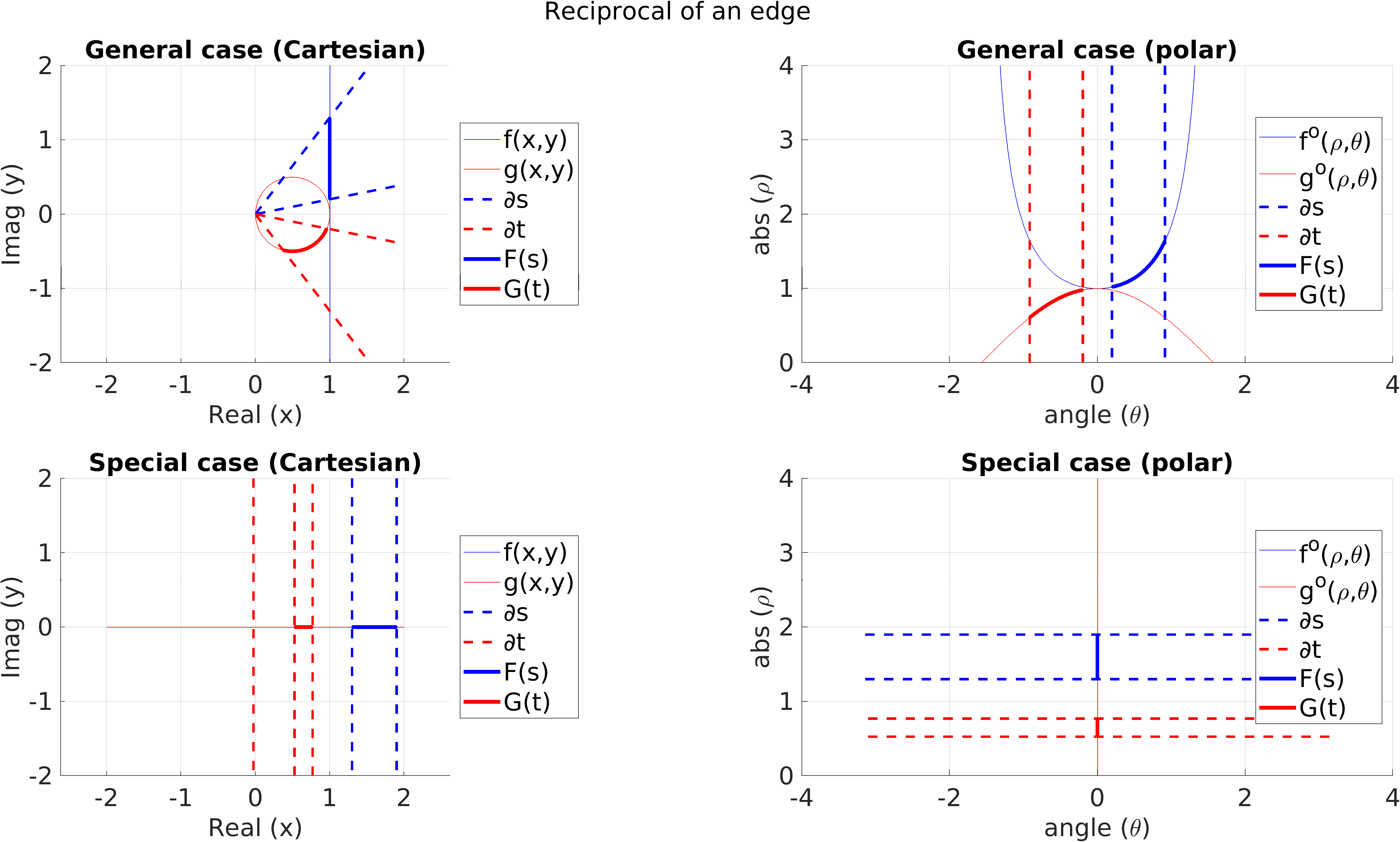}
\caption{Reciprocal of an edge, where $G(t)\!=\!1/F(s)$, $g(x,y)\!=\!1/f(x,y)$, and $\mathring{g}(\rho,\theta)\!=\!1/\mathring{f}(\rho,\theta)$. Parameters: $\boldsymbol{s}=(0.2,1.3)$ in the general case, and $\boldsymbol{s}=(1.3,1.9)$ in the special case.}
\label{fig:EdgeReciprocal}
\end{figure}

\paragraph{General case} ~\\

\begin{equation*}
\begin{array}{ll}
    f(x,y) &= x-1
    \\
    s(x,y) &= y
    \\
    \mathring{f}(\rho,\theta) &= \frac{1}{\rho}-\cos(\theta)
    \\
    \mathring{s}(\rho,\theta) &= \rho\sin(\theta)
\end{array}
\end{equation*}

\begin{equation*}
\begin{array}{ll}
    \boldsymbol{\Gamma}^{-1} 
    &= \left\lbrace \rho e^{i\theta} \big\vert \rho - \cos(-\theta)=0 , \frac{1}{\rho}\sin(-\theta)\in\boldsymbol{s}  \right\rbrace
    \\
    &= \left\lbrace \rho e^{i\theta} \big\vert \rho - \cos(\theta)=0 , -\frac{1}{\rho}\sin(\theta)\in\boldsymbol{s}  \right\rbrace
    \\
    &= \left\lbrace x+iy \big\vert \sqrt{x^2+y^2} - \cos(\atantwo(y,x))=0 , -\frac{1}{\sqrt{x^2+y^2}}\sin(\atantwo(y,x))\in\boldsymbol{s}  \right\rbrace
    \\
    &= \left\lbrace x+iy \big\vert \sqrt{x^2+y^2} - \frac{x}{\sqrt{x^2+y^2}}=0 , -\frac{1}{\sqrt{x^2+y^2}}\frac{y}{\sqrt{x^2+y^2}}\in\boldsymbol{s}  \right\rbrace
    \\
    &= \left\lbrace x+iy \big\vert x^2+y^2 - x=0 , -\frac{y}{x^2+y^2}\in\boldsymbol{s}  \right\rbrace
    \\
    &= \left\lbrace x+iy \big\vert (x-\frac{1}{2})^2+y^2-(\frac{1}{2})^2=0 , -\frac{y}{x^2+y^2}\in\boldsymbol{s}  \right\rbrace
\end{array}
\end{equation*}

\paragraph{Special case: the edge is on a zero crossing line} ~\\

\begin{equation*}
\begin{array}{ll}
    f(x,y) &= y
    \\
    s(x,y) &= x
    \\
    \mathring{f}(\rho,\theta) &= \tan(\theta)
    \\
    \mathring{s}(\rho,\theta) &= \rho
\end{array}
\end{equation*}

\begin{equation*}
\begin{array}{ll}
    \boldsymbol{\Gamma}^{-1} 
    &= \left\lbrace \rho e^{i\theta} \big\vert \tan(-\theta)=0 , \frac{1}{\rho}\in\boldsymbol{s}  \right\rbrace
    \\
    &= \left\lbrace \rho e^{i\theta} \big\vert -\tan(\theta)=0 , \frac{1}{\rho}\in\boldsymbol{s}  \right\rbrace
    \\
    &= \left\lbrace x+iy \big\vert -\tan(\atantwo(y,x))=0 , -\frac{1}{\sqrt{x^2+y^2}}\in\boldsymbol{s}  \right\rbrace
    \\
    &= \left\lbrace x+iy \big\vert y=0 , \frac{1}{x}\in\boldsymbol{s}  \right\rbrace
\end{array}
\end{equation*}

\vfill \eject

\subsubsection{Arc reciprocal}\label{sup:arc_reciprocal}

\begin{figure}[H]
\centering
\includegraphics[width=\textwidth,trim={0 0 0 10mm},clip]{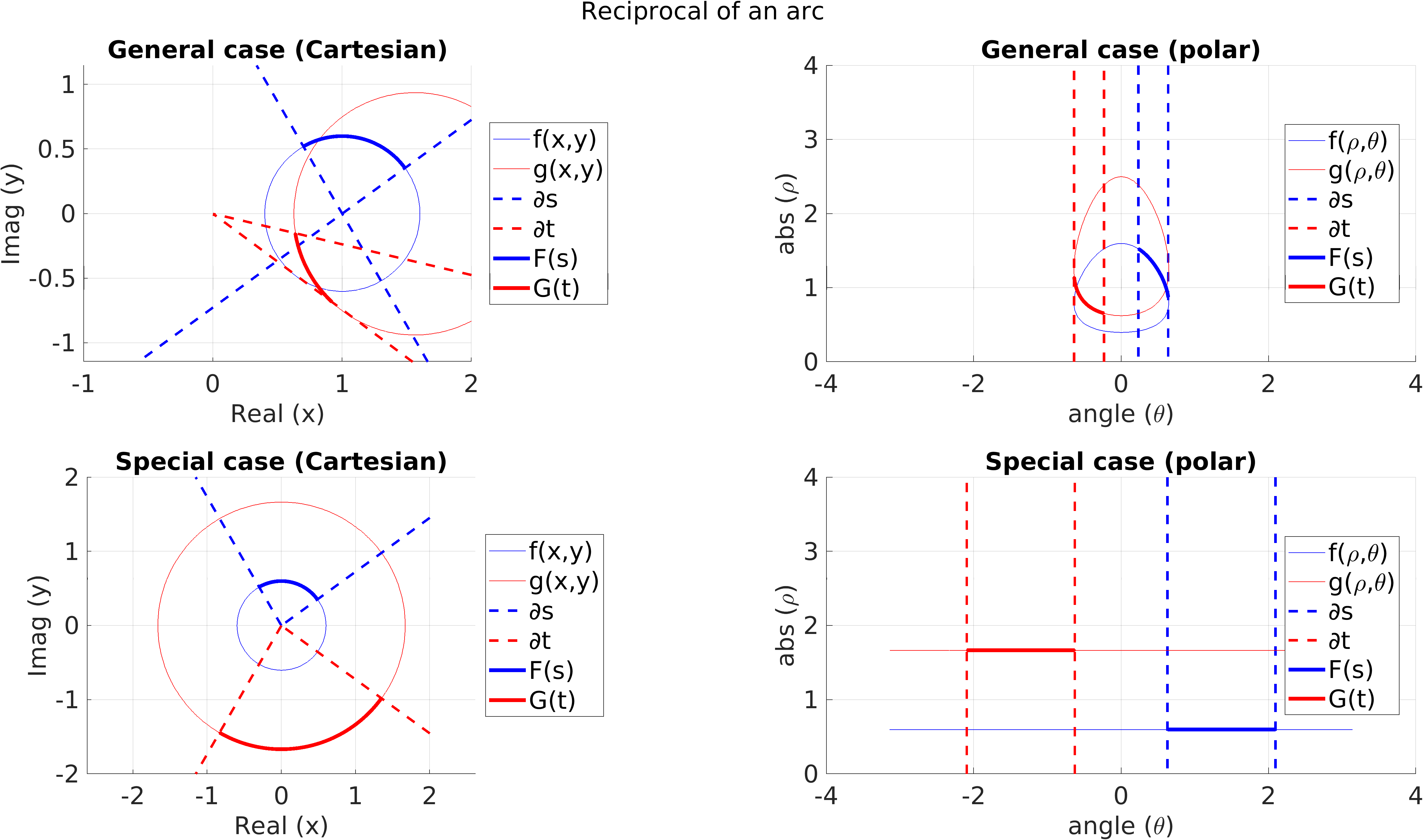}
\caption{Reciprocal of an arc, where $G(t)\!=\!1/F(s)$, $g(x,y)\!=\!1/f(x,y)$, and $\mathring{g}(\rho,\theta)\!=\!1/\mathring{f}(\rho,\theta)$. Parameters: $r=0.6$, $\boldsymbol{s}=(\pi/5,2\pi/3)$ in both cases.}
\label{fig:ArcReciprocal}
\end{figure}

\paragraph{General case} ~\\

\begin{equation*}
\begin{array}{ll}
    f(x,y) &= (x-1)^2 + y^2 - r^2
    \\
    s(x,y) &= \atantwo(y,x-1)
    \\
    \mathring{f}(\rho,\theta) &= \rho^2 - 2\rho\cos(\theta)+1-r^2
    \\
    \mathring{s}(\rho,\theta) &= \atantwo(\rho\sin(\theta),\rho\cos(\theta)-1)
\end{array}
\end{equation*}

\begin{equation*}
\begin{array}{ll}
    \boldsymbol{\Gamma}^{-1} 
    &= \left\lbrace \rho e^{i\theta} \big\vert \frac{1}{\rho^2} + \frac{2}{\rho}\cos(-\theta)+1-r^2=0 , \atantwo(\frac{1}{\rho}sin(-\theta),\frac{1}{\rho}\cos(-\theta)-1) \in \boldsymbol{s}  \right\rbrace
    \\
    &= \left\lbrace \rho e^{i\theta} \big\vert 1 + 2\rho\cos(\theta)+\rho^2(1-r^2)=0 , \atantwo(-\frac{1}{\rho}\sin(\theta),\frac{1}{\rho}\cos(\theta)-1) \in \boldsymbol{s}  \right\rbrace
    \\
    &= \left\lbrace x+iy \big\vert 1 + 2\sqrt{x^2+y^2}\cos(\atantwo(y,x))+(x^2+y^2)(1-r^2)=0 ,
    \right.\\ & \left. \quad\quad \atantwo(-\frac{1}{\sqrt{x^2+y^2}}\sin(\atantwo(y,x)), \frac{1}{\sqrt{x^2+y^2}}\cos(\atantwo(y,x))-1) \in \boldsymbol{s}  \right\rbrace
    \\
    &= \left\lbrace x+iy \big\vert 1 + 2\sqrt{x^2+y^2}\frac{x}{\sqrt{x^2+y^2}}+(x^2+y^2)(1-r^2)=0 ,
    \right.\\ & \left. \quad\quad \atantwo\left(-\frac{1}{\sqrt{x^2+y^2}}\frac{y}{\sqrt{x^2+y^2}},\frac{1}{\sqrt{x^2+y^2}}\frac{x}{\sqrt{x^2+y^2}}-1\right) \in \boldsymbol{s}  \right\rbrace
    \\
    &= \left\lbrace x+iy \big\vert 1 + 2x+(x^2+y^2)(1-r^2)=0 ,
    \atantwo\left(-\frac{y}{x^2+y^2},\frac{x}{x^2+y^2}-1\right) \in \boldsymbol{s}  \right\rbrace
    \\
    &= \left\lbrace x+iy \big\vert x^2 + \frac{2x}{1-r^2} + \frac{1}{1-r^2} +y^2 =0 ,
    \atantwo\left(-\frac{y}{x^2+y^2},\frac{x}{x^2+y^2}-1\right) \in \boldsymbol{s}  \right\rbrace
    \\
    &= \left\lbrace x+iy \big\vert \left(x-\frac{1}{1-r^2}\right)^2 +y^2 - \left(\frac{1}{1-r^2}\right)^2  =0 ,
    \atantwo\left(-\frac{y}{x^2+y^2},\frac{x}{x^2+y^2}-1\right) \in \boldsymbol{s}  \right\rbrace
\end{array}
\end{equation*}

\paragraph{Special case: the arc is on a zero centered circle} ~\\

\begin{equation*}
\begin{array}{ll}
    f(x,y) &= x^2 + y^2 - r^2
    \\
    s(x,y) &= \atantwo(y,x)
    \\
    \mathring{f}(\rho,\theta) &= \rho-r
    \\
    \mathring{s}(\rho,\theta) &= \theta
\end{array}
\end{equation*}

\begin{equation*}
\begin{array}{ll}
    \boldsymbol{\Gamma}^{-1} 
    &= \left\lbrace \rho e^{i\theta} \big\vert \frac{1}{\rho}-r=0 , -\theta\in\boldsymbol{s}  \right\rbrace
    \\
    &= \left\lbrace \rho e^{i\theta} \big\vert \rho=\frac{1}{r} , -\theta\in\boldsymbol{s}  \right\rbrace
    \\
    &= \left\lbrace \rho e^{i\theta} \big\vert \sqrt{x^2+y^2}=\frac{1}{r} , -\atantwo(y,x)\in\boldsymbol{s}  \right\rbrace
    \\
    &= \left\lbrace \rho e^{i\theta} \big\vert x^2+y^2-\frac{1}{r^2}=0 , -\atantwo(y,x)\in\boldsymbol{s}  \right\rbrace
\end{array}
\end{equation*}

\vfill \eject

\section{Binary operations} \label{sup:binary}



We use the following algorithm in SageMath to compute the implicit equation of the envelope, as described in Subsection 3.3.5.

\begin{verbatim}
P.<x,y,u,v,Q,R,X,Y> = PolynomialRing(QQ, 8,order='degrevlex(4),degrevlex(4)')
    #(x,y), (u,v): coordinates of the two operands
    #Q,R: fixed parameters such as radii and slopes
    #(X,Y): coordinates of the result
def imageideal(f,g,o):
        #f,g: equations of the two operands in (x,y) and (u,v) respectively
        #o: operation, 0 is addition, 1 is multiplication
    if o==0:
        re=x+u
        im=y+v
            #addition
    else:
        re=x*u-y*v
        im=x*v+y*u
            #multiplication
    h=jacobian((f,g,re,im),(x,y,u,v)).det()
        #h: equation for the critical points
    I=ideal(f,g,h,re-X,im-Y)
    B=I.groebner_basis()
    Bred=[p for p in B if p(x=0,y=0,u=0,v=0)==p]
        #we eliminate the "heavy" variables, the Groebner basis of image ideal remains
    return len(Bred),Bred
        #returns the number of basis functions and their list
\end{verbatim}

\textit{\textbf{Operand equations}}

For 
$\boldsymbol{A},\boldsymbol{B} \in \mathcal{I}(\mathbb{C})$, 
$\boldsymbol{\Gamma} \subset \partial \boldsymbol{A}$,
$\boldsymbol{\Gamma'} \subset \partial \boldsymbol{B}$

\begin{flalign*}
\boldsymbol{\Gamma}&=\left\lbrace x+iy\big\vert f(x,y)=0,s(x,y)\in \boldsymbol{s}\right\rbrace \\&=\left\lbrace \rho e^{i\theta } \big\vert\mathring{f} \left(\rho ,\theta \right)=0,s^{\circ } \left(\rho ,\theta \right)\in \boldsymbol{s}\right\rbrace \\&=\left\lbrace F(s)\big\vert s\in \boldsymbol{s}\right\rbrace \\
\\ \boldsymbol{\Gamma'}&=\left\lbrace x+iy\big\vert g\left(x,y\right)=0,t\left(x,y\right)\in \boldsymbol{t}\right\rbrace \\&=\left\lbrace \rho e^{i\theta } \big\vert\mathring{g} \left(\rho ,\theta \right)=0,t^{\circ } \left(\rho ,\theta \right)\in \boldsymbol{t}\right\rbrace \\&=\left\lbrace G(t)\big\vert t\in \boldsymbol{t}\right\rbrace
\end{flalign*}

\vfill\eject

\subsection{Addition}\label{sup:addition}

\textit{\textbf{Parametric combination}}
\begin{flalign*}
H(s,t)&=\left(\Re \left(F(s)+G(t)\right),\Im \left(F(s)+G(t)\right)\right)=
\big(\Re(F)(s)+\Re(G)(t), \Im(F)(s)+\Im(G)(t)\big)\\
\\
J(s,t)&=\left| \begin{array}{cc}
\frac{\partial H_{\Re } \left(s,t\right)}{\partial s} & \frac{\partial H_{\Re } \left(s,t\right)}{\partial t}\\
\frac{\partial H_{\Im } \left(s,t\right)}{\partial s} & \frac{\partial H_{\Im } \left(s,t\right)}{\partial t}
\end{array} \right|
=
\left\vert \begin{array}{cc}
    \Re(F)'(s) & \Re(G)'(t) \\
    \Im(F)'(s) & \Im(G)'(t)
\end{array} \right\vert \\
J(s,t) & =0 \Longrightarrow \text{envelope}
\end{flalign*}

\textit{\textbf{Implicit combination}}
\begin{flalign*}
\varphi_\oplus (x,y,u,v) & = (x+u,y+v) \\
\widetilde{h}(x,y,u,v) & = \left\vert \mathrm{Jac} \big(f,g,\Re(\varphi_\oplus),\Im(\varphi_\oplus)\big) \right\vert \\
& =\left\vert
\begin{array}{cccc}
    \partial f/ \partial x & \partial f/ \partial y & 0 & 0 \\
    0 & 0 & \partial g/ \partial u & \partial g/ \partial v \\
    1 & 0 & 1 & 0 \\
    0 & 1 & 0 & 1 
\end{array}
\right\vert \\
I & =(f,g,\widetilde{h},\Re(\varphi_\oplus)-X,\Im(\varphi_\oplus)-Y) \\
h & = \text{Gröbner basis elements involving only $X,Y$}
\end{flalign*}

\textit{\textbf{Mixed combination}}
\begin{flalign*}
u(x,y,t)&=s\left(x-G^{\Re } \left(t\right),y-G^{\Im } \left(t\right)\right)\\
\\
\hat{h}(x,y,t)&=f\left(x-G^{\Re } \left(t\right),y-G^{\Im } \left(t\right)\right)\\
\frac{\partial \hat{h}}{\partial t}=0 & \Longrightarrow t\left(x,y\right)\\
h(x,y)&=\hat{h}(x,y,t(x,y))=0\\
\\
x\left(s,t\right)&=\Re \left(F(s)+G(t)\right)\\
y\left(s,t\right)&=\Im \left(F(s)+G(t)\right)\\
J(s,t)&=h\left(x\left(s,t\right),y\left(s,t\right)\right)
\end{flalign*}

\subsubsection{Edge plus edge}\label{sup:edge_plus_edge}

\begin{figure}[H]
\centering
\includegraphics[width=\textwidth,trim={0 0 0 10mm},clip]{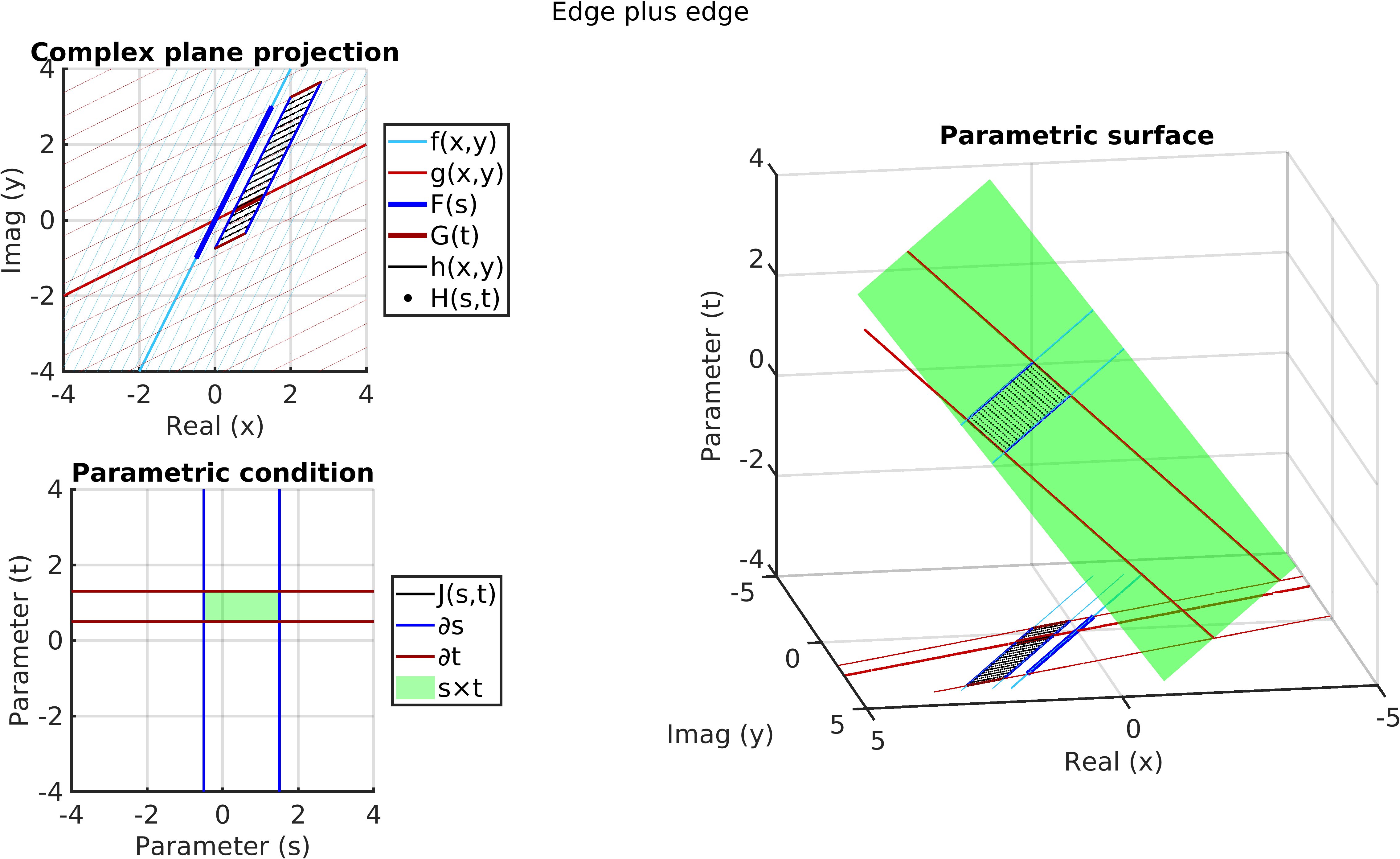}
\caption{Addition of two edges. Parameters: $a_1=2.0$, $a_2=0.5$, $s=(-0.5,1.5)$, $t=(0.5,1.3)$}
\label{fig:EdgePlusEdge}
\end{figure}

\textit{\textbf{Operand equations}}
\begin{flalign*}
f(x,y)&=a_1 x-y\\
s(x,y)&=x\\
F(s)&=s+i{a}_1 s=\left(s,a_1 s\right)\\
\\
g\left(x,y\right)&=a_2 x-y\\
t\left(x,y\right)&=x\\
G(t)&=t+{ia}_2 t=\left(t,a_2 t\right)
\end{flalign*}

\textit{\textbf{Parametric combination}}
\begin{flalign*}
H(s,t)&=\left(s+t,a_1 s+a_2 t\right)\\
J(s,t)&=\left\vert \begin{array}{cc}
1 & 1\\
a_1  & a_2 
\end{array}\right\vert
=a_2 -a_1 \\
\text{Envelope: }& 
\boxed{\left\lbrace \begin{array}{ll}
\text{all } (s,t) & \text{if } a_1 =a_2 \\
\emptyset  & \text{if } a_1 \neq a_2 
\end{array}\right.} 
\end{flalign*}

\textit{\textbf{Implicit combination}}
\begin{flalign*}
\text{Envelope: } &
\boxed{\left\lbrace \begin{array}{ll}
a_2 x-y & \text{if } a_1 =a_2 \\
1  & \text{if } a_1 \not= a_2
\end{array}\right.}\end{flalign*}

\textit{\textbf{Mixed combination}}
\begin{flalign*}
\hat{h}(x,y,t)
&=a_1 \left(x-t\right)-\left(y-a_2 t\right)\\
\frac{\partial \hat{h}}{\partial t}&=a_2 -a_1 \\
\frac{\partial \hat{h}}{\partial t}=0 &\Longrightarrow a_1 =a_2 
\\
\text{Envelope: }& 
\boxed{\left\lbrace \begin{array}{ll}
a_1 x-y=0 & \text{if } a_1 =a_2 \\
\emptyset  & \text{if } a_1 \neq a_2 
\end{array}\right. }
\end{flalign*}

\vfill \eject

\subsubsection{Arc plus edge}\label{sup:arc_plus_edge}

\begin{figure}[H]
\centering
\includegraphics[width=\textwidth,trim={0 0 0 10mm},clip]{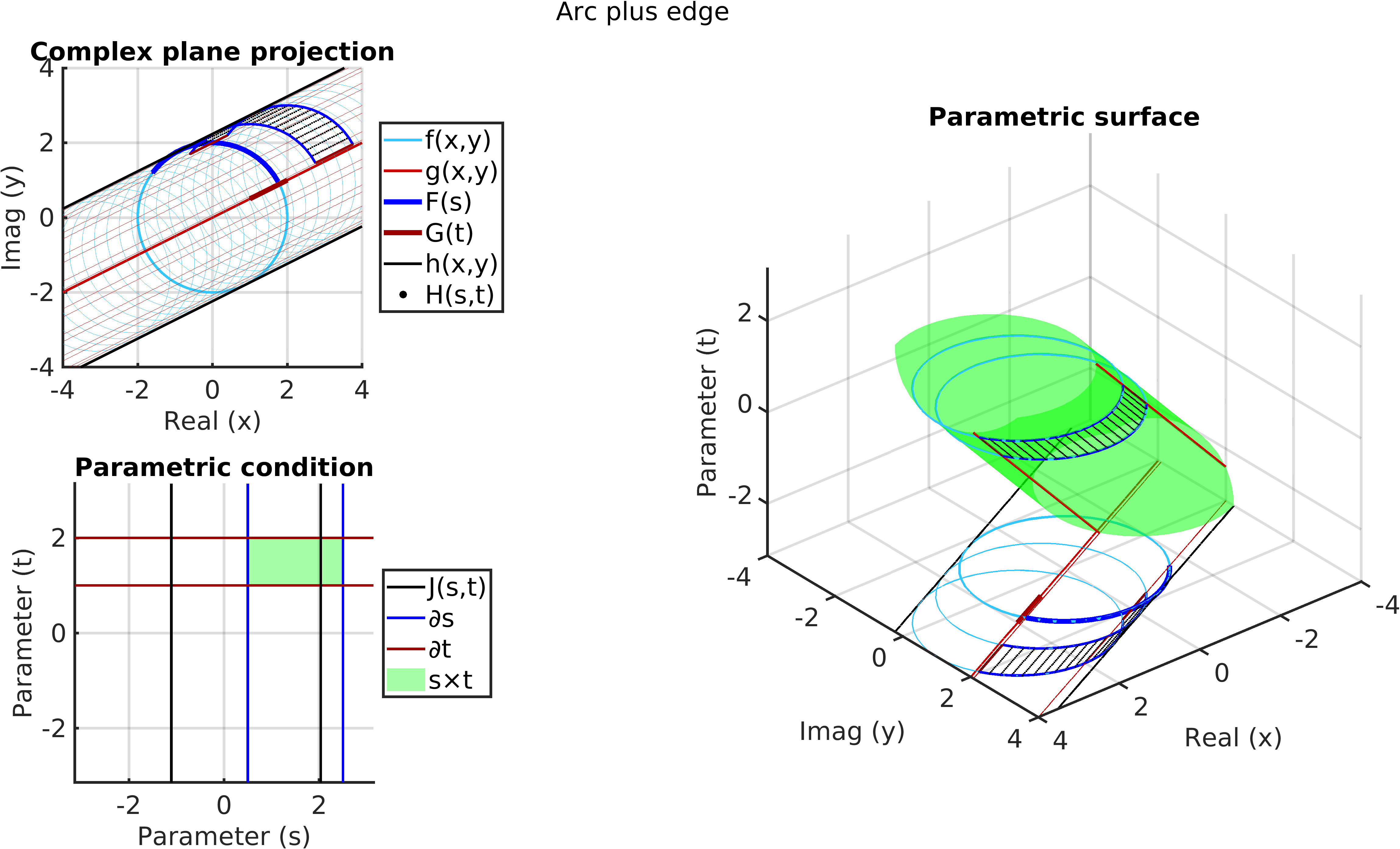}
\caption{Addition of an arc and an edge. Parameters: $r=2.0$, $a=0.5$, $s=(0.5,2.5)$, $t=(1.0,2.0)$}
\label{fig:EdgePlusArc}
\end{figure}

\textit{\textbf{Operand equations}}
\begin{flalign*}
f(x,y)&=x^2 +y^2 -r^2 \\
s(x,y)&=\atantwo\left(y,x\right)\\
F(s)&={re}^{is} =\left(r\cos (s),r\sin (s)\right)\\
\\ 
g\left(x,y\right)&=\mathrm{ax}-y\\
t\left(x,y\right)&=x\\
G(t)&=t+iat=\left(t,at\right)
\end{flalign*}

\textit{\textbf{Parametric combination}}
\begin{flalign*}
H(s,t)&=\left(r\cos (s)+t,r\sin (s)+at\right)\\
J(s,t)&=\left\vert \begin{array}{cc}
-r\sin (s) & 1\\
r\cos (s) & a
\end{array}\right\vert\\
&=-r\left(a\sin (s)+\cos (s)\right)\\
\text{Envelope: }&\boxed{ \tan \left(s\right)=-1/a}
\end{flalign*}

\textit{\textbf{Implicit combination}}
\begin{flalign*}
h(x,y)&=a^2 x^2 -a^2 r^2 -2axy+y^2 -r^2 \\
&={\left(ax-y\right)}^2 -r^2 \left(a^2 +1\right)\\
\text{Envelope: } &
\boxed{ax-y\pm r\sqrt{\left(a^2 +1\right)}=0}
\end{flalign*}

\textit{\textbf{Mixed combination}}
\begin{flalign*}
u(x,y,t)
&=\atantwo\left(y-at,x-t\right)
\\ \\
\hat{h}(x,y,t)
&={\left(x-t\right)}^2 +{\left(y-at\right)}^2 -r^2 \\
\frac{\partial \hat{h}}{\partial t}&=2t+2x+2a\left(y-at\right)\\
\frac{\partial \hat{h}}{\partial t} =0 & \Longrightarrow t\left(x,y\right)=\frac{x+\mathrm{ay}}{a^2 +1}\\
h(x,y)
&={\left(x-\frac{x+\mathrm{ay}}{a^2 +1}\right)}^2 +{\left(y-a\frac{x+\mathrm{ay}}{a^2 +1}\right)}^2 -r^2 \\
&={\left(\frac{a^2 x+x-x-\mathrm{ay}}{a^2 +1}\right)}^2 +{\left(\frac{a^2 y+y-\mathrm{ax}-a^2 y}{a^2 +1}\right)}^2 -r^2 \\
&={\left(a^2 x-ay\right)}^2 +{\left(y-\mathrm{ax}\right)}^2 -r^2 {\left(a^2 +1\right)}^2 \\
&=a^2 {\left(\mathrm{ax}-y\right)}^2 +{\left(ax-y\right)}^2 -r^2 {\left(a^2 +1\right)}^2 \\
&=\left(a^2 +1\right){\left(\mathrm{ax}-y\right)}^2 -r^2 {\left(a^2 +1\right)}^2 =0\\
&\Longrightarrow {\left(ax-y\right)}^2 -r^2 \left(a^2 +1\right)=0\\
\text{Envelope: } & \boxed{ax-y\pm r\sqrt{a^2 +1}=0}
\\ \\
x\left(s,t\right) &=r\cos (s)+t\\ 
y\left(s,t\right) &=r\sin (s)+at\\ 
J(s,t)
&={\left(ax\left(s,t\right)-y\left(s,t\right)\right)}^2 -r^2 \left(a^2 +1\right)\\
&={\left(a\left(r\cos (s)+t\right)-\left(r\sin (s)+at\right)\right)}^2 -r^2 \left(a^2 +1\right)\\
&={\left(ar\cos (s)-r\sin (s)\right)}^2 -r^2 \left(a^2 +1\right)=0\\
&\Longrightarrow {\left(ar\cos (s)-r\sin (s)\right)}^2 =r^2 \left(a^2 +1\right)\\
\text{Envelope: } & \boxed{\left(a\cos (s)-\sin (s)\right)^2 = a^2 +1}
\end{flalign*}

\vfill \eject

\subsubsection{Arc plus arc}\label{sup:arc_plus_arc}

\begin{figure}[H]
\centering
\includegraphics[width=\textwidth,trim={0 0 0 10mm},clip]{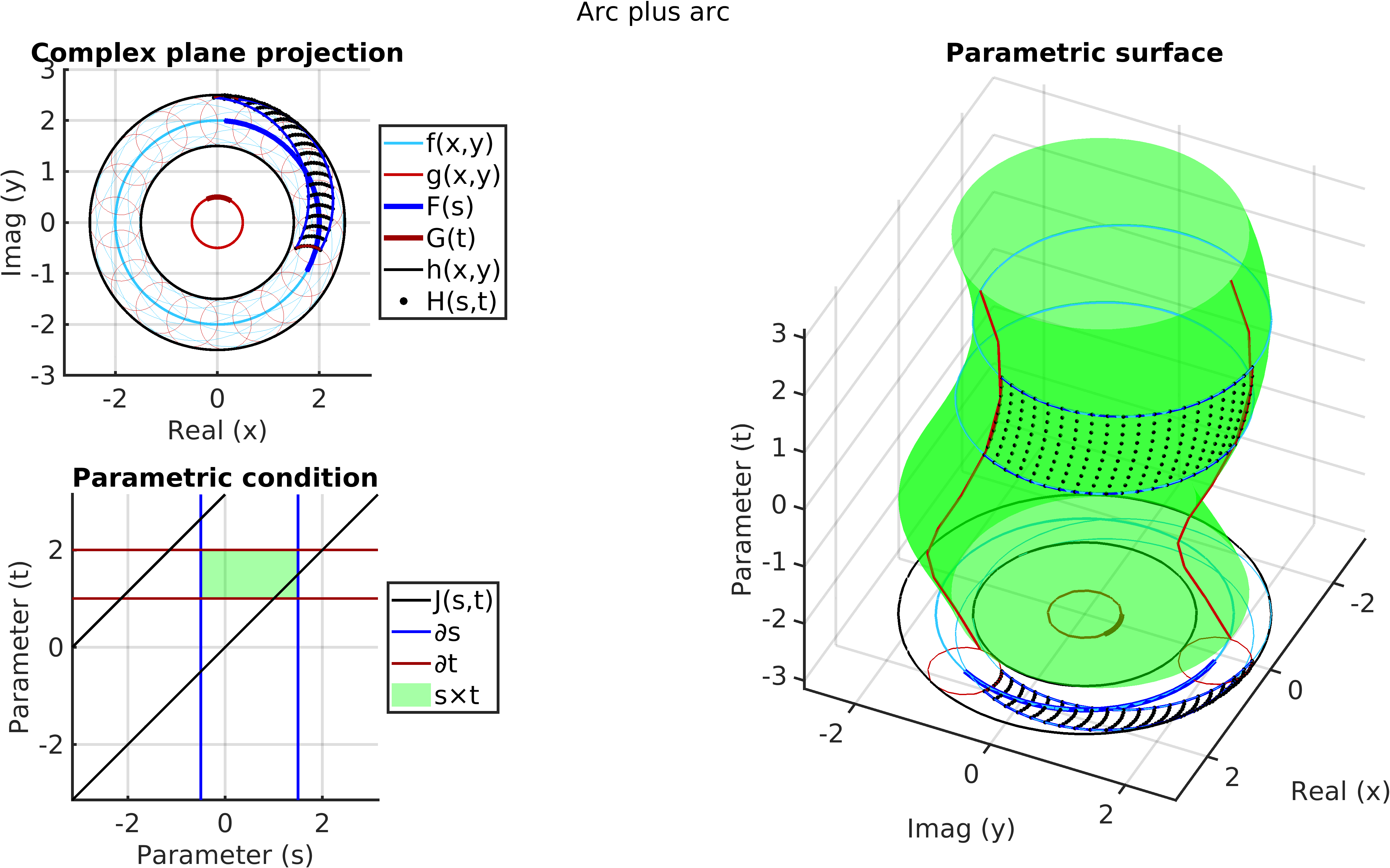}
\caption{Addition of two arcs. Parameters: $r_1=2.0$, $r_2=0.5$, $s=(-0.5,1.5)$, $t=(1.0,2.0)$}
\label{fig:ArcPlusArc}
\end{figure}

\textit{\textbf{Operand equations}}
\begin{flalign*}
f(x,y)&=x^2 +y^2 -r_1^2 \\
s(x,y)&=\atantwo\left(y,x\right)\\
F(s)&={r_1 e}^{is} =\left(r_1 \cos (s),r_1 \sin (s)\right)\\
\\ 
g\left(x,y\right)&=x^2 +y^2 -r_2^2 \\
t\left(x,y\right)&=\atantwo\left(y,x\right)\\
G(t)&={r_2 e}^{it} =\left(r_2 \cos (t),r_2 \sin (t)\right)
\end{flalign*}

\textit{\textbf{Parametric combination}}
\begin{flalign*}
H(s,t) &=\left(r_1 \cos (s)+r_2 \cos (t),r_1 \sin (s)+r_2 \sin (t)\right)
\\
J(s,t)&=\left| \begin{array}{cc}
-r_1 \sin (s) & -r_2 \sin (t)\\
r_1 \cos (s) & r_2 \cos (t)
\end{array} \right|\\
&=r_1 r_2 \left(\sin (t)\cos (s)-\sin (s)\cos (t)\right)\\
&=r_1 r_2 \sin (t-s) \\
\text{Envelope: } & \boxed{s=t+k\pi \; (k\in \mathbb{Z})}
\end{flalign*}

\textit{\textbf{Implicit combination}}
\begin{flalign*}
h(x,y)&=r_2^4 -2r_2^2 r_1^2 +r_1^4 -2r_2^2 x^2 -2r_1^2 x^2 +x^4 -2r_2^2 y^2 -2r_1^2 y^2 +2x^2 y^2 +y^4 \\
&=\left(x^2 +y^2 -{\left(r_1 +r_2 \right)}^2 \right)\left(x^2 +y^2 -{\left(r_1 -r_2 \right)}^2 \right)\\
\text{Envelope: } & \boxed{x^2 +y^2 -\left(r_1 \pm r_2 \right)^2=0}
\end{flalign*}

\textit{\textbf{Mixed combination}}
\begin{flalign*}
u(x,y,t)
&=\atantwo\left(y-r_2 \cos (t),x-r_2 \sin (t)\right)\\
\\
\hat{h}(x,y,t)
&={\left(x-r_2 \cos (t)\right)}^2 +{\left(y-r_2 \sin (t)\right)}^2 -r_1^2 \\
\frac{\partial \hat{h}}{\partial t}&=2r_2 x\sin (t)-2r_2^2 \cos (t)\sin (t)-2r_2 y\cos (t)+2r_2^2 \cos (t)\sin (t)\\
&=2r_2 \left(x\sin (t)-y\cos (t)\right)\\
\frac{\partial \hat{h}}{\partial t}=0&\Longrightarrow x\sin (t)=y\cos (t)\\ 
h(x,y)
&={\left(x-r_2 \cos \left(\atan\left(\frac{y}{x}\right)\right)\right)}^2 +{\left(y-r_2 \sin \left(\atan\left(\frac{y}{x}\right)\right)\right)}^2 -r_1^2 \\
&={{\left(x-\frac{r_2 \,x}{\sqrt{x^2 +y^2 }}\right)}}^2 +{{\left(y-\frac{r_2 \,y}{\sqrt{x^2 +y^2 }}\right)}}^2 -r_1^2 \\
&=x^2 -\frac{2{r_2 \,x}^2 }{\sqrt{x^2 +y^2 }}+\frac{{\left(r_2 \,x\right)}^2 }{x^2 +y^2 }+y^2 -\frac{2{r_2 \,y}^2 }{\sqrt{x^2 +y^2 }}+\frac{{\left(r_2 \,y\right)}^2 }{x^2 +y^2 }-r_1^2 \\
&=x^2 +y^2 -\frac{2r_2 \,\left(x^2 +y^2 \right)}{\sqrt{x^2 +y^2 }}+\frac{r_2^2 \left(x^2 +y^2 \right)}{x^2 +y^2 }-r_1^2 \\
&=x^2 +y^2 -2r_2 \sqrt{x^2 +y^2 }+r_2^2 -r_1^2 \\
&={\left(\sqrt{x^2 +y^2 }-r_2 \right)}^2 -r_1^2 =0\\
&\Longrightarrow \sqrt{x^2 +y^2 }-r_2 =\pm r_1 \\
&\Longrightarrow \sqrt{x^2 +y^2 }=r_2 \pm r_1 \\
\text{Envelope: } & \boxed{x^2 +y^2 -{{\left(r_1 \pm r_2 \right)}}^2 =0}\\
\\
x\left(s,t\right)
&=r_1 \cos (s)+r_2 \cos (t)\\
y\left(s,t\right)
&=r_1 \sin (s)+r_2 \sin (t)\\
J(s,t)
&={\left(r_1 \cos (s)+r_2 \cos (t)\right)}^2 +{\left(r_1 \sin (s)+r_2 \sin (t)\right)}^2 -{{\left(r_1 \pm r_2 \right)}}^2 \\
&=r_1^2 \cos^2 \left(s\right)+2r_1 r_2 \cos (s)\cos (t)+r_2^2 \cos^2 \left(t\right)+r_1^2 \sin^2 \left(s\right) \\ 
& \quad +2r_1 r_2 \sin (s)\sin (t)+r_2^2 \sin^2 \left(t\right)-{{\left(r_1 \pm r_2 \right)}}^2 \\
&=r_1^2 +r_2^2 +2r_1 r_2 \left(\cos (s)\cos (t)+\sin (s)\sin (t)\right)-r_1^2 \pm {2r}_1 r_2 -r_2^2 \\
&=2r_1 r_2 \cos \left(s-t\right)\pm {2r}_1 r_2
\\
J(s,t)=0 & \Longrightarrow \cos \left(s-t\right)=\pm 1\\
\text{Envelope: }&\boxed{s=t+k\pi \; (k\in \mathbb{Z}) }
\end{flalign*}

\vfill \eject

\subsection{Multiplication}\label{sup:multiplication}

\textit{\textbf{Parametric combination}}
\begin{flalign*}
H(s,t)&=\left(\Re \left(F(s)G(t)\right),\Im \left(F(s)G(t)\right)\right)\\
&=\big(\Re(F)(s)\Re(G)(t)-\Im(F)(s)\Im(G)(t),\Re(F)(s)\Im(G)(t)+\Im(F)(s)\Re(G)(t) \big)\\
\\
J(s,t)&=\left|\begin{array}{cc}
\frac{\partial H_{\Re } \left(s,t\right)}{\partial s} & \frac{\partial H_{\Re } \left(s,t\right)}{\partial t}\\
\frac{\partial H_{\Im } \left(s,t\right)}{\partial s} & \frac{\partial H_{\Im } \left(s,t\right)}{\partial t}
\end{array} \right| \\
&= \left\vert \begin{array}{cc}
    \Re(F)'(s)\Re(G)(t)-\Im(F)'(s)\Im(G)(t) & \Re(F)(s)\Re(G)'(t)-\Im(F)(s)\Im(G)'(t) \\
    \Re(F)'(s)\Im(G)(t)+\Im(F)'(s)\Re(G)(t) & \Re(F)(s)\Im(G)'(t)+\Im(F)(s)\Re(G)'(t)
\end{array} \right\vert\\
J(s,t) & =0 \Longrightarrow \text{envelope}
\end{flalign*}

\textit{\textbf{Implicit combination}}
\begin{flalign*}
\varphi_\otimes (x,y,u,v) & = (xu-yv,xv+yu) \\
\widetilde{h}(x,y,u,v) & =  \left\vert \mathrm{Jac} \big(f,g,\Re(\varphi_\otimes),\Im(\varphi_\otimes)\big) \right\vert \\
& =\left\vert
\begin{array}{cccc}
    \partial f/ \partial x & \partial f/ \partial y & 0 & 0 \\
    0 & 0 & \partial g/ \partial u & \partial g/ \partial v \\
    u & -v & x & -y \\
    v & u & y & x 
\end{array}
\right\vert \\
I & =(f,g,\widetilde{h},\Re(\varphi_\otimes)-X,\Im(\varphi_\otimes)-Y) \\
h & = \text{Gröbner basis elements involving only $X,Y$}
\end{flalign*}

\textit{\textbf{Mixed combination}}
\begin{flalign*}
u^{\circ } (\rho ,\theta ,t)&=s^{\circ } \left(\rho /|G(t)|,\theta -\angle G(t)\right)\\
u(x,y,t)&=u^{\circ } \left(\sqrt{x^2 +y^2 },\atantwo\left(y,x\right),t\right)\\
\\
h(\rho,\theta)&=\mathring{f} \left(\frac{\rho }{|G(t)|},\theta -\angle G(t)\right)\\
\hat{h}(x,y,t)&=h\left(\rho \left(x,y\right),\theta \left(x,y\right)\right)\\
\frac{\partial \hat{h}}{\partial t}=0 
&\Longrightarrow t(x,y)\\
h(x,y)&=\hat{h}(x,y,t(x,y))\\
\\
x\left(s,t\right)&=\Re \left(F(s)G(t)\right)\\
y\left(s,t\right)&=\Im \left(F(s)G(t)\right)\\
J(s,t)&=h\left(x\left(s,t\right),y\left(s,t\right)\right)
\end{flalign*}

\vfill\eject
\subsubsection{Edge times edge}\label{sup:edge_times_edge}

\paragraph{General case} ~\\

\begin{figure}[H]
\centering
\includegraphics[width=\textwidth,trim={0 0 0 10mm},clip]{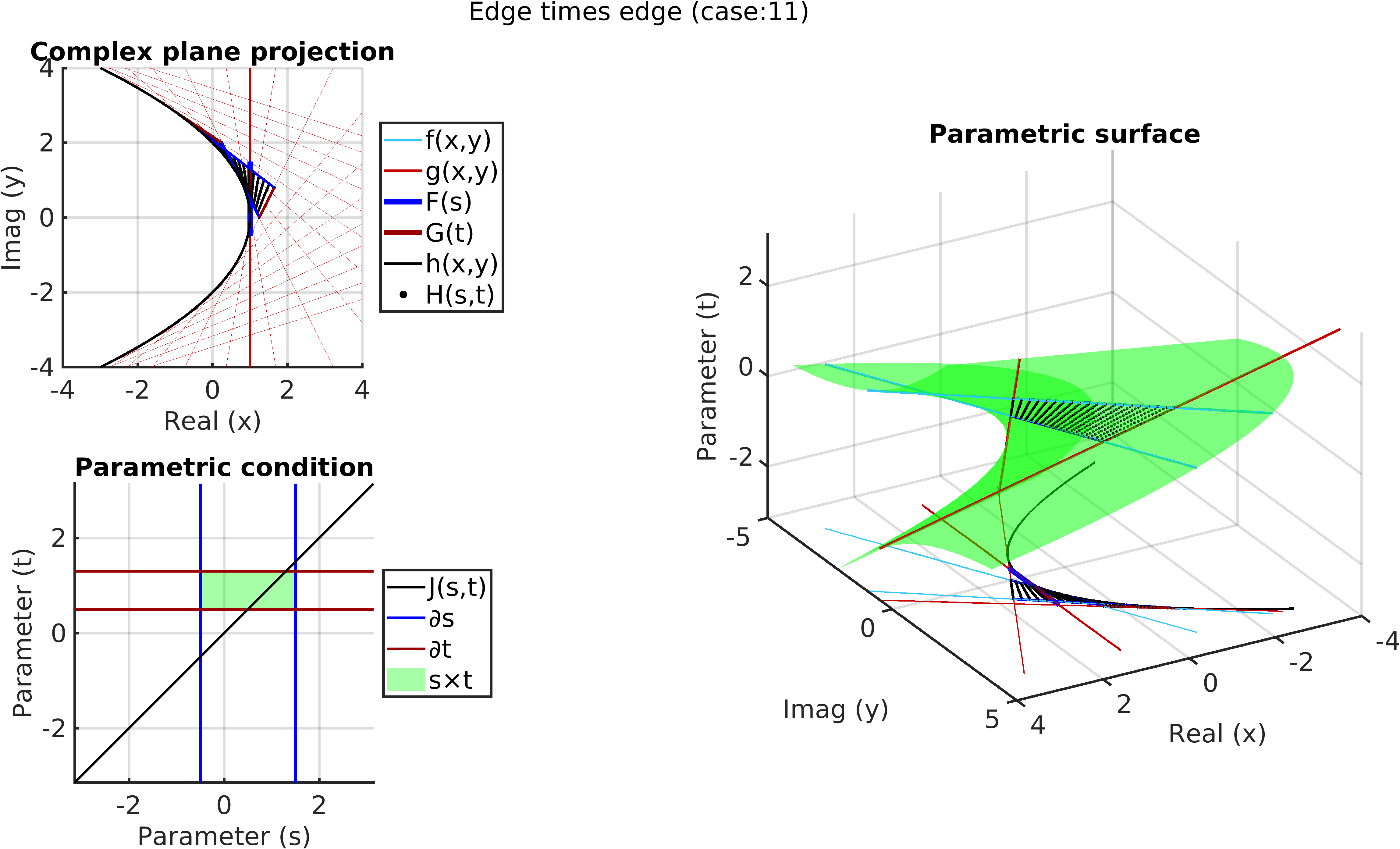}
\caption{Multiplication of two edges in the general case. Parameters: $s=(-0.5,1.5)$, $t=(0.5,1.3)$.}
\label{fig:EdgeTimesEdge11}
\end{figure}

\textit{\textbf{Operand equations}}
\begin{flalign*}
f(x,y)&=x-1\\ \mathring{f} \left(\rho ,\theta \right)&=1/\rho -\cos \left(\theta \right)\\
s(x,y)&=y\\ s^{\circ } \left(\rho ,\theta \right)&=\rho \sin \left(\theta \right)\\
F(s)&=1+is=\left(1,s\right)\\
\\ 
g\left(x,y\right)&=x-1
\\
{g}^{\circ } \left(\rho ,\theta \right)&=1/\rho -\cos \left(\theta \right)\\
t\left(x,y\right)&=y\\ t^{\circ } \left(\rho ,\theta \right)&=\rho \sin \left(\theta \right)\\
G(t)&=1+\mathrm{it}=\left(1,t\right)
\end{flalign*}

\textit{\textbf{Parametric combination}}
\begin{flalign*}
H(s,t)
&=\left(1-\mathrm{st},s+t\right)\\
J(s,t)&=\left| \begin{array}{cc}
-t & -s\\
1 & 1
\end{array} \right| \\
&=s-t\\
\text{Envelope: } & \boxed{s=t}
\end{flalign*}

\textit{\textbf{Implicit combination}}
\begin{flalign*}
h(x,y)&= y^2+4x-4\\
\text{Envelope: } & \boxed{x =-y^2/4 +1}
\end{flalign*}

\textit{\textbf{Mixed combination}}
\begin{flalign*}
u^{\circ } (\rho ,\theta ,t)
&=\frac{\rho }{\sqrt{1+t^2 }}\sin \left(\theta -\atantwo\left(t,1\right)\right)\\
&=\frac{\rho }{\sqrt{1+t^2 }}\sin \left(\theta -\atan\left(t\right)\right)\\
u(x,y,t)
&=\frac{\sqrt{x^2 +y^2 }}{\sqrt{1+t^2 }}\sin \left(\atantwo\left(y,x\right)-\atan\left(t\right)\right)\\
u(x,y,t)=0
&\Longrightarrow \atantwo\left(y,x\right)-\atan\left(t\right)=0+k\pi \\
\\
h(\rho,\theta)
&=\frac{\sqrt{t^2 +1}}{\rho }-\cos \left(\theta -\atantwo\left(t,1\right)\right)\\
&=\frac{\sqrt{t^2 +1}}{\rho }-\cos \left(\theta -\atan\left(t\right)\right)\\
\hat{h}(x,y,t)
&=\frac{\sqrt{t^2 +1}}{\sqrt{x^2 +y^2 }}-\cos \left(\atantwo\left(y,x\right)-\atan\left(t\right)\right)=0\\
&=\frac{\sqrt{t^2 +1}}{\sqrt{x^2 +y^2 }}-\cos \left(\atantwo\left(y,x\right)\right)\cos \left(\atan\left(t\right)\right)-\sin \left(\atantwo\left(y,x\right)\right)\sin \left(\atan\left(t\right)\right)=0\\
&=\frac{\sqrt{t^2 +1}}{\sqrt{x^2 +y^2 }}-\frac{x}{\sqrt{x^2 +y^2 }}\cdot \frac{1}{\sqrt{t^2 +1}}-\frac{y}{\sqrt{x^2 +y^2 }}\cdot \frac{t}{\sqrt{t^2 +1}}=0\\
&=t^2 +1-x-\mathrm{yt}\\
\frac{\partial \hat{h}}{\partial t}&=2t-y\\
\frac{\partial \hat{h}}{\partial t}=0
&\Longrightarrow t=\frac{y}{2}\\
h(x,y)&=\frac{y^2 }{4}+1-x-\frac{y^2 }{2}\\
h(x,y)=0 &\Longrightarrow x+\frac{y^2 }{4}-1=0\\
\text{Envelope: } & \boxed{x =-y^2/4 +1}
\\
x(s,t)
&=1-\mathrm{st}\\
y(s,t)
&=s+t\\
J(s,t)
&=h\left(1-st,s+t\right)\\
&=1-st+\frac{{\left(s+t\right)}^2 }{4}-1\\
J(s,t)=0
&\Longrightarrow {\left(s+t\right)}^2 =4st\\
&\Longrightarrow s^2 +2st+t^2 =4st\\
&\Longrightarrow s^2 -2st+t^2 =0\\
&\Longrightarrow {\left(s-t\right)}^2 =0\\
&\Longrightarrow s=t\\
\text{Envelope: } & \boxed{s=t}
\end{flalign*}

\vfill \eject

\paragraph{Special case: one edge is on a zero crossing line} ~\\

\begin{figure}[H]
\centering
\includegraphics[width=\textwidth,trim={0 0 0 10mm},clip]{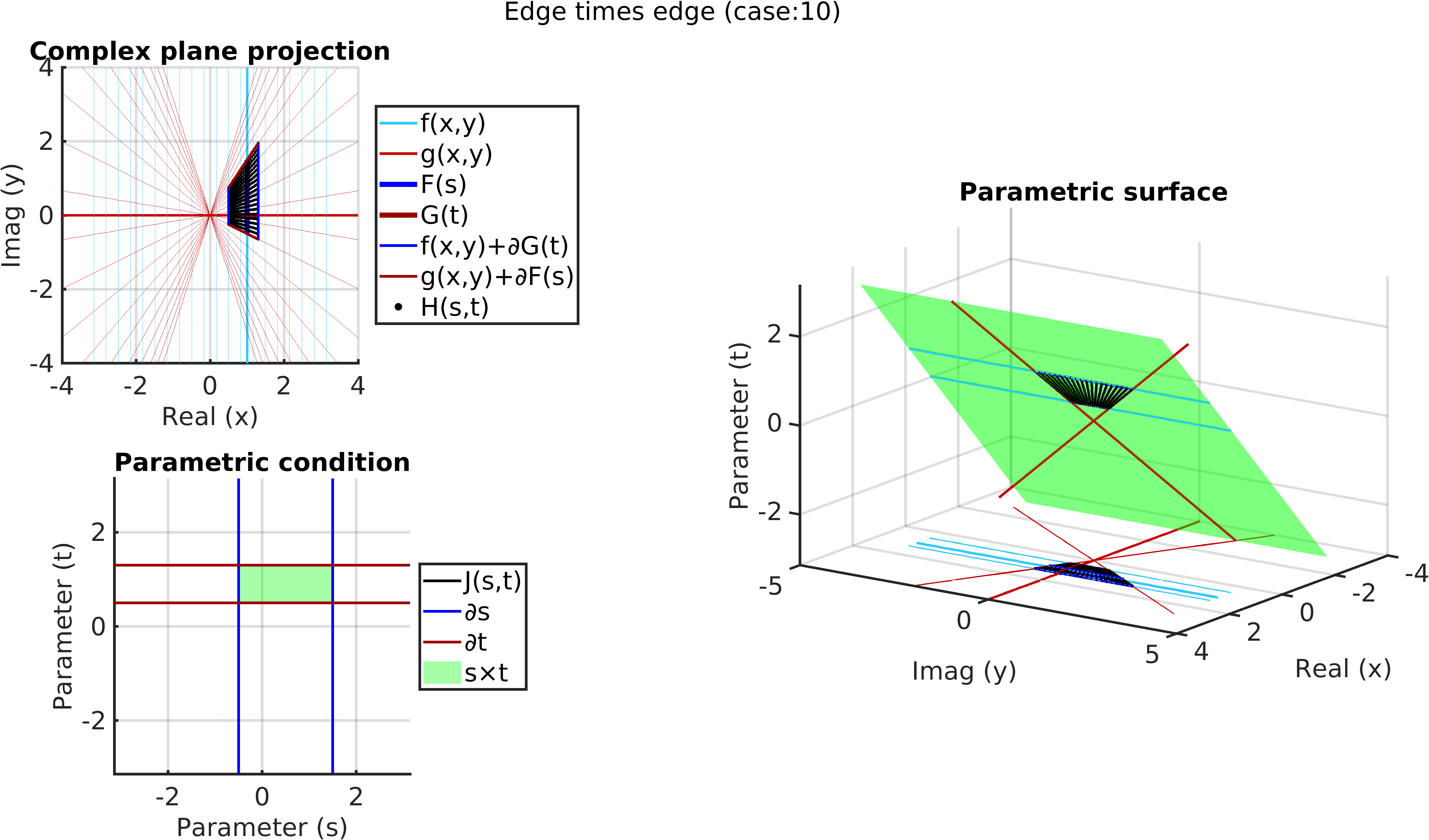}
\caption{Multiplication of two edges when both are on zero crossing lines. Parameters: $s=(-0.5,1.5)$, $t=(0.5,1.3)$.}
\label{fig:EdgeTimesEdge10}
\end{figure}

\textit{\textbf{Operand equations}}
\begin{flalign*}
f(x,y)&=x-1
\\ 
\mathring{f} \left(\rho ,\theta \right)&=1/\rho -\cos \left(\theta \right)\\
s(x,y)&=y
\\
s^{\circ } \left(\rho ,\theta \right)&=\rho \sin \left(\theta \right)\\
F(s)&=1+is=\left(1,s\right)\\
\\
g\left(x,y\right)&=y
\\
{g}^{\circ } \left(\rho ,\theta \right)&=\tan \left(\theta \right)\\
t\left(x,y\right)&=x
\\
t^{\circ } \left(\rho ,\theta \right)&=\rho \\
G(t)=&t=\left(t,0\right)
\end{flalign*}

\textit{\textbf{Parametric combination}}
\begin{flalign*}
H(s,t)
&=\left(s,st\right)\\
\\
J(s,t)&=\left| \begin{array}{cc}
1 & 0\\
t & s
\end{array} \right| \\
&=s\\
\text{Envelope: } & \boxed{s=0}
\end{flalign*}

\textit{\textbf{Implicit combination}}
\begin{flalign*}
h_1(x,y) &=x \\
h_2(x,y) &=y \\
\text{Envelope: } & \boxed{(0,0) \text{ (i.e. the origin)}} 
\end{flalign*}

\textit{\textbf{Mixed combination}}
\begin{flalign*}
u^{\circ } (\rho ,\theta ,t)
&=\frac{\rho }{t}\sin \left(\theta \right)\\
\\
u(x,y,t)
&=\frac{\sqrt{x^2 +y^2 }}{t}\sin \left(\mathrm{atan2}\left(y,x\right)\right)\\
\\
h(\rho,\theta)
&=\frac{t}{\rho }-\cos \left(\theta \right)\\
\\
\hat{h}(x,y,t)
&=\frac{t}{\sqrt{x^2 +y^2 }}-\cos \left(\mathrm{atan2}\left(y,x\right)\right)\\
&=\frac{t-x}{\sqrt{x^2 +y^2 }}\\
\hat{h}(x,y,t)=0 &\Longrightarrow t=x\\
\\
\frac{\partial \hat{h}}{\partial t}&=1\\
\\
h(x,y)&=\hat{h}(x,y,t(x,y))=\emptyset 
\\
\text{Envelope: } & \boxed{\emptyset}
\end{flalign*}

\vfill\eject

\paragraph{Special case: both edges are on zero crossing lines.} ~\\

\begin{figure}[H]
\centering
\includegraphics[width=\textwidth,trim={0 0 0 10mm},clip]{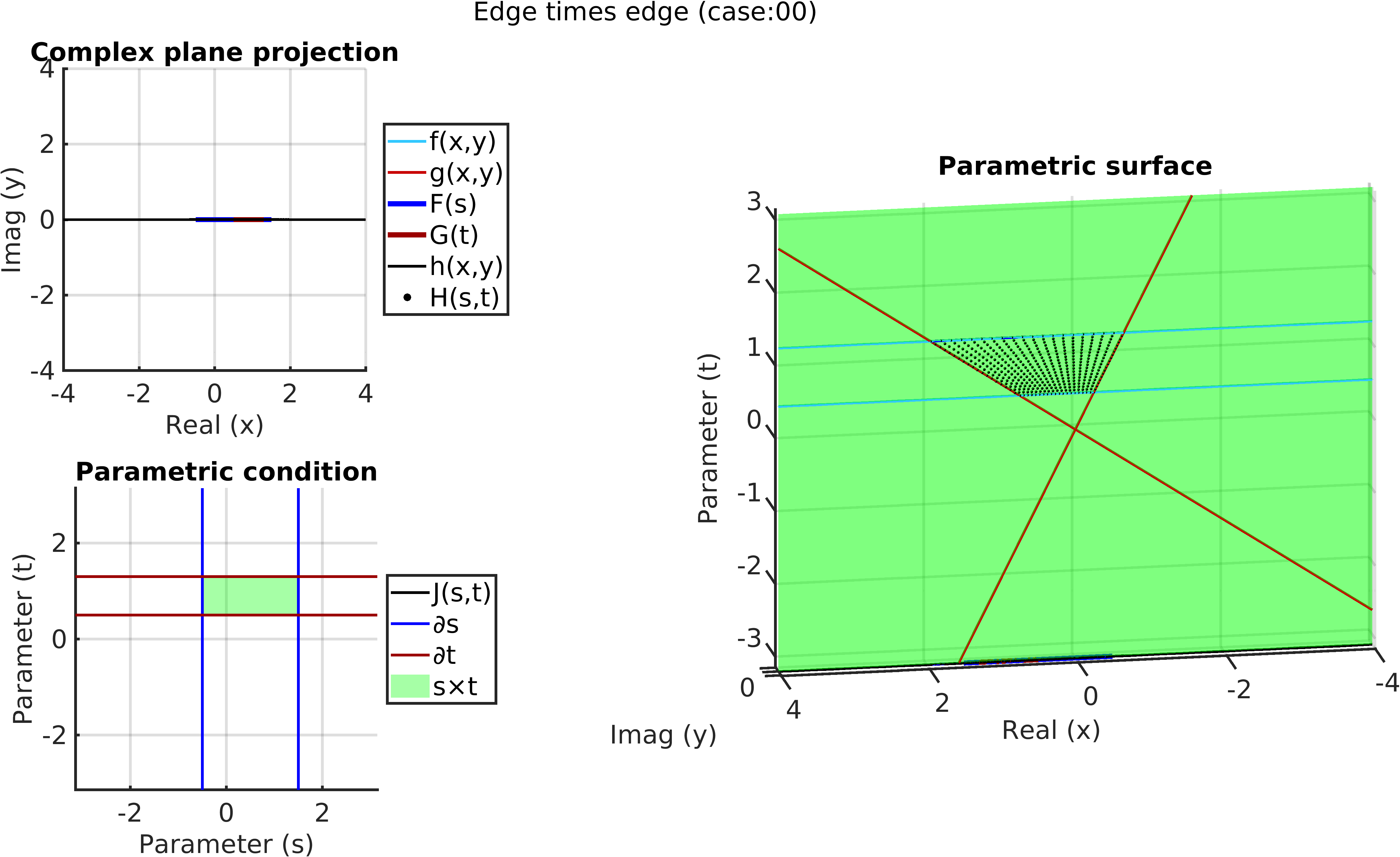}
\caption{Multiplication of two edges when both are on zero crossing lines. Parameters: $s=(-0.5,1.5)$, $t=(0.5,1.3)$.}
\label{fig:EdgeTimesEdge00}
\end{figure}

\textit{\textbf{Operand equations}}
\begin{flalign*}
f(x,y)&=y\\ \mathring{f} \left(\rho ,\theta \right)&=\tan \left(\theta \right)\\
s(x,y)&=x\\ s^{\circ } \left(\rho ,\theta \right)&=\rho \\
F(s)&=s=\left(s,0\right)\\
\\ 
g\left(x,y\right)&=y
\\
{g}^{\circ } \left(\rho ,\theta \right)&=\tan \left(\theta \right)\\
t\left(x,y\right)&=x\\ t^{\circ } \left(\rho ,\theta \right)&=\rho \\
G(t)&=t=\left(t,0\right)
\end{flalign*}

\textit{\textbf{Parametric combination}}
\begin{flalign*}
H(s,t)
&=\left(st,0\right)
\\
J(s,t)
&=\left| \begin{array}{cc}
t & s\\
0 & 0
\end{array} \right| \\
&=0 \\
\text{Envelope: } & \boxed{\text{all } (s,t)}
\end{flalign*}

\textit{\textbf{Implicit combination}}
\begin{flalign*}
h(x,y) &=y \\
\text{Envelope: } & \boxed{y=0}
\end{flalign*}

\textit{\textbf{Mixed combination}}
\begin{flalign*}
u^{\circ } (\rho ,\theta ,t)
&=\frac{\rho }{t}\\
\\
u(x,y,t)
&=\frac{\sqrt{x^2 +y^2 }}{t}\\
\\
h(\rho,\theta)
&=\tan \left(\frac{\theta }{t}\right)\\
\\
\hat{h}(x,y,t)
&=\tan \left(\frac{\atantwo\left(y,x\right)}{t}\right)\\
\hat{h}(x,y,t)=0
&\Longrightarrow \atantwo\left(y,x\right)=\atan\left(0\right)\cdot t\\
&\Longrightarrow \atantwo\left(y,x\right)=0+k\pi \\
&\Longrightarrow y=0\\
\\
\frac{\partial \hat{h}}{\partial t}&=0\\
\\
h(x,y)&=y\\
\text{Envelope: } & \boxed{y=0}
\end{flalign*}

\vfill \eject

\vfill \eject

\subsubsection{Arc times edge}\label{sup:arc_times_edge}

\paragraph{General case} ~\\

\begin{figure}[H]
\centering
\includegraphics[width=\textwidth,trim={0 0 0 10mm},clip]{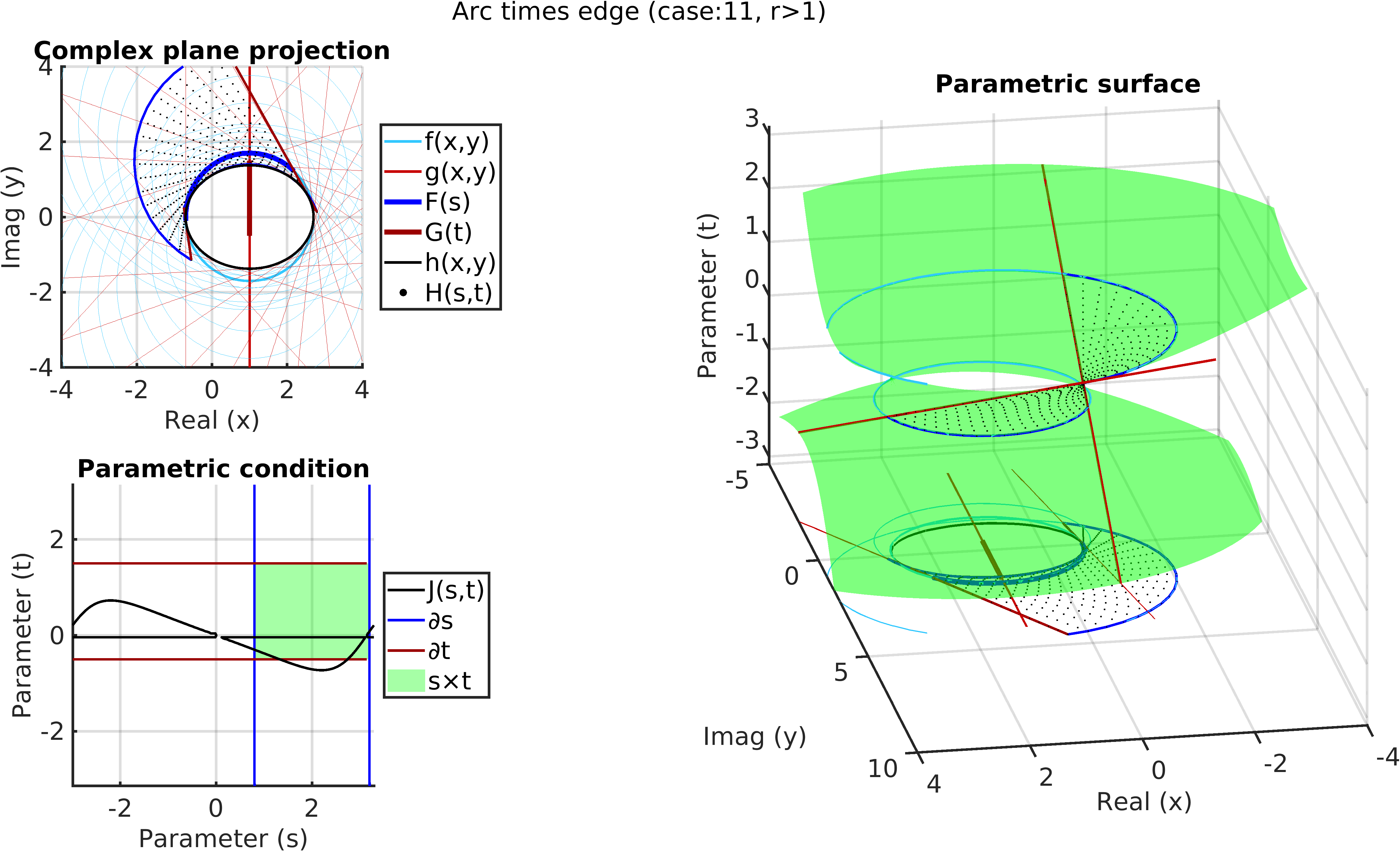}
\caption{Multiplication of an arc and an edge when neither the edge is on a zero crossing line, nor the arc is on a zero centered circle and the radius is greater than one. Parameters: $r=1.7$, $s=(0.8,3.2)$, $t=(-0.5,1.5)$.}
\label{fig:ArcTimesEdge11rg1}
\end{figure}

\begin{figure}[H]
\centering
\includegraphics[width=\textwidth,trim={0 0 0 10mm},clip]{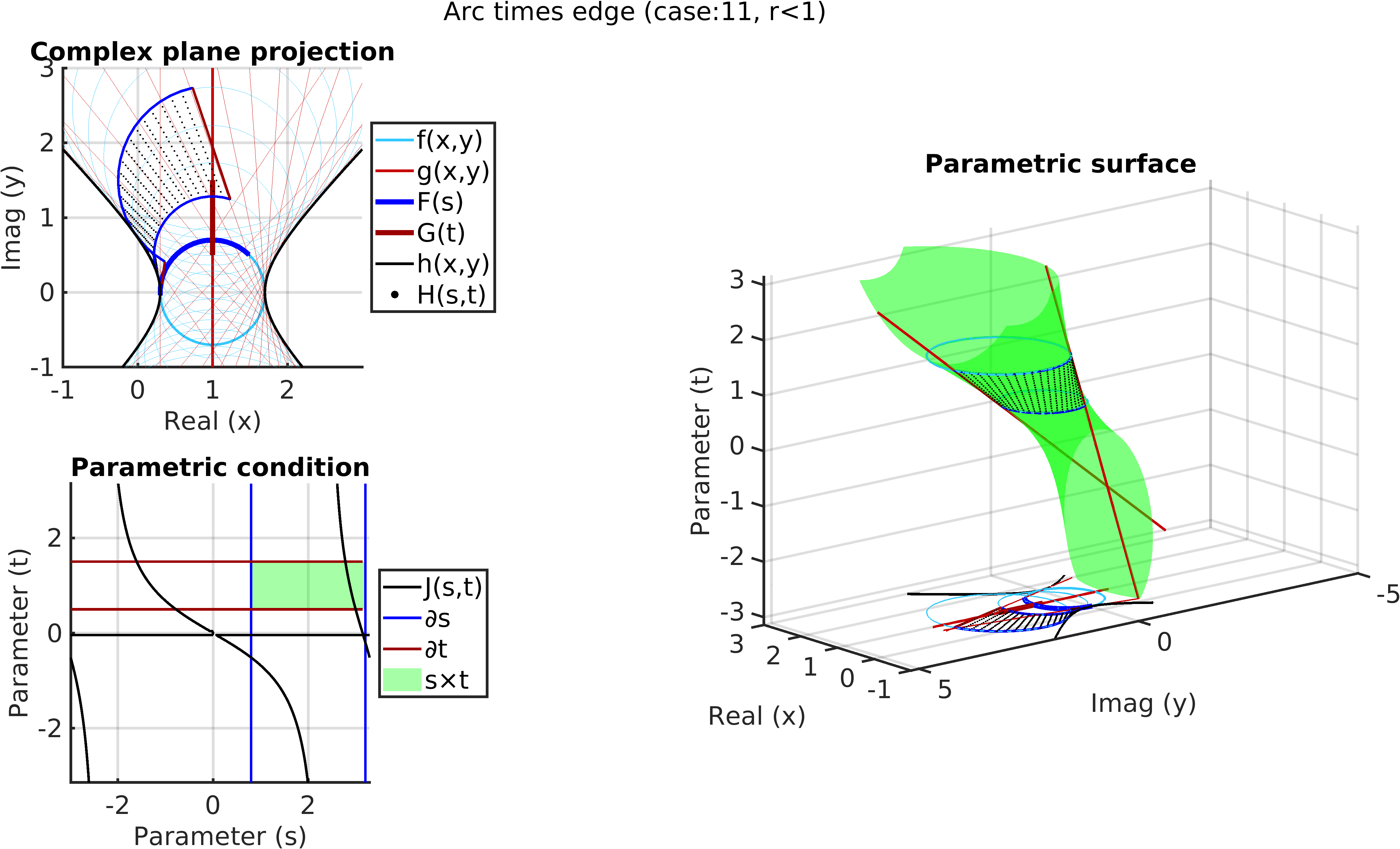}
\caption{Multiplication of an arc and an edge when neither the edge is on a zero crossing line, nor the arc is on a zero centered circle and the radius is less than one. Parameters: $r=0.7$, $s=(0.8,3.2)$, $t=(0.5,1.5)$.}
\label{fig:ArcTimesEdge11rl1}
\end{figure}

\begin{figure}[H]
\centering
\includegraphics[width=\textwidth,trim={0 0 0 10mm},clip]{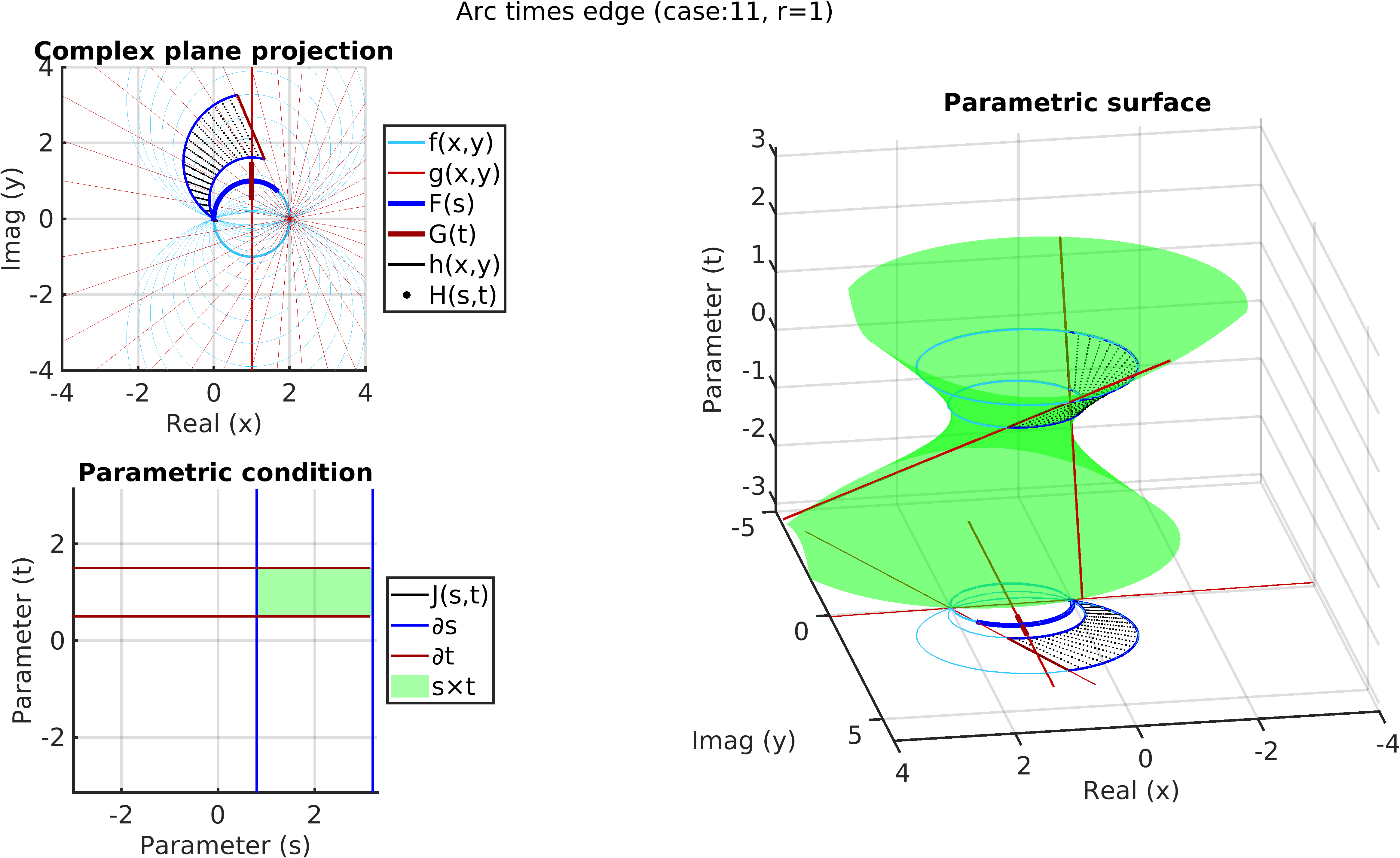}
\caption{Multiplication of an arc and an edge when neither the edge is on a zero crossing line, nor the arc is on a zero centered circle and the radius is equal to one. Parameters: $r=1$, $s=(0.8,3.2)$, $t=(-0.5,1.5)$.}
\label{fig:ArcTimesEdge11re1}
\end{figure}

\textit{\textbf{Operand equations}}
\begin{flalign*}
f(x,y)&={\left(x-1\right)}^2 +y^2-r^2 \\ \mathring{f} \left(\rho ,\theta \right)&=\rho^2 -2\rho \cos \left(\theta \right)+1-r^2 \\
s(x,y)&=\atantwo\left(y,x-1\right)\\ s^{\circ } \left(\rho ,\theta \right)&=\atantwo\left(\rho \sin \left(\theta \right),\rho \cos \left(\theta \right)-1\right)\\
F(s)&={re}^{is} +1=\left(1+r\cos (s),r\sin (s)\right)\\
\\ 
g\left(x,y\right)&=x-1
\\
{g}^{\circ } \left(\rho ,\theta \right)&=1/\rho -\cos \left(\theta \right)\\
t\left(x,y\right)&=y\\ t^{\circ } \left(\rho ,\theta \right)&=\rho \sin \left(\theta \right)\\
G(t)&=1+\mathrm{it}=\left(1,t\right)
\end{flalign*}

\textit{\textbf{Parametric combination}}
\begin{flalign*}
H(s,t)
&=\left(1+r\cos (s)-tr\sin (s),t+rt\cos (s)+r\sin (s)\right)\\
\\
J(s,t)
&=\left\vert \begin{array}{cc}
   -r \sin(s)-rt \cos(s)  & -r \sin (s) \\
   -rt \sin (s)+r \cos (s)  & 1+r \cos (s)
\end{array} \right\vert
\\
&=-r\sin (s)-rt\cos (s)-r^2 t \\
\text{Envelope: } &  \boxed{\sin (s)+t\cos (s)+tr=0}
\end{flalign*}

\textit{\textbf{Implicit combination}}
\begin{flalign*}
\text{Envelope: } & \boxed{r^4 - r^2x^2 - r^2y^2 + 2r^2 x - 2r^2 + x^2 - 2x + 1=0}
\end{flalign*}

\textit{\textbf{Mixed combination}}
\begin{flalign*}
u^{\circ } (\rho ,\theta ,t)
&=\atantwo\left(\frac{\rho }{\sqrt{1+t^2 }}\sin \left(\theta -\atantwo\left(t,1\right)\right),\frac{\rho }{\sqrt{1+t^2 }}\cos \left(\theta -\atantwo\left(t,1\right)\right)-1\right)\\
&=\atantwo\left(\frac{\rho }{\sqrt{1+t^2 }}\sin \left(\theta -\atan\left(t\right)\right),\frac{\rho }{\sqrt{1+t^2 }}\cos \left(\theta -\atan\left(t\right)\right)-1\right)\\
u(x,y,t)
&=\atantwo\left(\frac{\sqrt{x^2 +y^2 }}{\sqrt{1+t^2 }}\sin \left(\atantwo\left(y,x\right)-\atan\left(t\right)\right),\frac{\sqrt{x^2 +y^2 }}{\sqrt{1+t^2 }}\cos \left(\atantwo\left(y,x\right)-\atan\left(t\right)\right)-1\right)\\
u(x,y,t)=0
&\Longrightarrow \cos \left(\atantwo\left(y,x\right)\right)\cos \left(\atan\left(t\right)\right)+\sin \left(\atantwo\left(y,x\right)\right)\sin \left(\atan\left(t\right)\right)=\frac{\sqrt{1+t^2 }}{\sqrt{x^2 +y^2 }}\\
&\Longrightarrow \frac{x}{\sqrt{x^2 +y^2 }}\cdot \frac{t}{\sqrt{1+t^2 }}+\frac{y}{\sqrt{x^2 +y^2 }}\cdot \frac{1}{\sqrt{1+t^2 }}=\frac{\sqrt{1+t^2 }}{\sqrt{x^2 +y^2 }}\\
&\Longrightarrow xt+y=1+t^2 \\
\\
h(\rho,\theta)
&=\frac{\rho^2 }{1+t^2 }-\frac{2\rho }{\sqrt{1+t^2 }}\cos \left(\theta -\atan\left(t\right)\right)+1-r^2 \\
h(\rho,\theta)=0
&\Longrightarrow \frac{\rho^2 }{\sqrt{1+t^2 }}-2\rho \cos \left(\theta -\atan\left(t\right)\right)+\sqrt{1+t^2 }\left(1-r^2 \right)\\
\hat{h}(x,y,t)
&=\frac{x^2 +y^2 }{\sqrt{1+t^2 }}-2\sqrt{x^2 +y^2 }\cos \left(\atantwo\left(y,x\right)-\atan\left(t\right)\right)+\sqrt{1+t^2 }\left(1-r^2 \right)\\
&=\frac{x^2 +y^2 }{\sqrt{1+t^2 }}-2\sqrt{x^2 +y^2 }\left(\cos \left(\atantwo\left(y,x\right)\right)\cos \left(\atan\left(t\right)\right)+\sin \left(\atantwo\left(y,x\right)\right)\sin \left(\atan\left(t\right)\right)\right)
\\ & +\sqrt{1+t^2 }\left(1-r^2 \right)\\
\hat{h}(x,y,t)=0
&\Longrightarrow \frac{x^2 +y^2 }{\sqrt{1+t^2 }}-2\sqrt{x^2 +y^2 }\left(\frac{x}{\sqrt{x^2 +y^2 }}\cdot \frac{1}{\sqrt{t^2 +1}}+\frac{y}{\sqrt{x^2 +y^2 }}\cdot \frac{t}{\sqrt{t^2 +1}}\right)\\
&\quad\quad+\sqrt{1+t^2 }\left(1-r^2 \right)=0\\
&\Longrightarrow \frac{x^2 +y^2 }{\sqrt{1+t^2 }}-2\frac{x+yt}{\sqrt{t^2 +1}}+\sqrt{1+t^2 }\left(1-r^2 \right)=0\\
&\Longrightarrow x^2 +y^2 -2\left(xt+y\right)+\left(1+t^2 \right)\left(1-r^2 \right)=0\\
&\Longrightarrow {\left(x-1\right)}^2 +{\left(y-t\right)}^2 -r^2 \left(1+t^2 \right)=0\\
\frac{\partial \hat{h}}{\partial t}&=2t+2y-2r^2 t\\
\frac{\partial \hat{h}}{\partial t}=0
&\Longrightarrow t\left(1-r^2 \right)=y\\
&\Longrightarrow t=\frac{y}{1-r^2 }\\
h(x,y)
&={\left(x-1\right)}^2 +{\left(y-\frac{y}{1-r^2 }\right)}^2 -r^2 \left(1+{\left(\frac{y}{1-r^2 }\right)}^2 \right)\\
&={\left(x-1\right)}^2 +\frac{{\left(y\left(1-r^2 \right)-y\right)}^2 }{{\left(1-r^2 \right)}^2 }-r^2 \frac{{\left(1-r^2 \right)}^2 +y^2 }{{\left(1-r^2 \right)}^2 }\\
h(x,y)=0
&\Longrightarrow {\left(x-1\right)}^2 {\left(1-r^2 \right)}^2 +y^2 {\left(1-r^2 \right)}^2 -2y^2 \left(1-r^2 \right)+y^2 -r^2 {\left(1-r^2 \right)}^2 -r^2 y^2 =0\\
&\Longrightarrow {\left(x-1\right)}^2 {\left(1-r^2 \right)}^2 +y^2 {\left(1-r^2 \right)}^2 -2y^2 \left(1-r^2 \right)-r^2 {\left(1-r^2 \right)}^2 +y^2 \left(1-r^2 \right)=0\\
&\Longrightarrow \left\lbrack {\left(x-1\right)}^2 +y^2 -r^2 \right\rbrack {\left(1-r^2 \right)}^2 -y^2 \left(1-r^2 \right)=0\\
&\Longrightarrow {\left(x-1\right)}^2 +y^2 -r^2 -\frac{y^2 }{1-r^2 }=0\\
&\Longrightarrow {\left(x-1\right)}^2 -\frac{r^2 y^2 }{1-r^2 }-r^2 =0\\
\mathrm{if}\; r=1\\
h(x,y)&=\left(1-r^2 \right){\left(x-1\right)}^2 -r^2 y^2 -\left(1-r^2 \right)r^2 \\
h(x,y)=0
&\Longrightarrow y=0\\
h\left(x,0\right)&={\left(x-1\right)}^2 -r^2 \\
h(x,y)=0
&\Longrightarrow x=1\pm r=\left\lbrace 0,2\right\rbrace \\
\text{Envelope: } & 
\boxed{\left\lbrace \begin{array}{ll}
{\left(x-1\right)}^2 -\frac{r^2 y^2 }{1-r^2 }-r^2 =0 & \text{if } r \neq 0 \\
x=\{0,2\}  & \text{if } r=1 
\end{array}\right.} 
\end{flalign*}

\begin{flalign*}
x(s,t)
&=1+r\cos (s)-tr\sin (s),\\
y(s,t)
&=t+tr\cos (s)+r\sin (s)\\
J(s,t)
&=h\left(1+r\cos (s)-tr\sin (s),t+tr\cos (s)+r\sin (s)\right)\\
&={\left(1+r\cos (s)-tr\sin (s)-1\right)}^2 -\frac{r^2 {\left(t+tr\cos (s)+r\sin (s)\right)}^2 }{1-r^2 }-r^2 \\
J(s,t)=0
&\Longrightarrow \left(1-r^2 \right)r^2 {\left(\cos (s)-t\sin (s)\right)}^2 -r^2 {\left(t+r\left(t\cos (s)+\sin (s)\right)\right)}^2 -\left(1-r^2 \right)r^2 =0\\
&\Longrightarrow \left(1-r^2 \right)\left(\cos^2 \left(s\right)-2t\sin (s)\cos (s)+t^2 \sin^2 \left(s\right)-1\right)-{\left(t+r\left(t\cos (s)+\sin (s)\right)\right)}^2 =0\\
&\Longrightarrow \cos^2 \left(s\right)-2t\sin (s)\cos (s)+t^2 \sin^2 \left(s\right)-1-r^2 \cos^2 \left(s\right)-2r^2 t\sin (s)\cos (s) \\
&+r^2 t^2 \sin^2 \left(s\right)-r^2-t^2 -2rt^2 \cos (s)+2rt\sin (s)-r^2 {\left(t\cos (s)+\sin (s)\right)}^2 =0\\
&\Longrightarrow \cos^2 \left(s\right)-2t\sin (s)\cos (s)+t^2 \sin^2 \left(s\right)-1-r^2 \cos^2 \left(s\right)-2r^2 t\sin (s)\cos (s)+r^2 t^2 \sin^2 \left(s\right) \\
&-r^2-t^2 -2t^2 r\cos (s)+2tr\sin (s)-r^2 t^2 \cos^2 \left(s\right)+2r^2 \mathrm{tsin}\left(s\right)\cos (s)+r^2 \sin^2 \left(s\right)=0\\
&\Longrightarrow r^2 \,t^2 +2\,r\,t^2 \,\cos (s)+2\,r\,t\sin (s)\,+t^2 \,{\cos (s)}^2 +2t\,\sin (s)\,\cos (s)+\sin^2 \left(s\right)=0\\
&\Longrightarrow r^2 \,t^2 +2\,r\,t\,\left(\sin (s)+t\cos (s)\right)+{\left(\sin (s)+t\cos (s)\right)}^2 =0\\
&\Longrightarrow {\left\lbrack rt+\left(\sin (s)+t\cos (s)\right)\right\rbrack }^2 =0\\
&\Longrightarrow \frac{\sin (s)}{t}+\cos (s)+r=0\\
\text{Envelope: } &  \boxed{\sin (s)+t\cos (s)+tr=0}
\end{flalign*}


\paragraph{Special case: the edge is on a zero crossing line} ~\\

\begin{figure}[H]
\centering
\includegraphics[width=\textwidth,trim={0 0 0 10mm},clip]{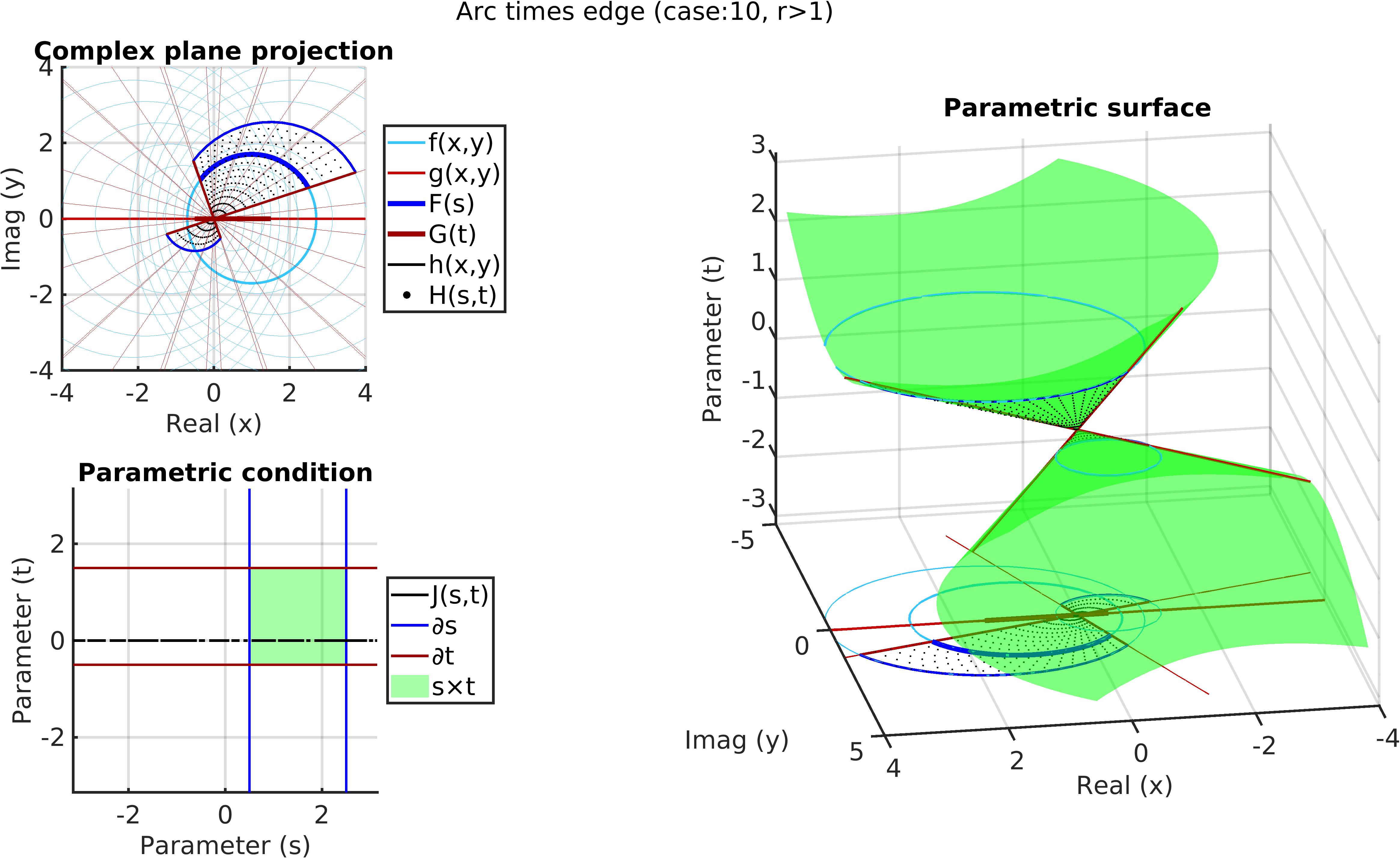}
\caption{Multiplication of an arc and an edge when the edge is on a zero crossing line, but the arc is not on a zero centered circle and the radius is greater than one. Parameters: $r=1.7$, $s=(0.5,2.5)$, $t=(-0.5,1.5)$.}
\label{fig:ArcTimesEdge10rg1}
\end{figure}

\begin{figure}[H]
\centering
\includegraphics[width=\textwidth,trim={0 0 0 10mm},clip]{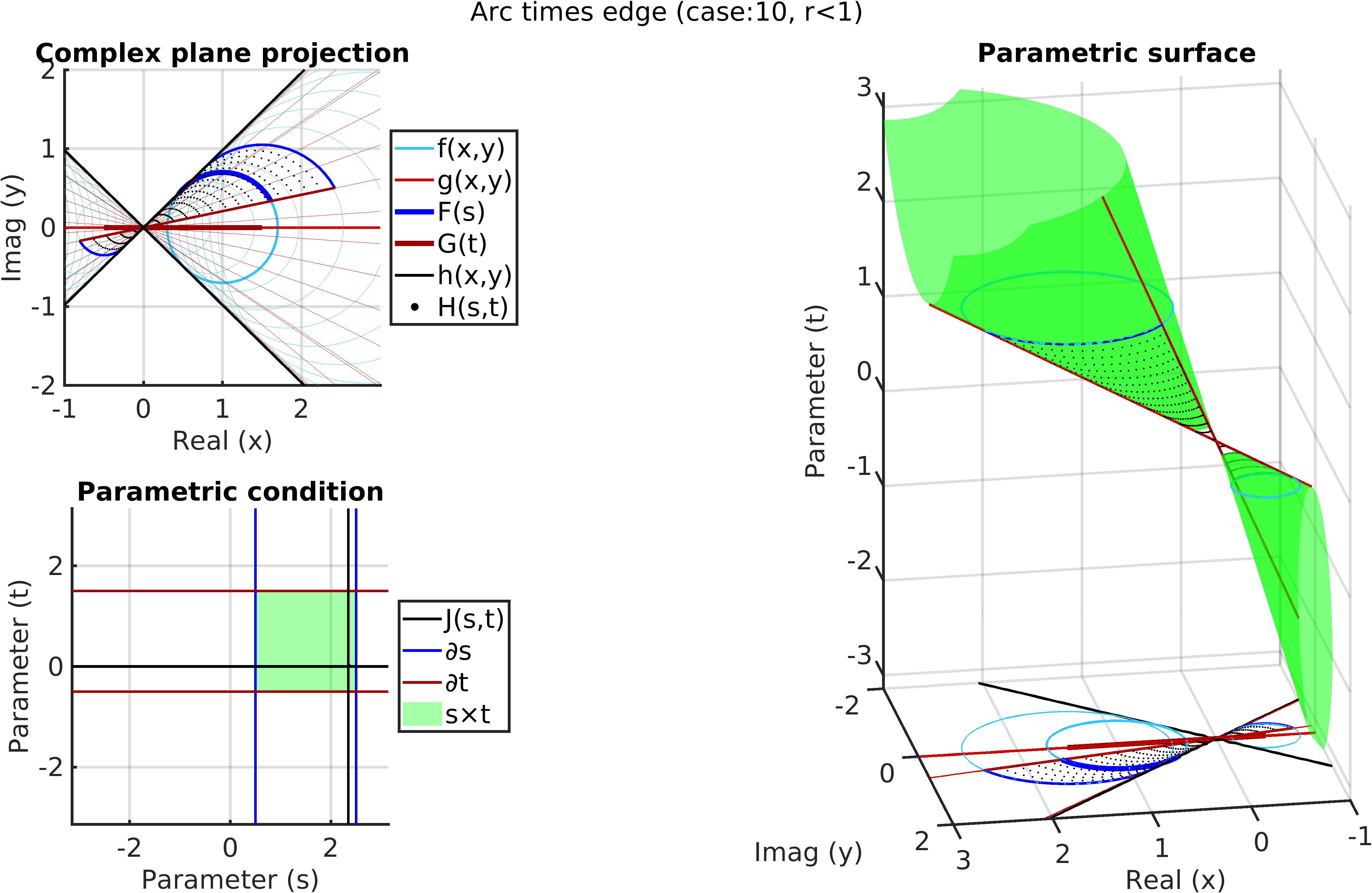}
\caption{Multiplication of an arc and an edge when the edge is on a zero crossing line, but the arc is not on a zero centered circle and the radius is less than one. Parameters: $r=0.7$, $s=(0.5,2.5)$, $t=(-0.5,1.5)$.}
\label{fig:ArcTimesEdge10rl1}
\end{figure}

\begin{figure}[H]
\centering
\includegraphics[width=\textwidth,trim={0 0 0 10mm},clip]{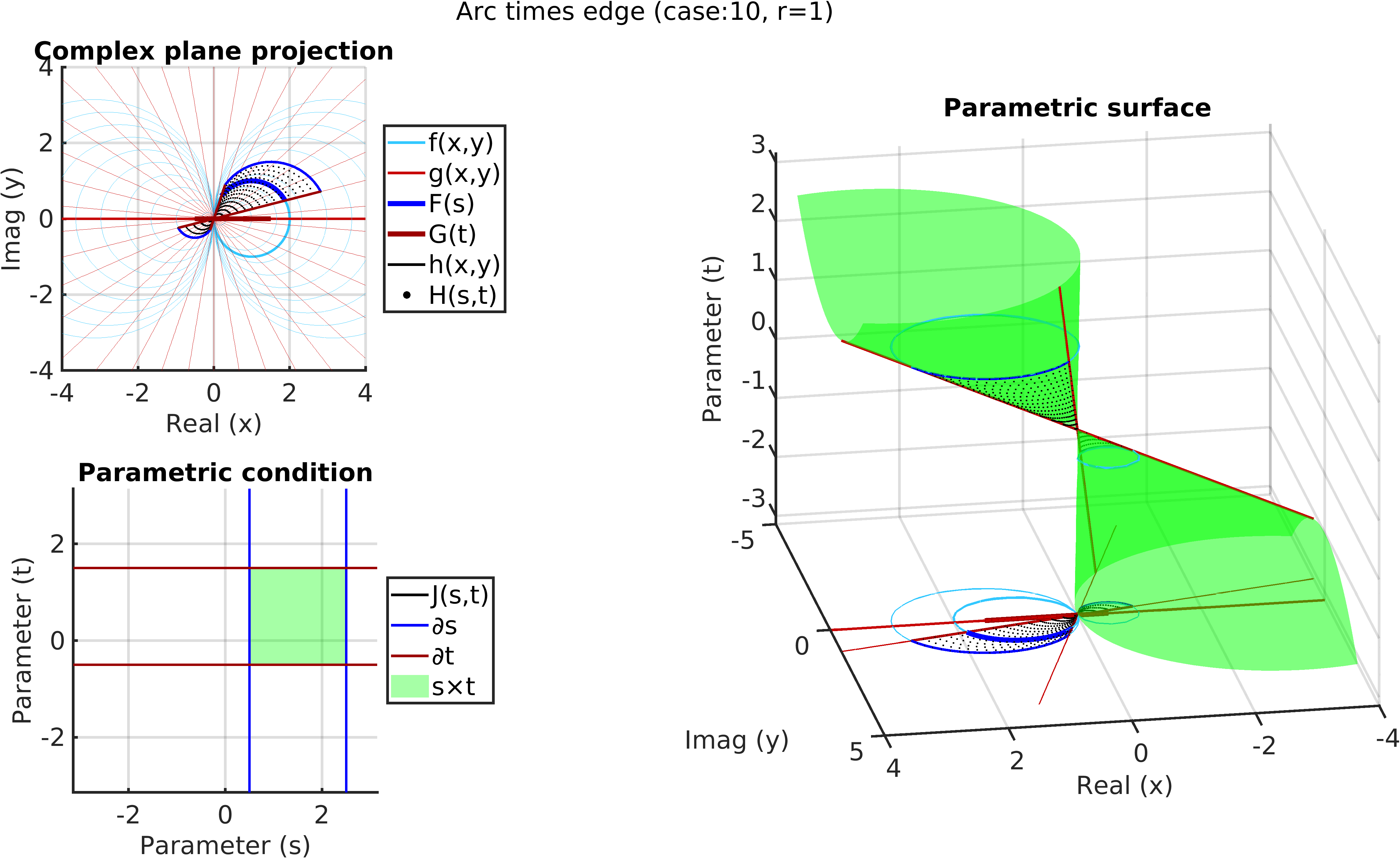}
\caption{Multiplication of an arc and an edge when the edge is on a zero crossing line, but the arc is not on a zero centered circle and the radius is equal to one. Parameters: $r=1$, $s=(0.5,2.5)$, $t=(-0.5,1.5)$.}
\label{fig:ArcTimesEdge10re1}
\end{figure}

\textit{\textbf{Operand equations}}
\begin{flalign*}
f(x,y)&={\left(x-1\right)}^2 +y^2-r^2 \\ \mathring{f} \left(\rho ,\theta \right)&=\rho^2 -2\rho \cos \left(\theta \right)+1-r^2 \\
s(x,y)&=\atantwo\left(y,x-1\right)\\ s^{\circ } \left(\rho ,\theta \right)&=\atantwo\left(\rho \sin \left(\theta \right),\rho \cos \left(\theta \right)-1\right)\\
F(s)&={re}^{is} +1=\left(1+r\cos (s),r\sin (s)\right)\\
\\ 
g\left(x,y\right)&=y
\\
{g}^{\circ } \left(\rho ,\theta \right)&=\tan \left(\theta \right)\\
t\left(x,y\right)&=x\\ t^{\circ } \left(\rho ,\theta \right)&=\rho \\
G(t)&=t=\left(t,0\right)
\end{flalign*}

\textit{\textbf{Parametric combination}}
\begin{flalign*}
H(s,t)
&=\left(t+rt\cos (s),rt\sin (s)\right)\\
\\
J(s,t)&=\left\vert \begin{array}{cc}
   -rt \sin(s)  & 1+r \cos(s) \\
   rt \cos (s)  & r \sin(s)
\end{array} \right\vert
\\
&= -r^2t \sin^2(s)-r^2t \cos^2(s) -rt\cos(s)\\
&=-rt(r+\cos (s)) \\
\text{Envelope: } &  \boxed{\left\lbrace \begin{array}{ll}
t=0 & \text{if } r>1 \\
\{\cos(s)=-r\} \cup \{t=0\}  & \text{if } r \le 1 
\end{array}\right.} 
\end{flalign*}

\textit{\textbf{Implicit combination}}
\begin{flalign*}
h(x,y) & = r^2x^2 + r^2y^2 - y^2 \\
& = r^2x^2 + (r^2-1)y^2 \\
\text{Envelope: } & \boxed{\left\lbrace \begin{array}{ll}
x=y=0 & \text{if } r>1 \\
x=\pm \frac{\sqrt{1-r^2}}{r} y  & \text{if } r \le 1 
\end{array}\right.}
\end{flalign*}

\textit{\textbf{Mixed combination}}
\begin{flalign*}
u^{\circ } (\rho ,\theta ,t)
&=\atantwo\left(\frac{\rho }{t}\sin \left(\theta \right),\frac{\rho }{t}\cos \left(\theta \right)-1\right)\\
\\
u(x,y,t)
&=\atantwo\left(\frac{\sqrt{x^2 +y^2 }}{t}\sin \left(\atantwo\left(y,x\right)\right),\frac{\sqrt{x^2 +y^2 }}{t}\cos \left(\atantwo\left(y,x\right)\right)-1\right)\\
u(x,y,t) = 0
&\Longrightarrow \frac{\sqrt{x^2 +y^2 }}{t}\cos \left(\atantwo\left(y,x\right)\right)-1=0,\frac{y}{t}>0\\
&\Longrightarrow \sqrt{x^2 +y^2 }\frac{x}{\sqrt{x^2 +y^2 }}=t\\
&\Longrightarrow x=t,\frac{y}{t}>0\\
h(\rho,\theta)
&=\frac{\rho^2 }{t^2 }-2\frac{\rho }{t}\cos \left(\theta \right)+1-r^2\\
h(\rho,\theta)=0
&\Longrightarrow \rho^2 -2\rho t\cos \left(\theta \right)+t^2 \left(1-r^2 \right)=0\\
\hat{h}(x,y,t)
&=x^2 +y^2 -2\sqrt{x^2 +y^2 }t\cos \left(\atantwo\left(y,x\right)\right)+t^2 \left(1-r^2 \right)\\
&=x^2 +y^2 -2tx+t^2 -r^2 t^2\\
\hat{h}(x,y,t)=0
&\Longrightarrow {\left(x-t\right)}^2 +y^2 -r^2 t^2 =0\\
\frac{\partial \hat{h}}{\partial t}&=2\left(t-x\right)-2{r}^2 t\\
\frac{\partial \hat{h}}{\partial t} = 0 
&\Longrightarrow t-x-r^2 t=0\\
&\Longrightarrow t\left(1-r^2 \right)=x\\
&\Longrightarrow t=\frac{x}{1-r^2 }\\
h(x,y)
&={\left(x-\frac{x}{1-r^2 }\right)}^2 +y^2 -r^2 {\left(\frac{x}{1-r^2 }\right)}^2 \\
&=\frac{{\left(-r^2 x\right)}^2 }{{\left(1-r^2 \right)}^2 }+y^2 -\frac{r^2 x^2 }{{\left(1-r^2 \right)}^2 }\\
h(x,y)=0
&\Longrightarrow r^4 x^2 +y^2 {\left(1-r^2 \right)}^2 -r^2 x^2 =0\\
&\Longrightarrow y^2 {\left(1-r^2 \right)}^2 =r^2 x^2 \left(1-r^2 \right)\\
&\Longrightarrow y^2 \left(1-r^2 \right)=r^2 x^2 \\
&\Longrightarrow \frac{r^2 }{1-r^2 }x^2 -y^2 =0\\
& \Longrightarrow x=\pm \frac{\sqrt{1-r^2}}{r} y\\
\mathrm{if}\;r=1\\
h(x,y)&=r^2 x^2 -y^2 \left(1-r^2 \right)\\
h(x,y)=0
&\Longrightarrow x=0\\
h(x,0)&=x^2 -\frac{\left(1-r^2 \right)}{r^2 }y^2 \\
h(x,0)=0
&\Longrightarrow y=0 \\
\text{Envelope: } & \boxed{\left\lbrace \begin{array}{ll}
x=y=0 & \text{if } r\geq1 \\
x=\pm \frac{\sqrt{1-r^2}}{r} y  & \text{if } r < 1 
\end{array}\right.}
\end{flalign*}

\begin{flalign*}
x(s,t)
&=t+rt\cos (s)\\
y(s,t)
&=rt\sin (s)\\
J(s,t)
&=h\left(t+rt\cos (s),rt\sin (s)\right)\\
&=\frac{r^2 }{1-r^2 }{\left(t+rt\cos (s)\right)}^2 -{\left(rt\sin (s)\right)}^2 \\
J(s,t)=0
&\Longrightarrow r^2 \left(t^2 +2rt^2 \cos (s)+r^2 t^2 \cos^2 \left(s\right)\right)=\left(1-r^2 \right)r^2 t^2 \sin^2 \left(s\right)\\
&\Longrightarrow t^2 +2rt^2 \cos (s)+r^2 t^2 \cos^2 \left(s\right)=t^2 \sin^2 \left(s\right)-r^2 t^2 \sin^2 \left(s\right)\\
&\Longrightarrow t^2 +2rt^2 \cos (s)+r^2 t^2 =t^2 \sin^2 \left(s\right)\\
&\Longrightarrow 1+2r\cos (s)+r^2 =1-\cos^2 \left(s\right),t=0\\
&\Longrightarrow \cos^2 \left(s\right)+2r\cos (s)+r^2 =0\\
&\Longrightarrow {\left(\cos (s)+r\right)}^2 =0\\
&\Longrightarrow s=\mathrm{acos}\left(-r\right),t=0\\
\text{Envelope: } &  \boxed{\left\lbrace \begin{array}{ll}
t=0 & \text{if } r>1 \\
\{\cos(s)=-r\} \cup \{t=0\}  & \text{if } r \le 1 
\end{array}\right.} 
\end{flalign*}

\vfill \eject

\paragraph{Special case: the arc is on a zero centered circle} ~\\

\begin{figure}[H]
\centering
\includegraphics[width=\textwidth,trim={0 0 0 10mm},clip]{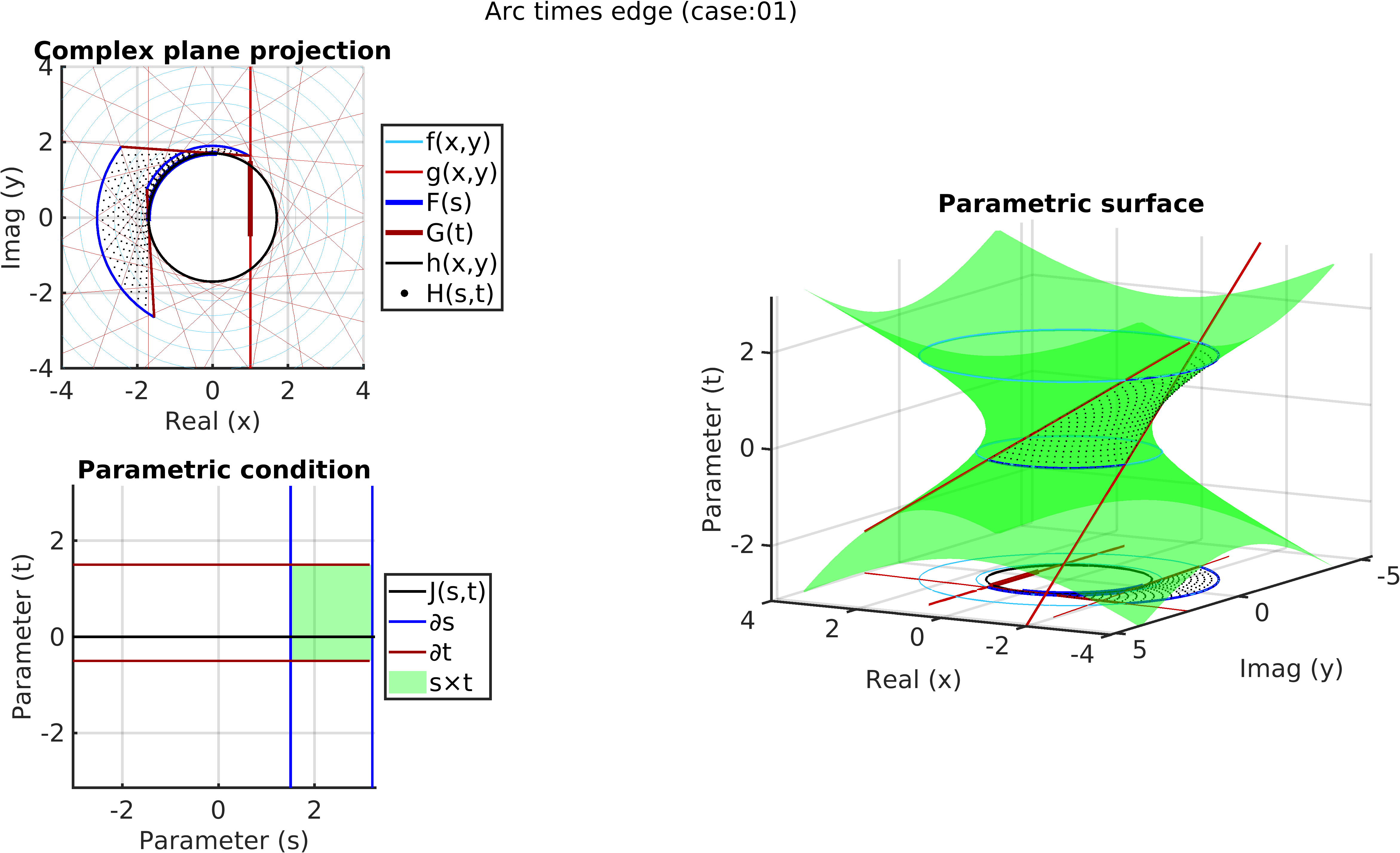}
\caption{Multiplication of an arc and an edge when the edge is not on a zero crossing line, but the arc is on a zero centered circle. Parameters: $r=1.7$, $s=(1.5,3.2)$, $t=(-0.5,1.5)$.}
\label{fig:ArcTimesEdge01}
\end{figure}

\textit{\textbf{Operand equations}}
\begin{flalign*}
f(x,y)&=x^2 +y^2-r^2 \\ \mathring{f} \left(\rho ,\theta \right)&=\rho -r\\
s(x,y)&=\atantwo\left(y,x\right)\\ s^{\circ } \left(\rho ,\theta \right)&=\theta \\
F(s)&={re}^{is} =\left(r\cos (s),r\sin (s)\right)\\
\\ 
g\left(x,y\right)&=x-1
\\
{g}^{\circ } \left(\rho ,\theta \right)&=1/\rho -\cos \left(\theta \right)\\
t\left(x,y\right)&=y\\ t^{\circ } \left(\rho ,\theta \right)&=\rho \sin \left(\theta \right)\\
G(t)&=1+\mathrm{it}=\left(1,t\right)
\end{flalign*}

\textit{\textbf{Parametric combination}}
\begin{flalign*}
H(s,t)
&=\left(r\cos (s)-rt\sin (s),rt\cos (s)+r\sin (s)\right)\\
J(s,t)&=\left\vert \begin{array}{cc}
   -r \sin(s)-rt \cos(s)  & -r \sin(s) \\
   -rt \sin (s)+r\cos(s)  & r \cos(s)
\end{array} \right\vert
\\
&= -r^2t (\sin^2(s)+\cos^2(s))\\
&=-r^2t \\
\text{Envelope: } & \boxed{t=0}
\end{flalign*}

\textit{\textbf{Implicit combination}}
\begin{flalign*}
h(x,y)& =r^2-x^2-y^2
\\
\text{Envelope: } & \boxed{x^2+y^2=r^2}
\end{flalign*}

\textit{\textbf{Mixed combination}}
\begin{flalign*}
u^{\circ } (\rho ,\theta ,t)
&=\theta -\atantwo\left(t,1\right)\\
\\
u(x,y,t)
&=\atantwo\left(y,x\right)-\atantwo\left(t,1\right)=0\\
&\Longrightarrow \atantwo\left(y,x\right)=\atan\left(t\right)
\end{flalign*}

\begin{flalign*}
h(\rho,\theta)
&=\frac{\rho }{\sqrt{t^2 +1}}-r\\
\\
\hat{h}(x,y,t)
&=\frac{\sqrt{x^2 +y^2 }}{\sqrt{t^2 +1}}-r=0\\
&=\frac{x^2 +y^2 }{t^2 +1}=r^2 \\
&\Longrightarrow x^2 +y^2 -r^2 t^2 -r^2 =0\\
\frac{\partial \hat{h}}{\partial t}&={-2r}^2 t=0\\
&\Longrightarrow t=0\\
h(x,y) &=x^2 +y^2 -r^2 =0\\
\text{Envelope: } & \boxed{x^2+y^2=r^2}
\end{flalign*}

\begin{flalign*}
x(s,t)
&=r\cos (s)-rt\sin (s)\\
y(s,t)
&=rt\cos (s)+r\sin (s)\\
J(s,t)
&=h\left(r\cos (s)-rt\sin (s),rt\cos (s)+r\sin (s)\right)\\
&={\left(r\cos (s)-rt\sin (s)\right)}^2 +{\left(rt\cos (s)+r\sin (s)\right)}^2 -r^2 \\
& = r^2 \cos^2 \left(s\right)-2r^2 t\sin (s)\cos (s)+r^2 t^2 \sin^2 \left(s\right)\\
&\quad\quad+r^2 t^2 \cos^2 \left(s\right)+2r^2 t\sin (s)\cos (s)+r^2 \sin^2 \left(s\right)-r^2\\
J(s,t)=0
&\Longrightarrow r^2 t^2 =0\\
\text{Envelope: } & \boxed{t=0}
\end{flalign*}

\vfill \eject

\paragraph{Special case: the edge is on a zero crossing line and the arc is on a zero centered circle} ~\\

\begin{figure}[H]
\centering
\includegraphics[width=\textwidth,trim={0 0 0 10mm},clip]{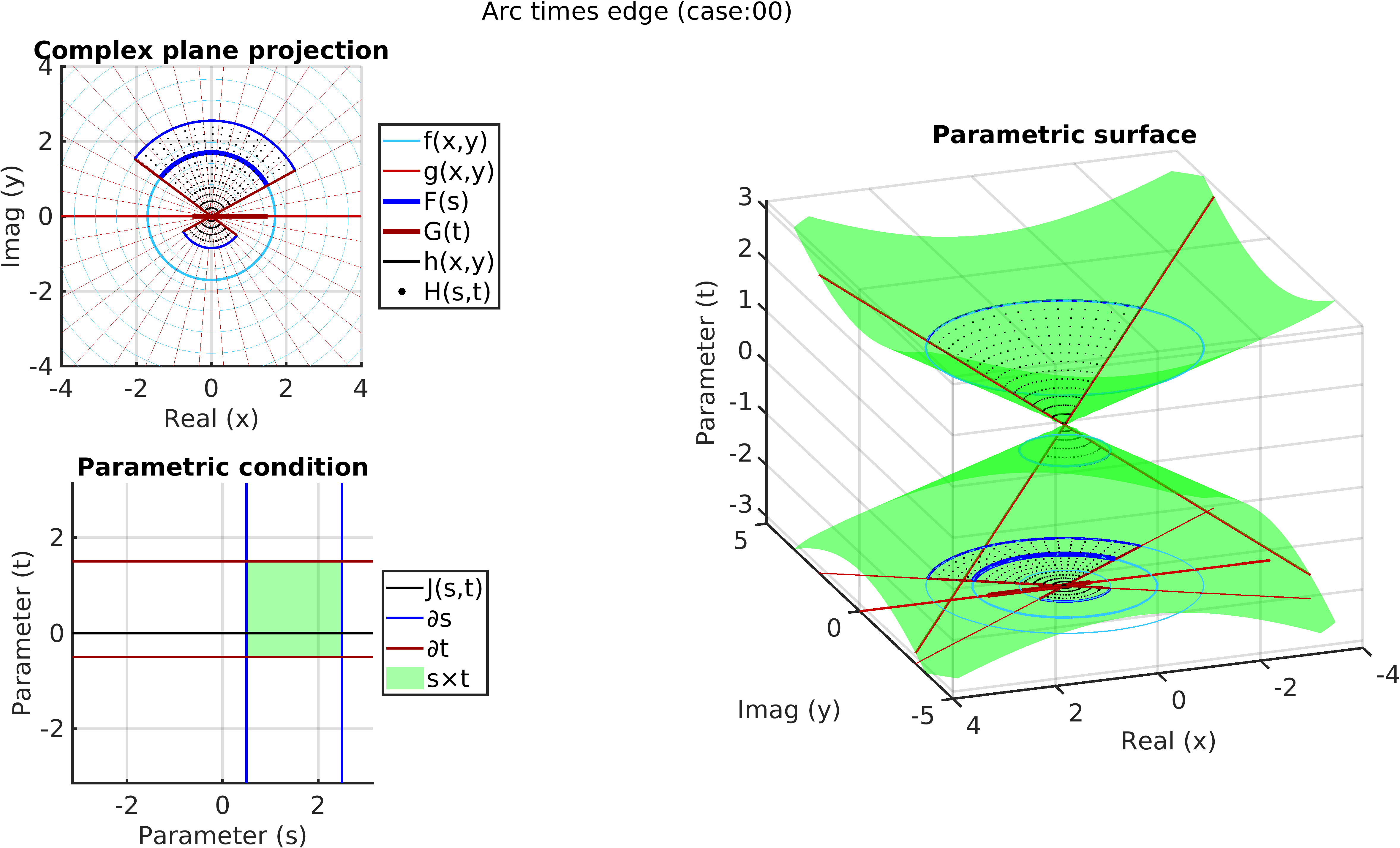}
\caption{Multiplication of an arc and an edge when the edge is on a zero crossing line, and the arc is on a zero centered circle. Parameters: $r=1.7$, $s=(0.5,2.5)$, $t=(-0.5,1.5)$.}
\label{fig:ArcTimesEdge00}
\end{figure}

\textit{\textbf{Operand equations}}
\begin{flalign*}
f(x,y)&=x^2 +y^2-r^2 \\ \mathring{f} \left(\rho ,\theta \right)&=\rho -r\\
s(x,y)&=\atantwo\left(y,x\right)\\ s^{\circ } \left(\rho ,\theta \right)&=\theta \\
F(s)&={re}^{is} =\left(r\cos (s),r\sin (s)\right)\\
\\ 
g\left(x,y\right)&=y,{g}^{\circ } \left(\rho ,\theta \right)=\tan \left(\theta \right)\\
t\left(x,y\right)&=x\\ t^{\circ } \left(\rho ,\theta \right)&=\rho \\
G(t)&=t=\left(t,0\right)
\end{flalign*}

\textit{\textbf{Parametric combination}}
\begin{flalign*}
H(s,t)
&=\left(rt\cos (s),rt\sin (s)\right)\\
J(s,t)&=\left\vert \begin{array}{cc}
   -r t\sin(s)  & r \cos(s) \\
   rt \cos (s)  & r \sin(s)
\end{array} \right\vert
\\
&= -r^2t (\sin^2(s)+\cos^2(s))\\
&=-r^2t \\
\text{Envelope: } & \boxed{t=0}
\end{flalign*}

\textit{\textbf{Implicit combination}}
\begin{flalign*}
h_1(x,y) &=x \\
h_2(x,y) & =y \\
\text{Envelope: } & \boxed{(0,0) \; \text{(i.e. the origin)}}
\end{flalign*}

\textit{\textbf{Mixed combination}}
\begin{flalign*}
u^{\circ } (\rho ,\theta ,t)
&=\theta \\
\\
u(x,y,t)
&=\atantwo\left(y,x\right)\\
u(x,y,t)=0
&\Longrightarrow y=0,x>0
\end{flalign*}

\begin{flalign*}
h(\rho,\theta)
&=\frac{\rho }{t}-r\\
\hat{h}(x,y,t)
&=\frac{\sqrt{x^2 +y^2 }}{t}-r\\
\hat{h}(x,y,t)=0
&\Longrightarrow x^2 +y^2 -r^2 t^2 =0\\
\frac{\partial \hat{h}}{\partial t}&={2r}^2 t\\
\frac{\partial \hat{h}}{\partial t}=0
&\Longrightarrow t=0\\
h(x,y)
&=x^2 +y^2 =0\\
\text{Envelope: } & \boxed{(0,0) \; \text{(i.e. the origin)}}
\end{flalign*}

\begin{flalign*}
x(s,t)
&=rt\cos (s)\\
\\
y(s,t)
&=rt\sin (s)\\
\\
J(s,t)
&=h\left(rt\cos (s),rt\sin (s)\right)\\
&={\left(rt\cos (s)\right)}^2 +{\left(rt\sin (s)\right)}^2 \\
&=r^2 t^2 \cos^2 \left(s\right)+r^2 t^2 \sin^2 \left(s\right)\\
&=r^2 t^2 \\
J(s,t)=0 &\Longrightarrow t=0\\
\text{Envelope: } & \boxed{t=0}
\end{flalign*}

\vfill \eject

\subsubsection{Arc times arc}\label{sup:arc_times_arc}

\paragraph{General case} ~\\

\begin{figure}[H]
\centering
\includegraphics[width=\textwidth,trim={0 0 0 10mm},clip]{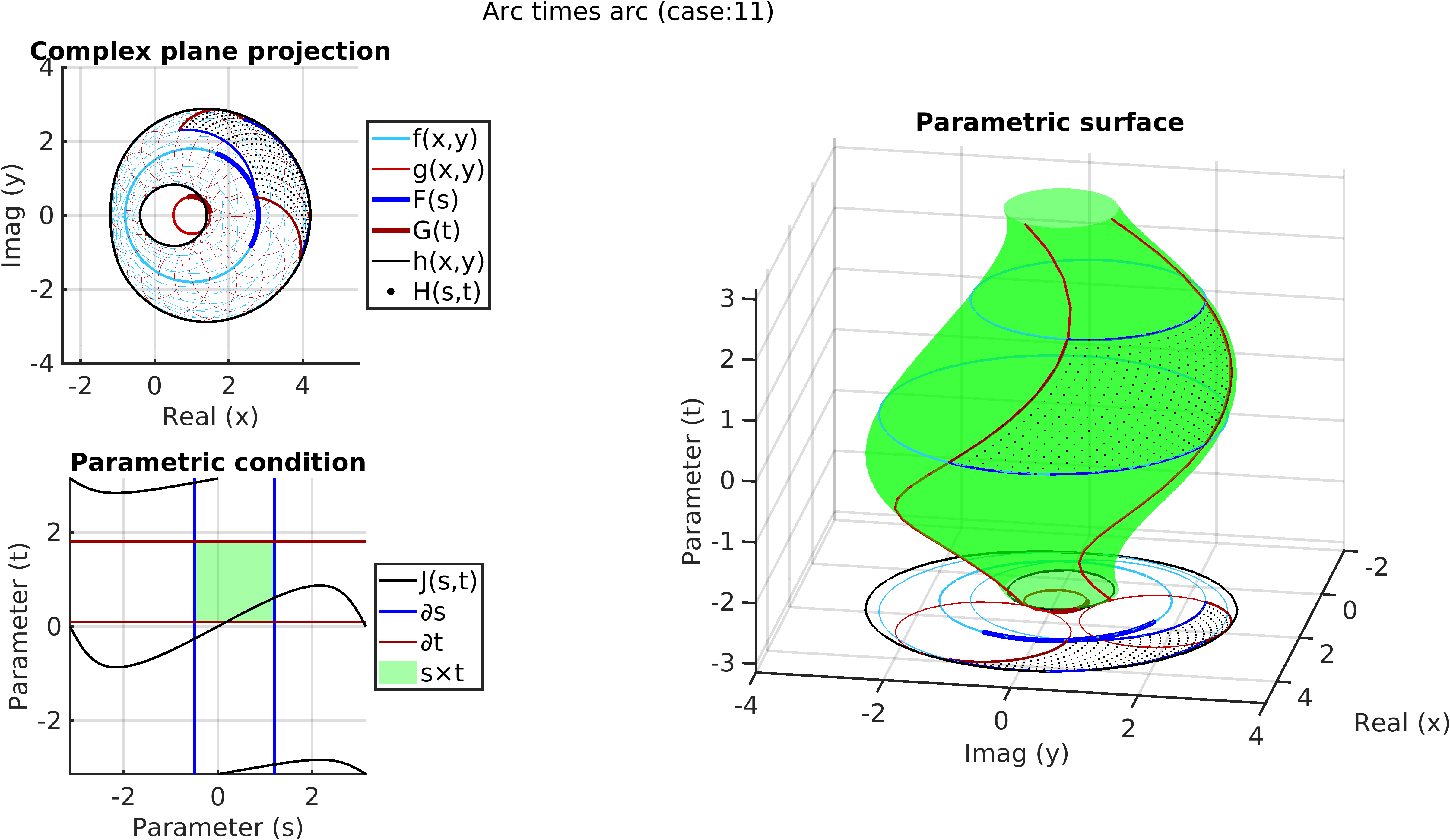}
\caption{Multiplication of two arcs when none of them is on a zero centered circle. Parameters: $r_1=1.8$, $r_2=0.5$, $s=(-0.5,1.2)$, $t=(0.1,1.8)$.}
\label{fig:ArcTimesArc11}
\end{figure}

\textit{\textbf{Operand equations}}
\begin{flalign*}
f(x,y)&={\left(x-1\right)}^2 +y^2-r_1^2 \\ \mathring{f} \left(\rho ,\theta \right)&=\rho^2 -2\rho \cos \left(\theta \right)+1-r_1^2 \\
s(x,y)&=\atantwo\left(y,x-1\right)\\ s^{\circ } \left(\rho ,\theta \right)&=\atantwo\left(\rho \sin \left(\theta \right),\rho \cos \left(\theta \right)-1\right)\\
F(s)&={r_1 e}^{is} +1=\left(1+r_1 \cos (s),r_1 \sin (s)\right)\\
\\ 
g\left(x,y\right)&={\left(x-1\right)}^2 +y^2-r_2^2 
\\
{g}^{\circ } \left(\rho ,\theta \right)&=\rho^2 -2\rho \cos \left(\theta \right)+1-r_2^2 \\
t\left(x,y\right)&=\atantwo\left(y,x-1\right)\\ t^{\circ } \left(\rho ,\theta \right)&=\atantwo\left(\rho \sin \left(\theta \right),\rho \cos \left(\theta \right)-1\right)\\
G(t)&={r_2 e}^{\mathrm{it}} +1=\left(1+r_2 \cos (t),r_2 \sin (t)\right)
\end{flalign*}

\textit{\textbf{Parametric combination}}
\begin{flalign*}
H(s,t)
&=\big(r_1 \cos (s)+r_2 \cos (t)-r_1 r_2 \sin (s)\sin (t)+r_1 r_2 \cos (s)\cos (t)+1,
\\&\quad\;\; r_1 \sin (s)+r_2 \sin (t)+r_1 r_2 \cos (s)\sin (t)+r_1 r_2 \sin (s)\cos (t)\big)\\
&=\big( 1+r_1 \cos(s)+r_2\cos(t)+r_1r_2\cos(s+t), r_1 \sin(s)+r_2\sin(t)+r_1r_2\sin(s+t) \big) \\
\\
J(s,t)
&=\left\vert \begin{array}{cc}
        -r_1\sin (s) - r_1r_2 \sin (s+t) & -r_2\sin t -r_1r_2 \sin (s+t) \\
        r_1\cos (s) + r_1r_2 \cos (s+t) & r_2\cos (t) +r_1r_2 \cos (s+t)
    \end{array} \right\vert
\\
&=r_1r_2\big( \sin(t-s)+r_1\sin((s+t)-s)+r_2\sin(t-(s+t))+r_1r_2\sin((s+t)(s+t)) \big) \\
&=r_1 r_2 \big(\sin (t-s) -r_2 \sin (s)+r_1 \sin (t)\big)\\
\text{Envelope: } & \boxed{\sin (s-t)=r_1 \sin (t)-r_2 \sin (s)}
\end{flalign*}

\textit{\textbf{Implicit combination}}
\begin{flalign*}
\text{Envelope: } & \boxed{
\begin{array}{l}
y^4 + 2x^2y^2 - 4xy^2 - 2r_2^2r_1^2y^2 - 2r_1^2y^2 - 2r_2^2y^2 + 2y^2 
\\
 + x^4 - 4x^3 - 2r_2^2r_1^2x^2 - 2r_1^2x^2 - 2r_2^2x^2 + 6x^2 - 4r_2^2r_1^2x + 4r_1^2x + 4r_2^2x - 4x 
\\
 + r_2^4r_1^4 - 2r_2^2r_1^4 + r_1^4 - 2r_2^4r_1^2 + 4r_2^2r_1^2 - 2r_1^2 + r_2^4 - 2r_2^2 + 1
\end{array}}
\end{flalign*}

\textit{\textbf{Mixed combination}}
\begin{flalign*}
u^{\circ } (\rho ,\theta ,t)
&=\atantwo\left(\frac{\rho }{\sqrt{{\left(1+r_2 \cos (t)\right)}^2 +r_2^2 \sin^2 \left(t\right)}}\sin \left(\theta -\atantwo\left(r_2 \sin (t),1+r_2 \cos (t)\right)\right)
\right. 
\\ & \left. \quad\quad\quad\quad,\frac{\rho }{\sqrt{{\left(1+r_2 \cos (t)\right)}^2 +r_2^2 \sin^2 \left(t\right)}}\cos \left(\theta -\atantwo\left(r_2 \sin (t),1+r_2 \cos (t)\right)\right)-1\right)\\
&=\atantwo\left(\frac{\rho }{\sqrt{r_2^2 +2r_2 \cos (t)+1}}\sin \left(\theta -\atantwo\left(r_2 \sin (t),1+r_2 \cos (t)\right)\right)\right.
\\ & \left. \quad\quad\quad\quad,\frac{\rho }{\sqrt{r_2^2 +2r_2 \cos (t)+1}}\cos \left(\theta -\atantwo\left(r_2 \sin (t),1+r_2 \cos (t)\right)\right)-1\right)\\
\\
u(x,y,t)
&=\atantwo\left(\frac{\sqrt{x^2 +y^2 }}{\sqrt{r_2^2 +2r_2 \cos (t)+1}}\sin \left(\atantwo\left(y,x\right)-\atantwo\left(r_2 \sin (t),1+r_2 \cos (t)\right)\right) \right.
\\ & \left. \quad\quad\quad\quad,\frac{\sqrt{x^2 +y^2 }}{\sqrt{r_2^2 +2r_2 \cos (t)+1}}\cos \left(\atantwo\left(y,x\right)-\atantwo\left(r_2 \sin (t),1+r_2 \cos (t)\right)\right)-1\right)\\
u(x,y,t)=0
&\Longrightarrow \frac{\sqrt{x^2 +y^2 }}{\sqrt{r_2^2 +2r_2 \cos (t)+1}}\cos \left(\atantwo\left(y,x\right)-\atantwo\left(r_2 \sin (t),1+r_2 \cos (t)\right)\right)-1=0\\
&\Longrightarrow \cos \left(\atantwo\left(y,x\right)-\atantwo\left(r_2 \sin (t),1+r_2 \cos (t)\right)\right)=\frac{\sqrt{r_2^2 +2r_2 \cos (t)+1}}{\sqrt{x^2 +y^2 }}\\
&\Longrightarrow \sin \left(\atantwo\left(y,x\right)\right)\sin \left(\atantwo\left(r_2 \sin (t),1+r_2 \cos (t)\right)\right)
\\ & +\cos \left(\atantwo\left(y,x\right)\right)\cos \left(\atantwo\left(r_2 \sin (t),1+r_2 \cos (t)\right)\right)=\frac{\sqrt{r_2^2 +2r_2 \cos (t)+1}}{\sqrt{x^2 +y^2 }}\\
&\Longrightarrow \frac{y}{\sqrt{x^2 +y^2 }} \frac{1+r_2 \cos (t)}{\sqrt{r_2^2 +2r_2 \cos (t)+1}}+\frac{x}{\sqrt{x^2 +y^2 }} \frac{r_2 \sin (t)}{\sqrt{r_2^2 +2r_2 \cos (t)+1}}\\
&\quad\quad=\frac{\sqrt{r_2^2 +2r_2 \cos (t)+1}}{\sqrt{x^2 +y^2 }}\\
&\Longrightarrow y\left(1+r_2 \cos (t)\right)+xr_2 \sin (t)=r_2^2 +2r_2 \cos (t)+1\\
&\Longrightarrow y\left(1+r_2 \cos (t)\right)+xr_2 \sin (t)-r_2^2 -2r_2 \cos (t)-1=0
\end{flalign*}

\begin{flalign*}
h(\rho,\theta)
&=\frac{\rho^2 }{r_2^2 +2r_2 \cos (t)+1}-\frac{2\rho }{\sqrt{r_2^2 +2r_2 \cos (t)+1}}\cos \left(\theta -\atantwo\left(r_2 \sin (t),1+r_2 \cos (t)\right)\right)\\
&\quad\quad+1-r_1^2 =0\\
\hat{h}(x,y,t)
&=\frac{x^2 +y^2 }{r_2^2 +2r_2 \cos (t)+1}-\frac{2\sqrt{x^2 +y^2 }}{\sqrt{r_2^2 +2r_2 \cos (t)+1}}\cos \left(\atantwo\left(y,x\right)-\atantwo\left(r_2 \sin (t),1+r_2 \cos (t)\right)\right)\\
&\quad+1-r_1^2 \\
&=\frac{x^2 +y^2 }{r_2^2 +2r_2 \cos (t)+1}-\frac{2\sqrt{x^2 +y^2 }}{\sqrt{r_2^2 +2r_2 \cos (t)+1}}\left\lbrack 
\frac{y (1+r_2 \cos (t)) + x r_2 \sin(t)}{\sqrt{x^2 +y^2 }\sqrt{r_2^2 +2r_2 \cos (t)+1}}\cdot \right\rbrack +1-r_1^2 \\
&=\frac{x^2 +y^2 }{r_2^2 +2r_2 \cos (t)+1}-2\frac{y\left(1+r_2 \cos (t)\right)+xr_2 \sin (t)}{r_2^2 +2r_2 \cos (t)+1}+1-r_1^2 \\
\hat{h}(x,y,t)=0
&\Longrightarrow x^2 +y^2 -2y\left(1+r_2 \cos (t)\right)+2xr_2 \sin (t)+\left(1-r_1^2 \right)\left(r_2^2 +2r_2 \cos (t)+1\right)=0\\
&\Longrightarrow \left(x-1-r_2 \cos (t){\left.\right)}^2 +\left(y-r_2 \sin (t){\left.\right)}^2 -r_1^2 \left(r_2^2 +2r_2 \cos (t)+1\right)=0\right.\right.\\
\\
\frac{\partial \hat{h}}{\partial t}&=2r_1^2 r_2 \sin (t)-2r_2 \cos (t)\left(y-r_2 \sin (t)\right)-2r_2 \sin (t)\left(r_2 \cos (t)-x+1\right)\\
\frac{\partial \hat{h}}{\partial t}=0
&\Longrightarrow \sin (t)\left(1-r_1^2 -x\right)+\mathrm{ycos}\left(t\right)=0\\
&\Longrightarrow t=\atan\left(\frac{r_1^2 +x-1}{y}\right)\\
h(x,y)
&=\left(x-1-r_2 {\cos \left(\atan\left(\frac{r_1^2 +x-1}{y}\right)\right)}\right)^2 
+ \left(y-r_2 {\sin \left(\atan\left(\frac{r_1^2 +x-1}{y}\right)\right)}\right)^2 
\\&-r_1^2 \left(r_2^2 +2r_2 \cos \left(\atan\left(\frac{r_1^2 +x-1}{y}\right)\right)+1\right)\\
&={\left(x-1-r_2 \frac{y}{\sqrt{y^2 +{\left(r_1^2 +x+1\right)}^2 }}\right)}^2 +{\left(y-r_2 \frac{r_1^2 +x-1}{\sqrt{y^2 +{\left(r_1^2 +x+1\right)}^2 }}\right)}^2 
\\&-r_1^2 \left(r_2^2 +2r_2 \frac{y}{\sqrt{y^2 +{\left(r_1^2 +x+1\right)}^2 }}+1\right) \\
h(x,y)= 0 
&\Longrightarrow {{\left(y-\frac{r_2 \,y}{\sqrt{{{\left({r_1 }^2 +x-1\right)}}^2 +y^2 }}\right)}}^2 -{r_1 }^2 \,r_2^2 +{{\left(x-\frac{r_2 \,{\left({r_1 }^2 +x-1\right)}}{\sqrt{{{\left({r_1 }^2 +x-1\right)}}^2 +y^2 }}\right)}}^2 =0\\
&\Longrightarrow
\text{... simplified with Matlab Symbolic Math toolbox}\\
\text{Envelope: } & \boxed{{{\left({{\left(x-1\right)}}^2 +y^2 -{r_1 }^2 \,{\left(r_2^2 +1\right)}+r_2^2 \right)}}^2 -4\,r_2^2 \,{\left({{\left(x-1+{r_1 }^2 \right)}}^2 +y^2 \right)}=0}
\end{flalign*}

\begin{flalign*}
x(s,t)
&=\left(1+r_1 \cos (s)\right)\left(1+r_2 \cos (t)\right)-r_1 r_2 \sin (s)\sin (t)\\
y(s,t)
&=\left(1+r_1 \cos (s)\right)r_2 \sin (t)+\left(1+r_2 \cos (t)\right)r_1 \sin (s)\\
J(s,t)
&=h\left(\left(1+r_1 \cos (s)\right)\left(1+r_2 \cos (t)\right)-r_1 r_2 \sin (s)\sin (t)
\right. \\ & \left. \quad\quad,\left(1+r_1 \cos (s)\right)r_2 \sin (t)+\left(1+r_2 \cos (t)\right)r_1 \sin (s)\right)\\
&=\left({\left(r_1 \,\sin (s)\,\left(r_2 \,\cos (t)+1\right)+r_2 \,\sin (t)\,\left(r_1 \,\cos (s)+1\right)\right)}^2 
\right. \\ & \left. +{\left(r_1 \,r_2 \,\sin (s)\,\sin (t)-\left(r_1 \,\cos (s)+1\right)\,\left(r_2 \,\cos (t)+1\right)+1\right)}^2 -{r_1 }^2 \,\left(r_2^2 +1\right)+r_2^2 \right)^2 \\
&-4\,r_2^2 \,\left({\left(\left(r_1 \,\cos (s)+1\right)\,\left(r_2 \,\cos (t)+1\right)+{r_1 }^2 -r_1 \,r_2 \,\sin (s)\,\sin (t)-1\right)}^2 
\right. \\ & \left. \quad\quad\quad +{\left(r_1 \,\sin (s)\,\left(r_2 \,\cos (t)+1\right)+r_2 \,\sin (t)\,\left(r_1 \,\cos (s)+1\right)\right)}^2 \right)\\
\text{Envelope: } & \boxed{\begin{array}{l}
4r_2^2 \left\lbrack {\left(r_1 r_2 \cos (s)+r_1^2 \cos (t)+r_1 \cos \left(s-t\right)+r_2 \right)}^2 
\right. \\ 
\left. \quad\quad\quad- {\left(r_1 \sin (s)+r_2 \sin (t)+r_1 r_2 \sin \left(s+t\right)\right)}^2 
\right. \\ 
\left. \quad\quad\quad-{\left(r_1 \cos (s)+r_2 \cos (t)+r_1 r_2 \cos \left(s+t\right)+r_1^2 \right)}^2 \right\rbrack 
\end{array}}
\end{flalign*}

\vfill \eject

\paragraph{Special case: one arc is on a zero centered circle} ~\\

\begin{figure}[H]
\centering
\includegraphics[width=\textwidth,trim={0 0 0 10mm},clip]{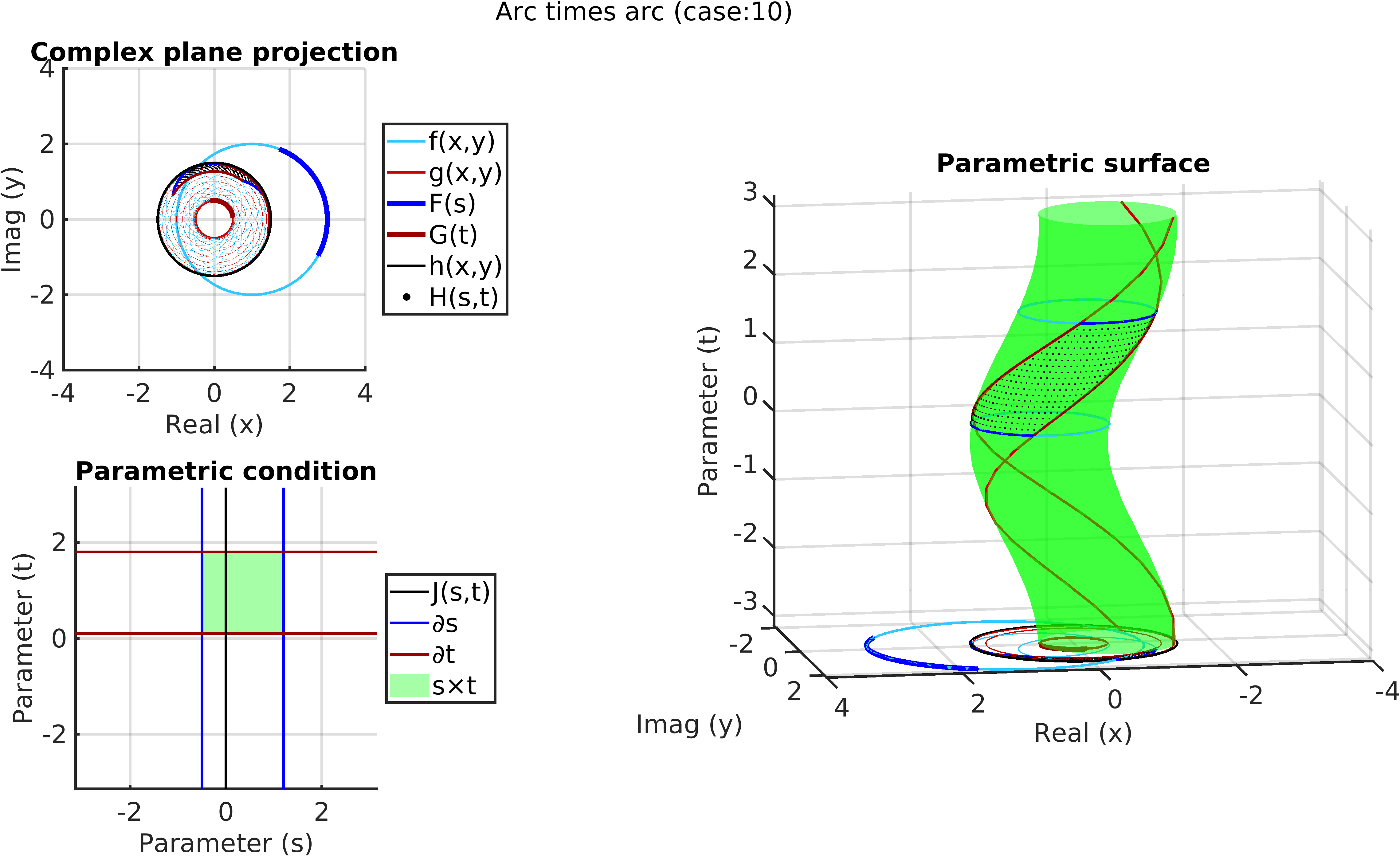}
\caption{Multiplication of two arcs when one of them is on a zero centered circle. Parameters: $r_1=2.0$, $r_2=0.5$, $s=(-0.5,1.2)$, $t=(0.1,1.8)$.}
\label{fig:ArcTimesArc10}
\end{figure}

\textit{\textbf{Operand equations}}
\begin{flalign*}
f(x,y)&={\left(x-1\right)}^2 +y^2-r_1^2 \\ \mathring{f} \left(\rho ,\theta \right)&=\rho^2 -2\rho \cos \left(\theta \right)+1-r_1^2 \\
s(x,y)&=\atantwo\left(y,x-1\right)\\ s^{\circ } \left(\rho ,\theta \right)&=\atantwo\left(\rho \sin \left(\theta \right),\rho \cos \left(\theta \right)-1\right)\\
F(s)&={r_1 e}^{is} +1=\left(1+r_1 \cos (s),r_1 \sin (s)\right)\\
\\ 
g\left(x,y\right)&=x^2 +y^2-r_2^2 
\\
{g}^{\circ } \left(\rho ,\theta \right)&=\rho -r_2 \\
t\left(x,y\right)&=\atantwo\left(y,x\right)\\ t^{\circ } \left(\rho ,\theta \right)&=\theta \\
G(t)&={r_2 e}^{\mathrm{it}} =\left(r_2 \cos (t),r_2 \sin (t)\right)
\end{flalign*}

\textit{\textbf{Parametric combination}}
\begin{flalign*}
H(s,t)
&=\left(\left(r_1 \cos (s)+1\right)r_2 \cos (t)-r_1 r_2 \sin (s)\sin (t),\left(r_1 \cos (s)+1\right)r_2 \sin (t)+r_1 r_2 \sin (s)\cos (t)\right)\\
&=r_2\big( \cos(t)+r_1 \cos(s+t), \sin(t) + r_1 \sin(s+t) \big)
\\
J(s,t)
&=r_2^2 \left\vert \begin{array}{cc}
        -\sin(t) -r_1 \sin (s+t) & -r_1\sin (s+t) \\
        \cos(t) +r_1\cos (s+t) & r_1\cos (s+t)
    \end{array} \right\vert\\
&=r_1r_2^2(\cos(t) \sin(s+t)-\sin(t) \cos(s+t)) \\
&=r_1r_2^2 \sin((s+t)-t) \\
&=r_1 r_2^2 \sin(s)\\
\text{Envelope: } & \boxed{s=k \pi \; (k \in \mathbb{Z})} 
\end{flalign*}

\textit{\textbf{Implicit combination}}
\begin{flalign*}
h(x,y) & =x^4+ y^4 + 2x^2y^2 - 2r_2^2x^2 - 2r_2^2y^2 - 2r_2^2r_1^2x^2 - 2r_2^2r_1^2y^2  - 2r_2^4r_1^2 + r_2^4r_1^4+ r_2^4 \\
&= (x^2+y^2-r_2^2+2r_1r_2^2-r_1^2r_2^2)
(x^2+y^2-r_2^2-2r_1r_2^2-r_1^2r_2^2) \\
\text{Envelope: }& \boxed{\left\{x^2+y^2=(r_2(1-r_1)^2)\right\} \cup \left\{x^2+y^2=(r_2(1+r_1)^2)\right\}}
\end{flalign*}

\textit{\textbf{Mixed combination}}
\begin{flalign*}
u^{\circ } (\rho ,\theta ,t)
&=\atantwo\left(\frac{\rho }{r_2 }\sin \left(\theta -t\right),\frac{\rho }{r_2 }\cos \left(\theta -t\right)-1\right)
\\
u^{\circ } (\rho ,\theta ,t)=0
&\Longrightarrow \frac{\rho }{r_2 }\cos \left(\theta -t\right)-1=0\\
&\Longrightarrow \frac{\rho }{r_2 }\left\lbrack \cos \left(\theta \right)\cos (t)+\sin \left(\theta \right)\sin (t)\right\rbrack - 1 \\
u(x,y,t)
&=\frac{\sqrt{x^2 +y^2 }}{r_2 }\left\lbrack \cos \left(\atantwo\left(y,x\right)\right)\cos (t)+\sin \left(\atantwo\left(y,x\right)\right)\sin (t)\right\rbrack - 1\\
u(x,y,t) = 0 
&\Longrightarrow \sqrt{x^2 +y^2 }\left(\frac{x\cos (t)}{\sqrt{x^2 +y^2 }}+\frac{y\sin (t)}{\sqrt{x^2 +y^2 }}\right)=r_2 \\
&\Longrightarrow x\cos (t)+y\sin (t)-r_2 =0
\end{flalign*}

\begin{flalign*}
h(\rho,\theta)
&=\frac{\rho^2 }{r_2^2 }-2\frac{\rho }{r_2 }\cos \left(\theta -t\right)+1-r_1^2 \\
h(\rho,\theta) = 0
&\Longrightarrow \rho^2 -2r_2 \rho \cos \left(\theta -t\right)+\left(1-r_1^2 \right)r_2 =0\\
&\Longrightarrow \rho^2 -2r_2 \rho \left\lbrack \cos \left(\theta \right)\cos (t)+\sin \left(\theta \right)\sin (t)\right\rbrack +\left(1-r_1^2 \right)r_2^2 =0\\
\hat{h}(x,y,t)
&=x^2 +y^2 -2r_2 \sqrt{x^2 +y^2 }\left\lbrack \cos \left(\atantwo\left(y,x\right)\right)\cos (t)+\sin \left(\atantwo\left(y,x\right)\right)\sin (t)\right\rbrack +\left(1-r_1^2 \right)r_2^2 \\
&=x^2 +y^2 -2r_2 \sqrt{x^2 +y^2 }\left(\frac{x\cos (t)}{\sqrt{x^2 +y^2 }}+\frac{y\sin (t)}{\sqrt{x^2 +y^2 }}\right)+\left(1-r_1^2 \right)r_2^2 \\
&=x^2 +y^2 -2r_2 x\cos (t)-2r_2 y\sin (t)+\left(1-r_1^2 \right)r_2^2 \\
\hat{h}(x,y,t)=0
&\Longrightarrow {\left(x-r_2 \cos (t)\right)}^2 +{\left(y-r_2 \sin (t)\right)}^2 -r_1^2 r_2^2 =0\\
\\
\frac{\partial \hat{h}}{\partial t}&=2\,r_2 \,\sin (t)\,{\left(x-r_2 \,\cos (t)\right)}-2\,r_2 \,\cos (t)\,{\left(y-r_2 \,\sin (t)\right)}\\
\frac{\partial \hat{h}}{\partial t} = 0
&\Longrightarrow 2r_2 x\sin (t)=2r_2 y\cos (t)\\
&\Longrightarrow t=\atan\left(\frac{y}{x}\right)\\
h(x,y)
&={\left(x-r_2 \cos \left(\atan\left(\frac{y}{x}\right)\right)\right)}^2 +{\left(y-r_2 \sin \left(\atan\left(\frac{y}{x}\right)\right)\right)}^2 -r_1^2 r_2^2 =0\\
&={{\left(x-\frac{r_2 \,x}{\sqrt{x^2 +y^2 }}\right)}}^2 +{{\left(y-\frac{r_2 \,y}{\sqrt{x^2 +y^2 }}\right)}}^2
-{r_1 }^2 \,r_2^2\\
h(x,y)=0
&\Longrightarrow x - \frac{r_2 x}{\sqrt{x^2+y^2}} +y - \frac{r_2 y}{\sqrt{x^2+y^2}} \pm r_1 r_2 = 0\\
&\Longrightarrow x - \frac{r_2 x}{\sqrt{x^2+y^2}} +y - \frac{r_2 y}{\sqrt{x^2+y^2}} \pm r_1 r_2 = 0\\
&\Longrightarrow
\text{... simplified with Matlab Symbolic Math toolbox}\\
\text{Envelope: } & \boxed{
x^2 +y^2 -(1\pm r_1 )^2 r_2^2}
\end{flalign*}

\begin{flalign*}
x(s,t)
&=\left(r_1 \cos (s)+1\right)r_2 \cos (t)-r_1 r_2 \sin (s)\sin (t)\\
y(s,t)
&=\left(r_1 \cos (s)+1\right)r_2 \sin (t)+r_1 r_2 \sin (s)\cos (t)\\
J(s,t)
&=h\left(\left(r_1 \cos (s)+1\right)r_2 \cos (t)-r_1 r_2 \sin (s)\sin (t),\left(r_1 \cos (s)+1\right)r_2 \sin (t)+r_1 r_2 \sin (s)\cos (t)\right)\\
&={{\left(r_2 \,\sin (t)\,{\left(r_1 \,\cos (s)+1\right)}+r_1 \,r_2 \,\cos (t)\,\sin (s)\right)}}^2 -{{\left(r_2 +r_1 \,r_2 \right)}}^2 
\\ & + {{\left(r_2 \,\cos (t)\,{\left(r_1 \,\cos (s)+1\right)}-r_1 \,r_2 \,\sin (s)\,\sin (t)\right)}}^2 \\
J(s,t)=0
&\Longrightarrow 2\,r_1 \,r_2^2 \,{\left(\cos (s)-1\right)}=0\\
\text{Envelope: } & \boxed{s=k \pi \; (k \in \mathbb{Z})} 
\end{flalign*}

\vfill \eject

\paragraph{Special case: both arcs are on zero centered circles} ~\\

\begin{figure}[H]
\centering
\includegraphics[width=\textwidth,trim={0 0 0 10mm},clip]{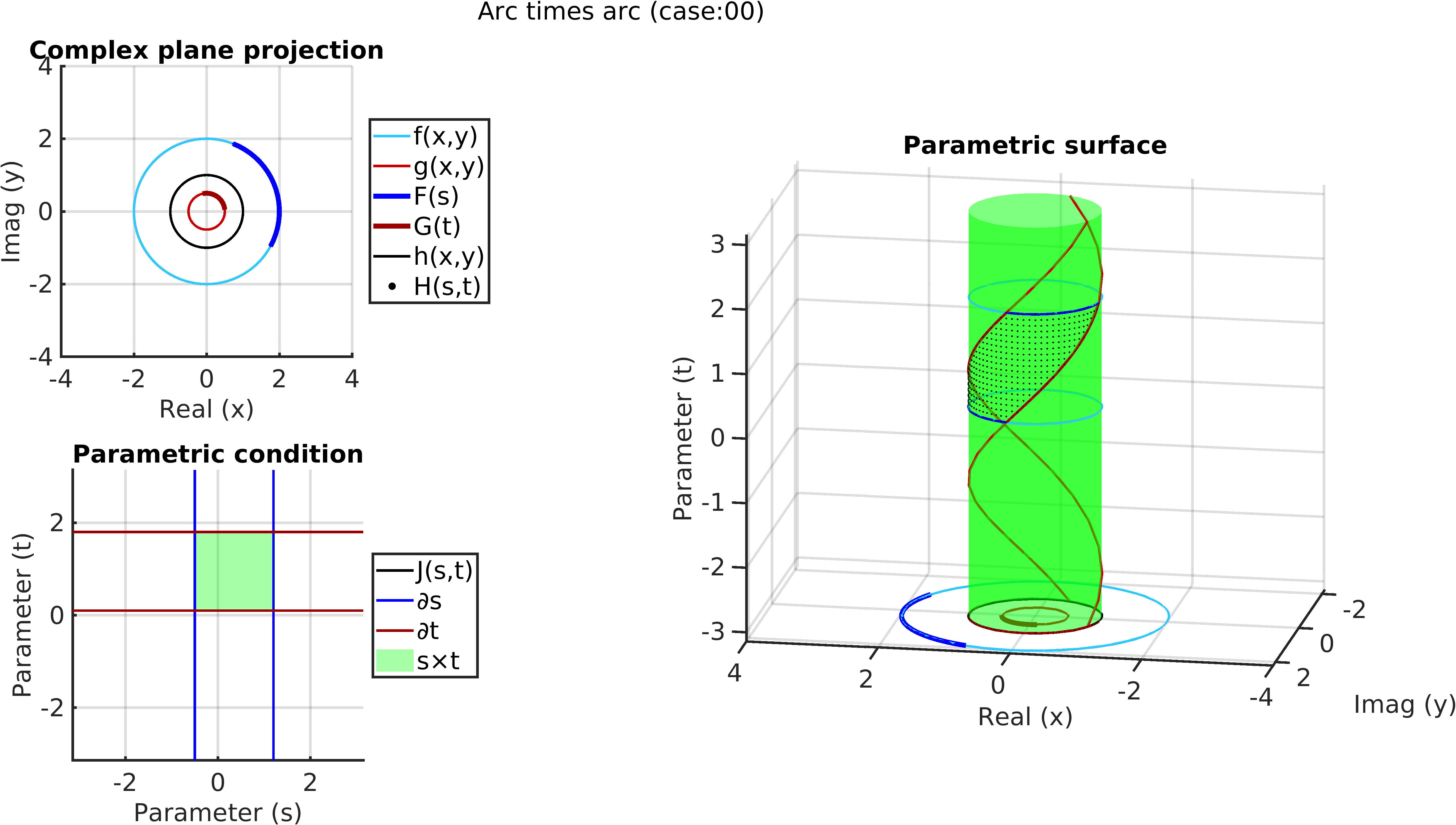}
\caption{Multiplication of two arcs when both are on zero centered circles. Parameters: $r_1=2.0$, $r_2=0.5$, $s=(-0.5,1.2)$, $t=(0.1,1.8)$.}
\label{fig:ArcTimesArc00}
\end{figure}

\textit{\textbf{Operand equations}}
\begin{flalign*}
f(x,y)&=x^2 +y^2-r_1^2 \\ \mathring{f} \left(\rho ,\theta \right)&=\rho -r_1 \\
s(x,y)&=\atantwo\left(y,x\right)\\ s^{\circ } \left(\rho ,\theta \right)&=\theta \\
F(s)&={r_1 e}^{is} =\left(r_1 \cos (s),r_1 \sin (s)\right)\\
\\ 
g\left(x,y\right)&=x^2 +y^2-r_2^2 
\\
{g}^{\circ } \left(\rho ,\theta \right)&=\rho -r_2 \\
t\left(x,y\right)&=\atantwo\left(y,x\right)\\ t^{\circ } \left(\rho ,\theta \right)&=\theta \\
G(t)&={r_2 e}^{\mathrm{it}} =\left(r_2 \cos (t),r_2 \sin (t)\right)
\end{flalign*}

\textit{\textbf{Parametric combination}}
\begin{flalign*}
H(s,t)
&=\left(r_1 r_2 \cos (s)\cos (t)-r_1 r_2 \sin (s)\sin (t),r_1 r_2 \cos (s)\sin (t)+r_1 r_2 \sin (s)\cos (t)\right)\\
&=r_1 r_2 (cos(s+t),sin(s+t))
\\
J(s,t)
&=r_1^2r_2^2 
\left\vert \begin{array}{cc}
        -\sin (s+t) & -\sin (s+t) \\
        \cos (s+t) & \cos (s+t)
    \end{array} \right\vert \\
&=0 \\
\text{Envelope: } & \boxed{\text{all } (s,t)}
\end{flalign*}

\textit{\textbf{Implicit combination}}
\begin{flalign*}
h(x,y) & =x^2+y^2-r_1^2r_2^2 \\
\text{Envelope: } & \boxed{x^2+y^2=(r_1r_2)^2}
\end{flalign*}

\textit{\textbf{Mixed combination}}
\begin{flalign*}
u^{\circ } (\rho ,\theta ,t)
&=\theta -t\\
u(x,y,t)
&=\atantwo\left(y,x\right)-t\\
\\
h(\rho,\theta)
&=\frac{\rho }{r_2 }-r_1 \\
h(\rho,\theta) = 0 &\Longrightarrow
\rho - r_1 r_2 = 0\\
\hat{h}(x,y,t)
&=\sqrt{x^2 +y^2 } - r_1 r_2 \\
\hat{h}(x,y,t) = 0
&\Longrightarrow x^2 +y^2 -r_1^2 r_2^2 =0\\
\frac{\partial \hat{h}}{\partial t}&=0\\
h(x,y)
&=x^2 +y^2 -r_1^2 r_2^2 \\
\text{Envelope: } & \boxed{x^2+y^2=(r_1r_2)^2}\\
\\
x(s,t)
&=r_1 r_2 \cos (s)\cos (t)-r_1 r_2 \sin (s)\sin (t)\\
y(s,t)
&=r_1 r_2 \cos (s)\sin (t)+r_1 r_2 \sin (s)\cos (t)\\
J(s,t)
&=h\left(r_1 r_2 \cos (s)\cos (t)-r_1 r_2 \sin (s)\sin (t),r_1 r_2 \cos (s)\sin (t)+r_1 r_2 \sin (s)\cos (t)\right)\\
&={{\left(r_1 \,r_2 \,\cos (s)\,\sin (t)+r_1 \,r_2 \,\cos (t)\,\sin (s)\right)}}^2 +{{\left(r_1 \,r_2 \,\sin (s)\,\sin (t)-r_1 \,r_2 \,\cos (s)\,\cos (t)\right)}}^2 -{r_1 }^2 \,r_2^2 \\
&=0\\
\text{Envelope: } & \boxed{\text{all } (s,t)}
\end{flalign*}

\vfill \eject